\documentclass{article}
\usepackage[utf8]{inputenc}
\usepackage{amsmath,amssymb,amsthm, xcolor, enumitem, comment, todonotes}
\usepackage{url,hyperref}
\usepackage{verbatim}

\theoremstyle{definition}
\newtheorem{definition}{Definition}

\newtheorem{claim}[definition]{Claim}

\theoremstyle{plain}
\newtheorem{theorem}[definition]{Theorem}
\newtheorem{proposition}[definition]{Proposition}
\newtheorem{lemma}[definition]{Lemma}
\newtheorem{corollary}[definition]{Corollary}

\newtheorem{remark}[definition]{Remark}
\newtheorem{convention}[definition]{Convention}

\newcommand{\fwc}{\mathcal{F}_{WC}}
\newcommand{\T}{\mathcal{T}}
\newcommand{\EF}{\mbox{EF}}
\newcommand{\MA}{\mbox{MA}}

\newcommand{\height}{\mbox{ht}}
\def\rest{\restriction}

\newcommand{\name}{\dot}

\newcommand{\la}{\langle}

\newcommand{\ra}{\rangle}

\newcommand{\elem}{\prec}
\newcommand{\uhr}{\restriction}

\newcommand{\ol}{\ol}
\newcommand{\po}{\mathbb{P}}

\newcommand{\qo}{\mathbb{Q}}

\newcommand{\cC}{\mathcal{C}}

\DeclareMathOperator{\dom}{dom}

\DeclareMathOperator{\supp}{supp}

\DeclareMathOperator{\cof}{Cof}
\DeclareMathOperator{\col}{Coll}

\DeclareMathOperator{\rng}{rng}

\DeclareMathOperator{\Ord}{Ord}

\DeclareMathOperator{\Lev}{Lev}

\newcommand{\coll}{\col(\omega_1,<\kappa)}

\title{Aronszajn trees and maximality\thanks{This project has received funding from the European Research Council (ERC) under the
European Union’s Horizon 2020 research and innovation programme (grant agreement No
101020762), as well as from the Research Center at the Einstein Institute of Mathematics.
The first and third authors would like to thank the Israel Science Foundation (Grant  1302/23) for their support.
The second author was supported by Magnus Ehrnrooth Foundation.
The fourth author was supported by the Academy of Finland grants number 322795 and 368671.}}
\author{Omer Ben-Neria
\\ 
Hebrew University\\ Jerusalem, Israel
\and
Siiri Kivim\"aki
\\ Universit\'e Paris Cit\'e,
\\
University of Helsinki 
\and 
Menachem Magidor
\\ Hebrew University\\ Jerusalem, Israel
\and
Jouko V\"a\"an\"anen
\\ University of Helsinki \\ Helsinki, Finland}
\date{October 2025}

\begin{document}

\maketitle

\begin{abstract}
Assuming the consistency of a weakly compact cardinal above a regular uncountable cardinal $\mu$, we prove the consistency of the existence of a wide $\mu^+$-Aronszajn tree, i.e. a tree of height and cardinality $\mu^+$ with no branches of length $\mu^+$, into which every wide $\mu^+$-Aronszajn tree can be embedded.
\end{abstract}

\smallskip


\smallskip

\section{Introduction}\label{Section:Introduction}

The topic of this paper is maximality among wide $\kappa$-Aronszajn trees, i.e. trees of cardinality and height $\kappa$ without branches of length $\kappa$. Such a tree is called maximal if every such tree can be embedded into it. We show the consistency of maximal trees relative to the consistency of a weakly compact cardinal. This has been an open problem for 30 years.

Trees in this paper are partial orders in which the set of predecessors of every element are well-ordered by the partial order, and there is a unique smallest element. The order-type of the set of predecessors of an element of a tree is called the \emph{height} of the element, and the supremum of all heights in a tree is called the \emph{height} of the tree. The set of elements of a fixed height is called a \emph{level} of the tree. There is a natural quasi-ordering of the class of all trees: a tree $T$ is below a tree $T'$ if $T$ can be monomorphically embedded into $T'$. For any class $\cC$ of trees it is natural to ask if $\cC$ has a maximal element $T$ under embeddability. Then, up to isomorphism, the class $\cC$ consists just of subtrees of $T$ in $\cC$.

Our focus is on trees of cardinality and height $\kappa\ge\omega_1$ with no branches, i.e. linearly ordered subsets, of size $\kappa$. Following \cite{Sh:1186}, we call such trees \emph{wide $\kappa$-Aronszajn trees}. Our main result (Theorem~1 below) is that it is consistent, relative to the consistency of a weakly compact cardinal, to have a maximal wide $\kappa$-Aronszajn tree. Our proof
works for any successor of a regular cardinal $>\aleph_0$. This result complements the fact that it is a consequence of the Generalized Continuum Hypothesis that there are no maximal wide $\kappa^+$-Aronszajn trees  for any infinite regular $\kappa$ \cite{MR1071304}.
Under the stronger assumption  $V=L$, no wide $\kappa^+$-Aronszajn tree is maximal even just for $\kappa^+$-Souslin trees, for 
%we prove in \cite{part1}
it can be proved, improving a result in \cite{MR1712000}, that, assuming $V=L$, for every $\kappa^+$-Aronszajn tree $T$, $\kappa$ regular, there is a $\kappa^+$-Souslin tree which is not  embeddable into $T$. 

If the levels of a wide $\kappa$-Aronszajn tree are of cardinality $<\kappa$, we drop ``wide" and call such trees just $\kappa$-Aronszajn, or, to be more specific, \emph{narrow} $\kappa$-Aronszajn. Furthermore, if $\kappa=\omega_1$, we call the trees Aronszajn, or wide Aronszajn, respectively. For some $\kappa$ there may be no $\kappa$-Aronszajn trees, and then $\kappa$ is said to have the \emph{tree property}. By K\H{o}nig's Lemma, $\omega$ has the tree property. No singular cardinal has the tree property for trivial reasons. An inaccessible cardinal has the tree property if and only if it is weakly compact. 

Examples of Aronszajn trees are so-called Souslin trees, which are instrumental in understanding, and proving the independence of, the so-called Souslin Hypothesis i.e. the hypothesis that the order-type of the real numbers is the unique, up to isomorphism, dense complete linear order without end-points in which all families of disjoint non-empty open sets are countable. 

Wide $\kappa$-Aronszajn trees are important  in the study of model theoretic properties of uncountable structures, namely, trees can be used as a weak substitute for ordinals when  uncountable models are investigated by means of transfinite games \cite{MR1071304} and, more generally, in the study of generalized Baire spaces \cite{MR1242054}. For example, the existence of a particular kind of maximal tree, a so-called Canary Tree (see below), is equivalent, assuming CH, to the isomorphism class of the free Abelian group of cardinality $\aleph_1$ being $\Delta^1_1$ in the generalized Baire Space $\omega_1^{\omega_1}$. This emphasises the importance of understanding better the global ordering of trees, especially the existence of maximal trees. 

There is no maximal countable ordinal, but if we identify ordinals with trees without infinite branches, and consider generalized ordinals i.e. the class $\T_\alpha$ of trees of cardinality $\aleph_\alpha$ without branches of length $\aleph_\alpha$, the situation is more opaque. When $\alpha>0$, the structure of the class of such trees is  much more complicated than the structure of ordinals. For example, the structure of $\T_1$ is highly non-linear as it is easy to construct pairs of wide (or narrow) Aronszajn trees so that neither can be mapped even by a strict order preserving homomorphism to the  other. Furthermore, the structure is highly non-absolute. Several partial results are known about $\T_\alpha$, $\alpha>0$, \cite{MR1242054,MR1222536,MR1712000,MR2329763,Sh:1186}.
\medskip

The main result of this paper is:

\begin{theorem}\label{Theorem:main}Suppose that $\kappa$ is a weakly compact cardinal and $\mu < \kappa$ is regular uncountable. There is a (set) forcing extension of the universe in which $\kappa = \mu^+$ and there is a maximal wide $\kappa$-Aronszajn tree i.e. a wide $\kappa$-Aronszajn tree $T$ such that any other wide $\kappa$-Aronszajn tree can be monomorphically embedded into $T$. \end{theorem}

To simplify our presentation, we will prove the theorem for the case $\mu = \omega_1$  (i.e., $\kappa = \omega_2$). It will be apparent throughout this work that modification to an arbitrary regular uncountable cardinal $\mu$ is straightforward. The theorem holds also for $\mu = \omega$,  see \cite{siirialef1}.\\

\noindent 
This theorem further emphasises the difference between the order of trees with no infinite branches and the class of trees  with no branches of length $\kappa>\omega$. 
%In the Baire space $\omega^\omega$ the set of reals that code an element of $\T_0$ form a complete $\boldsymbol{\Pi}^1_1$-set without a largest element. In the generalized Baire space $\omega_1^{\omega_1}$ the set of functions that code an element of $\T_1$ is again a complete $\boldsymbol{\Pi}^1_1$-set without a largest element, if CH is assumed \cite{MR1242054}, but in sharp contrast to this, a $\boldsymbol{\Delta}_1^1$-set with a largest element in the model obtained from Theorem~\ref{Theorem:main}. 

We shall now define in detail the central concepts of this paper. We have already defined the concept of a wide $\kappa$-Aronszajn tree as well as its special case, the (narrow) $\kappa$-Aronszajn tree, agreeing to drop $\kappa$ if $\kappa=\omega_1$.  While Aronszajn trees always exist, the existence of an $\aleph_2$-Aronszajn tree is independent of ZFC, assuming the consistency of weakly compact cardinals: Specker proved  the existence of an $\aleph_2$-Aronszajn tree from CH \cite{MR39779}. Mitchell and Silver proved the consistency of the non-existence of $\aleph_2$-Aronszajn trees, relative to the consistency of a weakly compact cardinal \cite[Theorem 5.8]{MR313057}. They also showed that if there are no $\aleph_2$-Aronszajn trees, then $\aleph_2$ is weakly compact in $L$. As opposed to the case of (narrow) $\kappa$-Aronszajn trees, it is easy to construct a \emph{wide} $\kappa$-Aronszajn tree in ZFC by bundling together isolated branches of all lengths $<\kappa$. 
 
As discussed already, our main topic in this paper is the existence of trees that are maximal in some specific sense.  There are several ways in which two trees $T$ and $T'$ can be compared to each other in order for the concept of maximality to make sense. Originally the question of maximality was raised \cite{MR1242054} in connection with comparing trees by asking whether there is  a \emph{homomorphism} from one to the other i.e. a mapping from one tree $T$ to another $T'$, that preserves strict ordering:
$$t<_Tt'\Rightarrow f(t)<_{T'}f(t').$$ Note that such a mapping need not be one-one because incomparable elements can be mapped to the same element. We follow \cite{Sh:1186} in calling  such a mapping a \emph{weak embedding}\footnote{There was a  claim in \cite{MR1242054} that a maximal tree exists in the sense of weak embeddings in $\T_{1}$ if Martin's Axiom (MA) and $2^\omega>\omega_1$ were assumed. This claim was proved wrong in \cite{Sh:1186}.}.  The reason for the emergence of weak embeddings as a way to order classes of trees is its close connection to certain games, introduced below.   While proving Theorem~\ref{Theorem:main} the authors realized that they can actually prove the consistency of the existence of a maximal tree under a stronger order, namely the order according to (monomorphic) embeddability. In the end, ordering trees by the existence of an embedding is very  natural. From a general mathematical perspective it can be considered even more natural than ordering by weak embeddings. 

Let us  write $T\le^* T'$ if there is an embedding (i.e. a monomorphism) $T\to T'$. If $T\le^* T'$ and $T'\le^* T$ we write $T\equiv^* T'$. Respectively, if $T\le^* T'$ but $T'\not\le^* T$, we write $T<^*T'$. 
If there is a weak embedding from $T$ to $T'$, we write $T\le T'$. If $T\le T'$ and $T'\le T$ we write $T\equiv T'$. Finally, if $T\le T'$ but $T'\not\le T$, we write $T<T'$. Of course, $T\le^* T'$ implies $T\le T'$. 

If $B_\alpha$ is the tree of descending chains of elements of $\alpha$, ordered by end-extension, then $\alpha\le\beta$ if and only if $B_\alpha\le B_\beta$. Thus in the class of trees without infinite branches the weak embedding order reflects the received ordering of the class of all ordinal numbers.
 Again, we may ask, whether there is a maximal tree under the weaker ordering $\le$ in the class of all Aronszajn trees. If we assume $\MA_{\aleph_1}$, then no wide Aronszajn tree is $\le$-above all Aronszajn trees \cite{Sh:1186,MR2329763}. Similarly, if $V=L$, then for every wide Aronszajn tree $T$ there is a Souslin tree $S$ such that $S\not\le T$  \cite{MR1712000}.

We may now ask in two different senses whether  there is a maximal tree in a given class of trees:

\begin{quote}
    {\bf The  Maximality Question:} Given a class $\cC$ of trees, is there a tree $T$ in $\cC$  such that $S\le^* T$ for every $S\in \cC$?
\end{quote}

\begin{quote}
    {\bf The Weak Maximality Question:} Given a class $\cC$ of trees, is there a tree $T$ in $\cC$  such that $S\le T$ for every $S\in \cC$?
\end{quote}

Trivially, a positive solution to the  Maximality Question gives a positive solution to the Weak Maximality Question.

Both the full and the Weak Maximality Questions are meaningful even if the maximal tree $T$ is not in $\cC$ but satisfies some weaker constraints. 
%Our Theorem \ref{Theorem:main} is an example of this. Another example is the following:
%
%
%It is consistent that there is a wide Aronszajn-tree which is $<$-above all  Aronszajn-trees. (If $V=L$, there is none by Theorem~\ref{TV}.)
%
For example, it is consistent, relative to the consistency of ZF, that CH holds and every Aronszajn-tree is special (Jensen, \cite[Theorem 8.5]{MR1623206}). Thus in this model there is a wide  Aronszajn-tree that is $\le$-above all Aronszajn-trees. However, this tree $T$ is (a priori) not Aronszajn, so we do not obtain a solution to the Weak Maximality Question for the class of Aronszajn trees.
%
%  The following question remains open:
%
%\begin{question}\label{question:widearonszajn}
%Is it consistent to have a $\le$-maximal wide $\kappa$-Aronszajn-tree? The question is also open if ``wide" is dropped. 
%\end{question}
%
%This is open even for $\kappa=\aleph_1$.
%
 Let us call a wide $\aleph_2$-Aronszajn-tree $T$ \emph{special} if there is $f:T\to\omega_1$ such that $t<t'$ always implies $f(t)\ne f(t')$. Consistency of a weakly compact cardinal implies the consistency of $2^{\aleph_0}=\aleph_1$ $+$  $2^{\aleph_1}>\aleph_2+$  every wide $\aleph_2$-Aronszajn-tree is special \cite{MR603771}.  In this model there is a tree $T$ $\le$-above all $\aleph_2$-Aronszajn-trees such that $T$ has no $\aleph_2$-branches.  Here $|T|>\aleph_2$, so again $T$ is not an answer to the Weak Maximality Question for wide $\aleph_2$-Aronszajn trees.  
 %Asperó and   Golshani have announced the result, so far unpublished (see however \cite{arXiv}), that it is consistent, relative to the consistency of a weakly compact cardinal, that GCH holds and all  $\aleph_2$-Aronszajn-trees are special.  This would give an upper bound, a wide $\aleph_2$-Aronszajn-tree,  for all  $\aleph_2$-Aronszajn trees. 
 Our Theorem~\ref{Theorem:main} gives a positive solution to the (full) Maximality Question for wide $\kappa$-Aronszajn trees, $\kappa$ a successor of a regular cardinal $>\aleph_0$. As we see below, it is impossible to combine this with GCH.

Both $T\le^* T'$ and $T\le T'$ measure in their own ways how big the trees $T$ and $T'$ are with respect to each other. If $B_{\omega^*}$ denotes the tree consisting of the single branch of length $\omega$, then $B_\alpha\le B_{\omega^*}$ holds for all $\alpha$ but of course $B_\alpha\not\le^* B_{\omega^*}$ when $\alpha>1$. Thus $B_{\omega^*}$ is $\le$-above a proper class of non-$\equiv$-equivalent trees. There can be only $2^{|T|}$ trees $\le^*$-below a given tree $T$, up to $\equiv^*$. This illustrates the different senses in which $\le$ and $\le^*$ measure the bigness of trees.

A still further ordering of Aronszajn trees is the following: If $T$ is an Aronszajn tree and $C\subseteq \omega_1$, then we use $T\restriction C$ to denote the suborder of $T$ consisting of nodes in $T$ the height of which is in $C$. Suppose $T$ and $T'$ are Aronszajn trees. We say that a partial map $\pi:T\to T'$ is an embedding (or an isomorphism) \emph{on a club} if there is a club $C\subseteq\omega_1$ such that $\pi$ is an embedding (or respectively an isomorphism) $T\restriction C\to T'\restriction C$. It follows from the Proper Forcing Axiom that any two Aronszajn trees are isomorphic on a club \cite{MR788070, MR3826541}.

The following  useful operation on trees is due  to Kurepa \cite{MR77612}:
If $T$ is a tree, let $\sigma(T)$ be the tree of ascending chains in $T$, ordered by end-extension. For well-founded trees this is like the successor function on ordinals in the sense that  $\sigma(B_\alpha)\equiv B_{\alpha+1}$. It is easy to see that  if $T$ is any tree, then $T<\sigma(T)$. Moreover, if $T$ has no branches of length $\kappa$, neither has $\sigma(T)$. So from the point of view of lengths of branches $\sigma(T)$ is similar to  $T$. However, it is perfectly possible that $|T[<|\sigma(T)|$. For example, if every node in $T$ splits, then $|\sigma(T)|\ge 2^\omega$. 

%Let us write $T\ll T'$ if $\sigma(T)\le T'$. It is easy to see that $\ll$ is well-founded \cite{MR1071304}. The relation $T\ll T'$ is a stronger form of $T<T'$, which it implies, and it is sometimes more useful than mere $T<T'$.

%The following comparison game from \cite{MR1071304} illustrates the  orderings $T\le T'$ and $T'\ll T$. It is a potentially transfinite game between $I$ and $II$ on $T$ and $T'$. First $I$ chooses $t_0\in T$. Then $II$ chooses $t'_0\in T'$. Always, after $II$ has moved, $I$ chooses, if he can, a node $t\in T$ which is in $T$ above all the previous elements he has played. Then $II$ chooses, if she can, a node $t'\in T'$ which is  in $T'$ above all the previous elements she has played. The player who cannot move loses and the opponent wins. In this game Player $II$ has a winning strategy  if and only if $T\le T'$ and  Player $I$ has a winning strategy if and only if $T' \ll T$.

The $\sigma$-operation shows that if $\aleph_\alpha^{<\aleph_\alpha}=\aleph_\alpha$, the class $\T_\alpha$ does not have a $\le$-maximal element. So in that case even the Weak Maximality Question has a negative answer for the class $\T_\alpha$. In consequence, $\mu^{<\mu}>\mu$ holds in the final model of our Theorem~\ref{Theorem:main}.

%For $\kappa^{<\kappa}>\kappa$ the  question is open (see Question~\ref{question:widearonszajn}).

If $A\subseteq\omega_1$ is co-stationary, let $T(A)$ be the tree of closed increasing sequences of elements of $A$. The class of such trees $T(A)$ is an interesting subclass of  trees without uncountable branches.  A tree without uncountable branches which is of cardinality $\le 2^\omega$ and $\le$-above all such $T(A)$ is known as a \emph{Canary Tree}. The existence of Canary Trees is independent of $ZFC+GCH$ \cite{MR1222536,MR1877015}. Assuming CH, a Canary Tree is perhaps not maximal in the entire class of trees in $\T_1$     but it still $\le$-majorises the large class of trees of the form $T(A)$.

\subsection*{Trees as game clocks}

We already alluded to the fact that, assuming CH,  Canary Trees can be used to show that the isomorphism class of a particular structure, in this case the free Abelian group of cardinality $\aleph_1$, is $\Delta^1_1$ in the generalized Baire space $\omega_1^{\omega_1}$. This is an example of the use of trees as clocks in games in the way we now describe. A maximal tree would represent  a kind of universal clock. To see what this means, suppose $\delta$ is an ordinal and $G$ is a  game of length $\delta$ between $I$ and $II$ in which $I$ and $II$ produce a $\delta$-sequence of elements of a fixed set $M$, alternating moves, $I$ starting each round. We fix a set $W\subseteq M^\delta$  and say that $II$ wins if the sequence played is in $W$. Otherwise $I$ wins. We assume the game is closed in the sense that if $s\notin W$ then there is an initial segment $s'$ of $s$ such that no extension of $s'$ is in $W$. If $T$ is any tree (of height $\delta$), we can define a new game $G_T$, a kind of approximation of $G$, as follows. Every time $I$ moves in $G$ he also picks a node $t$ in $T$ in such a way that $t$ is above all nodes he has picked during previous rounds of $G$. If he cannot pick such a $t$ then he loses. Otherwise the game is played as $G$. Clearly, if $I$ has a winning strategy in $G_T$, he has also in $G$. The role of $T$ in $G_T$ is to make it harder for $I$ to win. If $T$ is well-founded, player $I$ can only win $G_T$ if he can win $G$ in finitely many moves but he does not have to tell in advance how many moves he needs in order to win. If $T$ has no branches of length $\delta$, player $I$ can only win $G_T$ if he can win $G$ in $<\delta$ moves but, again, he does not have to tell in advance how long $\delta'$-sequence, $\delta'<\delta$, of moves he needs in order to win. He can change his mind about this during the game. 

The following implications are immediate:

\begin{enumerate}
\item If $II$ has a winning strategy in $G_{T'}$ and $T\le T'$, then $II$ has a winning strategy in $G_{T}$. 
\item If $I$ has a winning strategy in $G_T$ and $T\le T'$, then $I$ has a winning strategy in $G_{T'}$. 
\item If $II$ has a winning strategy in $G_T$ and $I$ has a winning strategy in $G_{T'}$, then $T< T'$. 
\end{enumerate}

These implications emphasise the role of maximal trees for the games $G_T$. Let us see how this works, first on a general level and then more specifically.  Let $\cC_G$ be the class of trees $T$  such that II has a winning strategy in $G_T$. If $II$ has a winning strategy even in the non-approximated game $G$, the class $\cC_G$ is simply the class of all trees. The other extreme is that $\cC_G=\emptyset$, which happens if $W=\emptyset$. Suppose $S$ is $\le$-above all trees in $\cC_G$. Then $\sigma(S)\notin\cC_G$. So, maximality of the tree gives us negative information about winning strategies of $II$. Let $\cC'_G$ be the possibly bigger class of trees $T$  such that $I$ does not have a winning strategy in $G_T$. Again, $\cC'_G$ may be the class of all trees and it is also possible that $\cC'_G=\emptyset$.  Suppose $S'$ is $\le$-above all trees in $\cC'_G$. Then $\sigma(S')\notin\cC'_G$ i.e. $I$ has a winning strategy in $G_{\sigma(S')}$. So, maximality of the tree gives us positive information  about winning strategies of 
$I$.

A particular closed game of interest in this connection is the transfinite  EF-game. Suppose $M$ and $N$ are models of the same vocabulary, $|M|=|N|=\delta$ and $M\ncong N$. Let $\tau$ be the canonical enumeration strategy (i.e. $I$ enumerates $M\cup N$) of $I$ in the EF-game $\EF^\delta(M,N)$ of length $\delta$ on $M$ and $N$ such that both players are allowed to play a sequence of $<|\delta|$ elements at a time. Because we assume $M\ncong N$, $\tau$ is a winning strategy of $I$. The pairs  $(T,T')$, $T\le T'$,  of trees such that Player II has a winning strategy in  $\EF^\delta(M,N)_T$ but Player I has a winning strategy in $\EF^\delta(M,N)_{T'}$ give information about how far or close  $M$ and $N$ are from being isomorphic. Such pairs outline a kind of boundary where advantage in the game $\EF^\delta(M,N)$ moves from Player II to Player I. Every tree with a branch of length $\delta$ is above the boundary. If $\delta=\omega$, the boundary is (up to $\equiv$) exactly one tree, namely $B_\alpha$ for some (unique) countable ordinal $\alpha$. If $\delta>\omega$, the boundary may be quite wide. Let us assume $\delta=\omega_1$. If the first order theory of $M$ is classifiable in the sense of stability theory, the boundary lies between well-founded trees and non-well-founded trees \cite{MR1083551}. In the opposite case the boundary may be quite high in the class of trees without uncountable branches. For models of size $\aleph_1$ of unstable theories it is above any tree in $\T_1$, if CH is assumed \cite{MR1111753}.

%Given $M$ and $N$ of cardinality $\delta$ such that $M\ncong N$, and a cardinal $\lambda$ such that $\lambda^{<\delta}=\lambda$, let $\cC$ be the class of trees $T$ of cardinality $\le\lambda$ such that II has a winning strategy in $\EF^\delta_T(M,N)$. Since $M\ncong N$, the trees in $\cC$ have no branches of length $\delta$. Suppose $T\in\cC$ is $\le$-maximal. Then $II$ does not have a winning strategy in $\EF^\delta_{\sigma(T)}(M,N)$. Let $\cC'$ be the class of trees $S$ of cardinality $\le\lambda$ such that $I$ does not have a winning strategy in $\EF^\delta_S(M,N)$. Again such trees $S$ cannot have branches of length $\delta$. Suppose $T'\in\cC$ is $\le$-maximal. Clearly then $I$ has a winning strategy in $\EF^\delta_{\sigma T'}(M,N)$. It follows that $T\le T'$. The pair $(T,T')$ of maximal trees is a kind of boundary among trees of cardinality $\le\lambda$, comparable to the more general $\partial(\EF^\delta(M,N))$, where advantage in $\EF^\delta(M,N)$ shifts from $II$ to $I$.

\begin{quote}
    {\bf Open Question:} Are there, for every tree $T\in\T_1$ non-isomorphic models $M$ and $N$ of cardinality $\aleph_1$ such that Player II has a winning strategy in $\EF^{\omega_1}(M,N)_T$?
\end{quote}

A positive answer is known only for the extremely simple trees which consist of countable branches bunched together at the root \cite{MR2431056}. A positive answer follows also from CH \cite{MR1111753}. If there is a weakly maximal tree $T$ in $\T_1$, solving the above question for $T$ gives  automatically a positive answer for all trees in $\T_1$.

The analogue of the Scott height of a countable model in this context is the following, introduced in \cite{MR1111753}: 
A tree $T$ without branches of length $\omega_\alpha$ is called a \emph{universal non-equivalence tree} for a model $M$ of cardinality $\aleph_\alpha$ if for all models $N$ of cardinality $\aleph_\alpha$ in the same vocabulary, if $M\ncong N$, then Player $I$ has a winning strategy in $\EF^{\omega_\alpha}(M,N)_T$. For example, a Canary Tree is (if it exists) a universal non-equivalence tree for the free Abelian group of cardinality $\aleph_1$.
A tree $T$ without branches of length $\omega_\alpha$ is called a \emph{universal equivalence tree} for a model $M$ of cardinality $\aleph_\alpha$ if for all models $N$ of cardinality $\aleph_\alpha$ in the same vocabulary, if Player $II$ has a winning strategy in $\EF^{\omega_\alpha}(M,N)_T$, then $M\cong N$. If $\alpha=0$, every countable model has a universal non-equivalence tree $B_{\alpha+1}$ and a universal equivalence tree $B_{\alpha}$, where $\alpha$ is the Scott height of the model.
For uncountable models the existence of such universal trees depends on stability theoretic properties of the first order theory of the model \cite{MR1111753,MR1295983,MR1367209,MR1777775}. By and large, depnding on cardinal arithmetic, models whose first order theory is unstable have no universal equivalence tree \cite{MR1111753}. Models whose first order theory is superstable, NOTOP and NDOP, have a universal equivalence tree
\cite{MR1111753}.

 %The proof of our main result can be used to get the consistency of the existence of  universal wide $\kappa$ tree, with $2^{<\kappa}>\kappa$ .  

%This paper continues  in spirit \cite{part1} but is self-contained and can be read without knowledge of \cite{part1}.
Our main result  %improves the main result of \cite{part1} by extending the maximality from narrow $\kappa$-Aronszajn trees to wide $\kappa$-Aronszajn trees. This 
leaves %still 
open the possibility of having a narrow $\kappa$-Aronszajn tree which is maximal with respect to wide $\kappa$-Aronszajn trees under strict order preserving homomorphisms.

\subsection*{An outline of the paper}

After some preliminaries in Section \ref{Section:Preliminaries}, we use in Section \ref{Section:BuildingTree} a weakly compact cardinal $\kappa$ to force  a wide $\aleph_2$-Aronszajn tree $T$ by Levy-collapsing $\kappa$ to $\aleph_2$. The tree $T$ is the tree that will be the desired maximal tree in the final model. The levels of $T$ are sufficiently collapse-generic to permit the wide $\aleph_2$-Aronszajn trees arising in the construction to be embedded into $T$. We then define in Section \ref{Section:TheEmbeddingPoset} a $\sigma$-closed countable support iteration of length $\aleph_3$ of forcing with side conditions designed by means of an appropriate book-keeping to force for every wide $\aleph_2$-Aronszajn tree $S$ an embedding $S\to T$. Naturally, we have to make sure $\aleph_2$ is not collapsed during this forcing. Section \ref{Section:StrongProperness} is devoted to showing that our iterated forcing has the right kind of strong properness to guarantee the $\kappa^+$-chain condition and thereby the preservation of $\aleph_2$. We have to also make sure that our tree $T$ will not acquire a long branch during the iteration. This is shown in Section \ref{Section:NoNewBranches}. Theorem~\ref{Theorem:main} is then proved.
We conclude this work with a short open problems section \ref{Section:OpenProblems}.

\subsection*{Our methodology}
 
 When we want to get the consistency of the existence of a maximal wide $\kappa$-Aronszajn  tree,  we face the challenges of preserving $\kappa$, and of showing that the intended universal tree does not obtain a cofinal branch by the iteration. 
 We deal with the former challenge by maintaining that the forcing is strongly proper with respect to sufficiently many structures of cardinality $<\kappa$. 
    Proper forcing methods involving specializing Aronszajn trees has been used in \cite{BG23}. The transition to wide trees and tree embeddings requires the development of a new type of argument for maintaining strong properness, which is developed in  Section 5 of this paper. 
    In addition, having no cofinal branch is a typical example of a second order property of an object that is supposed to be preserved under the iteration.

 \begin{comment}
 The concept that turned out to be useful for this  is the following strengthening of the notion of strongly generic conditions. 
 \begin{definition} A forcing notion $\po$ is said to be \textbf{pairwise strongly proper} with respect to the structure $M$ if there is an $M$ residue function $p\mapsto [p]_M$ such that if $[p]_M=[q]_M$ and $w
 \in M\cap \po,w\leq [p]_M$, then there are $p'\leq p,w$ and $q'\leq q,w$ such that $[p']_M=[q']_M$.
 \end{definition}\todo{We don't have this strong version}

 In Section \ref{Section:NoNewBranches} we further develop this notion, prove it is satisfied by our iteration, and use it to prove that our universal tree remains Aronszajn in the final generic extension.
 \end{comment}
 
 To secure the preservation of $\kappa$ and to make sure that no cofinal branches are introduced to the maximal tree, we introduce a forcing with certain special features:
 \begin{itemize}
     \item The original object we intend to become the maximal tree is highly generic. Specifically, a key requirement of each individual poset is that whenever it maps a node $s$ at level $\alpha < \kappa$ of a given tree $S$ to a node $t$ in the intended universal tree then 
the local branch $b_t$ below $t$  is mutually generic from the generic information of the local branch $b_s$ below $s$.

    \item  We make use of substructures $M$ as ``side-conditions'' to guide the generically constructed embeddings. It is crucial that the chosen structures reflect second order assertions about the objects involved in the forcing. The existence of such structures requires the involvement of large cardinals.
 \end{itemize}

%We believe that these arguments can be framed in much more general context and yield a forcing iteration theorem , with many more applications.
${}$
%One of the main objectives of this research is to develop similar schemes 
We adopt the following general schema
for proving the consistency of the existence of a universal object of cardinality $\kappa$ in a class of structures satisfying some second order sentence $\Phi$. The scheme consists of: 
\begin{enumerate}
\item Force an object $A$ intended to be the universal object for the property $\Phi$.
\item By dovetailing, iterate forcings which embed each individual structure satisfying $\Phi$ into $A$.
\item Show that the iteration  is proper (or strongly proper) so that we do not collapse the relevant cardinals.
\item Show that $A$  satisfies $\Phi$ after the iteration by using a ``splitting argument" and the fact that the iteration satisfies a certain variation of strong properness. 
\end{enumerate}

\section{Preliminaries}\label{Section:Preliminaries}

The rest of the paper is devoted to showing that consistently, there can be a maximal wide $\kappa$-Aronszajn tree, for any double successor cardinal $\kappa$. We review some facts needed for the proof. Our notation is standard and follows \cite{jech2003set}.

\subsubsection*{Trees: Preliminaries}\label{Section:Preliminaries}
A tree $(T,<_T)$ is a partially ordered set with a minimal element (root) and with the property that for every $t \in T$, the set of its $<_T$-predecessors 
$b_t = \{ \bar{t} \in T \mid \bar{t} <_T t\}$ is well-ordered by $<_T$.
We refer to $b_t$ as the \emph{branch below $t$}.
For an ordinal $\alpha$, the $\alpha$-th level of $T$, denoted  $\Lev_\alpha(T)$ is the set of all $t \in T$ so that $b_t$ has ordertype $\alpha$ in $<_T$. The union 
$\bigcup_{\alpha' < \alpha}\Lev_{\alpha'}(T)$
 is denoted by $\Lev_{<\alpha}(T)$.
The height of the tree $T$ is the minimal $\kappa$ such that $\Lev_\kappa(T) = \emptyset$.
Let $T$ be a tree of height and size $\kappa$.
 We say that $T$ is \emph{narrow}  if $|\Lev_\alpha(T)| < \kappa$ for every $\alpha < \kappa$. Otherwise, we say that $T$ is \emph{wide}.
A subset $b \subseteq T$ is a cofinal branch if it is well ordered by $<_T$ and has order-type $\kappa$. We say that $T$ is $\kappa$-Aronszajn if it has no cofinal branches, i.e. branches of ordertype $\kappa$.
If $M$ is a transitive set that is closed under taking predecessors in the tree order $<_T$, and $t \in T \setminus M$, we define the \emph{exit node} $e_T(t,M)$ of $t$ from $M$ to be the $<_T$-minimal node $e \in b_t \cup \{t\}$ that does not belong to $M$.

\subsubsection*{Weakly compact cardinals: Preliminaries}
A second-order formula $\psi$ is $\Pi^1_1$ if it has at most one second-order quantifier, and that quantifier is universal.

A cardinal $\kappa$ is \textit{weakly compact} if for every $B \subseteq V_\kappa$ and every $\Pi_1^1$ statement $\psi$ true in $(V_\kappa,\in,B)$ the set 
$$
A_\psi = \{\alpha < \kappa \mid (V_\alpha,\in,B \cap V_\alpha) \models \psi\}$$
is nonempty. It follows from the definition that the collection of sets 
$$\{ A_\psi \mid (V_\kappa,\in,B) \models \psi, \thinspace \psi \text{ is } \Pi_1^1 \text{ and } B \subseteq V_\kappa\}$$
generates a $\kappa$-complete normal filter on $\kappa$, denoted by $\fwc$. 

\begin{definition}[Reflecting 
structures]\label{Def:ReflectingSequenceofModels}
Let $\theta \geq \kappa^{++}$ be a regular cardinal and $<_\theta$ a well-ordering of $H_\theta$. For every $P \in H_\theta$ 
we define the reflecting sequence of $P$, 
$$\vec{M}^P = \la M^P_\alpha \mid \alpha \in \dom(\vec{M}^P)\ra$$ 
to consist of all Skolem-hull substructures of $H_\theta$ of the form \[M^P_\alpha = Hull^{(H_\theta,\in,<_\theta,P)}(\alpha)\] with the following properties:
\begin{itemize}
    \item $M^P_\alpha \cap V_\kappa = V_\alpha$, 
    \item for every $Q \subseteq V_\kappa$ such that $Q\in M^P_\alpha$ and a $\Pi_1^1$ statement $\psi$ true in $(V_\kappa,\in,Q)$, if $(V_\kappa,\in,Q) \models \psi$ then $(V_\alpha,\in, Q\cap V_\alpha) \models \psi$. 
\end{itemize}
\end{definition}

\noindent It follows from a standard argument that 
$\dom(\vec{M}^P)$ belongs to $\fwc$ for every $P \in H_\theta$. 

\subsubsection*{Levy Collapse: Preliminaries} Let $\po = \col(\omega_1,<\kappa)$ be the Levy-collapse poset. Conditions $p \in \po$ are countable partial functions $p : \omega_1 \times \kappa \to \kappa$ with the property that $p(\nu,\alpha) < \alpha$ for every $(\nu,\alpha) \in \dom(p)$.
Let $G \subseteq \po$ be a generic filter. For each $\alpha <\kappa$ let $f^G_\alpha : \omega_1 \to \alpha$ be given by $f^G_\alpha(\nu) = \alpha'$ iff there is $p \in G$ and $p(\nu,\alpha) = \alpha'$.
We refer to $f^G_\alpha$ as the collapse generic surjection from $\omega_1$ onto $\alpha$ that is derived from $G$.
By a well-known argument, $\po$ is isomorphic to any number $\tau \leq \kappa$ of copies of itself.
Fix an isomorphism between $\po$ and $\kappa \times \kappa \times \kappa$ copies of itself, 
$$\po \cong \prod_{(\eta,\beta,\alpha) \in \kappa^3} \po(\eta,\beta,\alpha)$$ 
i.e.,  $\po(\eta,\beta,\alpha) = \col(\omega_1,<\kappa)$ for all $\eta,\beta,\alpha \in \kappa$. 
The isomorphism breaks a generic filter $G\subseteq \col(\omega_1,<\kappa)$ to mutually generic filters 
$$\la G(\eta,\beta,\alpha) \mid (\eta,\beta,\alpha) \in \kappa^3\ra,$$ $G(\eta,\beta,\alpha) \subseteq \po(\eta,\beta,\alpha)$.
For each $\tau < \kappa$, let $$f^{G(\eta,\beta,\alpha)}_\tau : \omega_1 \to \tau$$ denote the collapse generic surjection from $\omega_1$ to $\tau$, derived from $G(\eta,\beta,\alpha)$. \\

%\noindent Let $\kappa$ be a cardinal and $X$ be a set of cardinality $\kappa$ (e.g., $\kappa^3= \kappa \times \kappa \times \kappa$). We say that two functions $h,h' :X \to \kappa$ \emph{disagree almost everywhere } if $$|\{ x \in X \mid h(x) =h'(x) \}| < \kappa.$$

\section{Building the Wide Tree $T$}\label{Section:BuildingTree}

In this section, we construct a wide tree $T$ in a generic extension $V[G]$, where $G $ is a generic filter on $\col(\omega_1,<\kappa)$. This will be the maximal tree in the final model.

%For a node $t$ of level $\alpha$ we can ask about the minimal $\beta(t)$ such that the branch leading to $t$ goes through nodes whose index is less than $\beta(t)$. Since we are talking about wide trees, we need to prepare $\kappa$ many nodes $t$ of level $\alpha$ whose $\beta(t)$ is a particular $\beta$. So each of the branches at level $\alpha$ will be associated with a particular $h_\delta(\alpha)$, a particular value of $\beta(t)$,  and an index $<\kappa$.  
%\noindent 
A key ingredient of the construction of the tree $T$ is that its local branches are generically independent. More precisely, considering all $\alpha < \kappa$ of uncountable cofinality, and $t \in \Lev_\alpha(T)$, their associated branches $b_t = \{ t' \in \Lev_{<\alpha}(T) \mid t' <_{T} t\}$ are generically independent of each other in the sense that the parts of the collapse  generic filter $G$ that is required to determine their identity are independent. 
To do this, we associate to each node $t \in \Lev_\alpha(T)$ four parameters
%$(\eta,\beta,\alpha,\tau) \in \kappa^4$ 
$(\alpha,\beta,\gamma,\delta) \in \kappa^4$
called the \emph{collapse index of} $t$ and use them to determine a segment of $G$ that will define $b_t$. 

We can now give the construction of the wide tree $T \in V[G]$.
The domain of $T$ is the set $\kappa\times\kappa$, and each $\alpha$-th level is the set 
$$\Lev_{\alpha}(T) = \{\alpha\}\times\kappa.$$

The ordering $<_{T}$ is constructed level-by-level and makes use of a fixed isomorphism between $\col(\omega_1,<\kappa)$ and the poset obtained by taking $\kappa^3$ many copies of $\col(\omega_1,<\kappa)$,  %$$\prod_{(\eta,\beta,\alpha) \in \kappa^3} \po(\eta,\beta,\alpha)$$ 
$$\prod_{(\alpha,\beta,\gamma) \in \kappa^3} \po(\alpha,\beta,\gamma)$$ 
with the conventions given in Section \ref{Section:Preliminaries} above: the conditions in $\po$

We maintain an inductive assumption that the restriction of $<_{T}$ to $\Lev_{<\alpha}(T)$ belongs to the intermediate extension $V[G\uhr \alpha \times \kappa \times \kappa]$, 
generic over $V$ for the partial product poset $$\po \uhr (\alpha \times \kappa \times \kappa)  = \prod_{ (\alpha',\beta,\gamma) \in \alpha \times \kappa^2 } \po(\alpha',\beta,\gamma).$$

Suppose that $<_{T}\uhr \alpha \times \kappa$ has been defined  for some $\alpha<\kappa$, i.e. the tree order of $T$ has been defined on $\Lev_{<\alpha}(T)$. \\

\noindent
If $\alpha = \alpha' + 1$ is a successor ordinal then we extend the order to the new $\alpha$-th level of $T$ by adding $\kappa$ many successive nodes above each $t' \in \{\alpha'\} \times \kappa = \Lev_{\alpha'}(T)$. We use a pairing bijection $\la \cdot,\cdot \ra : \kappa \times \kappa \to \kappa$ from $V$ to do this.\\

\noindent
Suppose that $\alpha < \kappa$ is a limit ordinal. To define the extension of $<_{T}$ to the $\alpha$-th level, it suffices to assign each $t \in \{ \alpha\} \times \kappa$ a cofinal branch $b_t \subseteq \Lev_{<\alpha}(T)$, as we can then define $t' <_{T} t$ if and only if $t' \in b_t$. We consider separately the cases when $\alpha$ is of countable cofinality and when $\alpha$ is of uncountable cofinality. \\ 

\noindent
If $\cof(\alpha) = \omega$ and $\la \alpha_n \mid n < \omega\ra$ is an increasing cofinal sequence in $\alpha$, then each cofinal branch $b \subseteq \Lev_{<\alpha}(T)$ is determined by its $\omega$-subseqence $\la b(\alpha_n) \mid n < \omega\ra \in (\alpha \times \kappa)^\omega$. As $|(\alpha\times\kappa)^\omega| = \kappa$, we may pair the cofinal branches through $\Lev_{<\alpha}(T)$ and the nodes at the $\alpha$-th level. Indeed, we enumerate all cofinal branches in $\Lev_{<\alpha}(T)$ as $\la b'_\nu \mid \nu< \kappa\ra$,  and  for each node $t = ( \alpha,\nu) \in \Lev_\alpha(T)$, define $b_t = b'_{\nu}$. So each cofinal branch has a supremum in $\Lev_\alpha(T)$, and each node in $\Lev_\alpha(T)$ has a unique cofinal branch below it. The tree $T$ will be $\sigma$-closed.\\

\noindent
Suppose that $\alpha < \kappa$ is a limit ordinal of uncountable cofinality. 
The following definition will be used to define the cofinal branches $b \subseteq \Lev_{<\alpha}(T)$ that will be extended to nodes at level $\alpha$.

\begin{definition}\label{Def:b_f}
Suppose that $T'$ is a $\sigma$-closed normal tree on $\alpha \times \beta$, $\delta \geq \alpha \cdot \beta$ is an ordinal and $f : \omega_1 \to \delta$ is a function such that for every $\mu < \delta$, $f^{-1}(\mu) \subseteq \omega_1$ is unbounded. 

\begin{enumerate}
\item 
Define the branch $b^f \subseteq T'$ determined by $f$ to be the sequence $\la t_i \mid i < \omega_1\ra$ defined as follows.

To define $t_0$ we look at the minimal $j < \omega_1$ for which  $f(j) < \delta$ is of the form  $f(j) = \alpha\cdot \beta_0 + \alpha_0 < \alpha \cdot \beta$. We then set $t_0 = (\alpha_0,\beta_0)$.

Suppose that $\la t_i \mid i < i^*\ra$ has been defined.
If $i^*$ is a limit ordinal then $t_{i^*}$ is the limit of the countable sequence $\la t_i \mid i < i^*\ra$. 
Otherwise $i^* = i +1$. 
Define $t_{i^*} = (\alpha^*,\beta^*)$ by looking at $j^* < \omega_1$ minimal such that 
$f(j^*) = \alpha\cdot \beta^* + \alpha^* < \alpha \cdot \beta$, 
and $t_i <_{T'} (\alpha^*,\beta^*)$ for all $i < i^*$. We then set $t_{i^*} = (
\alpha^*,\beta^*)$.

\item If $q  = f \uhr \nu:  \nu \to \delta$ is an initial segment of $f$, then $q$ naturally determines an initial sequence $\la t_i \mid i \leq i_q\ra$ of $b^f$ which has a maximal element $t_{i_q}$. We denote the last node $t_{i_q}$ by $\pi_q(b^f)$ and call it the \emph{projection} of $b_f$ determined by $q$. 
\end{enumerate}
\end{definition}

\noindent
The next lemma follows from a standard density argument and the definition of cofinal branches $b_f$ (\ref{Def:b_f}).

\begin{lemma}\label{Lemma:b_fISCofinal}
Suppose that $T'$ is a $\sigma$-closed normal tree on $\alpha \times \beta$. If $\delta\geq\alpha\cdot\beta$ and $f : \omega_1 \to \delta$ is a $\col(\omega_1,\delta)$-generic over a model $V'$ that contains $T'$, then $b^f \subseteq T'$ is a cofinal branch in $T'$. 
\end{lemma}

\begin{comment}
\begin{lemma}\label{Lemma:FlexibleProjections1} \todo{I think this lemma is not needed}
Suppose that $T' \in V'$. Let $\qo = \prod_{n < \omega} \col(\omega_1,\delta_n)$ be the product of collapse posets with all $\delta_n \geq \alpha \cdot \beta$. Let $\la \name{f_n} \mid n < \omega\ra$ be the $\qo$-name for the sequence of generic collapse functions $f_n : \omega_1 \to \delta_n$, and 
$\la \name{b}^{f_n} \mid  n < \omega\ra$ be the corresponding names of generic cofinal branches $b^{f_n}$ in $T'$. 

For every condition $\vec{q} = \la q_n \mid n <\omega\ra \in \qo$ and a sequence of nodes $\la t_n \mid n < \omega\ra$ in $T'$ such that $\pi_{q_n}(\name{b}^{f_n}) <_{T'} t_n$, there is some $\vec{q^*} = \la q^*_n \mid n < \omega\ra$ extending $\vec{q}$ so that 
$\pi_{q^*_n}(\name{b}^{f_n}) = t_n$ for each $n$. 
\end{lemma}
\end{comment}

We shall use Definition \ref{Def:b_f} to determine the branches $b_t$  for $t \in\Lev_\alpha(T)$ by choosing subtrees $T'$ of $\Lev_{<\alpha}(T)$ and using collapse generic functions $f$ to form branches $b^f$ through $T'$. 
In this setup, $T'$ belongs to the intermediate extension $V'$ of $V$ by $\po' = \po\uhr \alpha \times \kappa \times \kappa$. This means that the tree projection maps $\name{\pi}_p(t)$ from definition \ref{Def:b_f} will be $\po'$-names of nodes in $\Lev_{<\alpha}(T)$. We make the following observation about the nature of the construction that follows from this setup. 

\begin{lemma}\label{Lemma:TreeProjectionAtTopONly}
Suppose that $\name{T'}$ is $\po'$-name of a tree on $\alpha \times \beta$ for some poset $\po'$, $\delta \geq \alpha \cdot \beta$ is an ordinal, and $\name{f}$ is a $\col(\omega_1,\delta)$-name for the generic collapse. Let $b^{\name{f}}$ be the $\po' \times \col(\omega_1,\delta)$-name for the associated generic branch through $T'$, and let $q \mapsto \name{\pi}_q(b^{\name{f}})$ denote the $\po'$-name for the projection assignment to conditions $q \in \col(\omega_1,\delta)$. 
Suppose that $(p',q) \in \po' \times \col(\omega_1,\delta)$ is such that for some $t^* \in T'$, 
$$p' \Vdash_{\po'} \name{\pi}_q(b^{\dot{f}}) = \check{t}^*.$$
For every $t' \in T'$ there is an extension $q' \leq q$ such that 
$$ p' \Vdash_{\po'} \text{ if } t' >_{T'} \check{t}^* \text{ then } t'\leq_{T'}\name{\pi}_{q'}(\dot b^f).$$
\end{lemma}

\noindent We use definition \ref{Def:b_f} above to define the branches $b_t$ for $t \in \Lev_{\alpha}(T)$.

\begin{definition}\label{Def:ValidQuadrupels}
Call a quadruple $(\alpha,\beta,\gamma,\delta) \in \kappa^4$ \emph{valid for $\Lev_\alpha(T)$ }
if the following conditions hold:
\begin{itemize}
    \item The subset $\alpha \times \beta$ of $\Lev_{<\alpha}(T)$ is closed under $<_{T}$,\footnote{I.e., $T\uhr (\alpha \times \beta)$ is a subtree of $\Lev_{<\alpha}(T)$.}

    \item $\alpha\cdot\beta\leq\gamma\leq\delta$.

    %\item $g_\gamma = h\uhr (\alpha+1)$.\footnote{$g_\gamma$ is the $\gamma$-th function in the enumeration $\la g_\gamma \mid \gamma < \kappa\ra$ of ${}^{<\kappa}\kappa$.} 
\end{itemize}
\end{definition}

For each valid quadruple  $(\alpha,\beta,\gamma,\delta)$, 
the function $f = f^{G(\alpha,\beta,\gamma)}_\delta : \omega_1 \to \delta$ satisfies the assumption of Lemma \ref{Lemma:b_fISCofinal} with respect to the tree $T' = \alpha \times \beta \subseteq \Lev_{<\alpha}(T)$. 
Since the tree ordering on $\Lev_{<\alpha}(T)$ is assumed to be defined in $V' = V[G\uhr \alpha \times \kappa \times \kappa]$ and $f_\delta^{G(\alpha,\beta,\gamma)}$ is $\col(\omega_1,\delta)$ generic over $V'$, it follows from Lemma \ref{Lemma:b_fISCofinal} that the branch $b^{f^{G(\alpha,\beta,\gamma)}_\delta}$ is cofinal in $\Lev_{<\alpha}(T)$. \\

\noindent
It is clear that the set of valid quadruples for $\Lev_{\alpha}(T)$ has size $\kappa$.
Let \[\la (\alpha_\nu,\beta_\nu,\gamma_\nu,\delta_\nu) \mid \nu < \kappa\ra\] be an enumeration of all valid quadruples for $\Lev_{\alpha}(T)$.
%We may assume that $\alpha_\nu,\beta_\nu,\tau_\nu < \nu$ for each $\nu > 0$.

\begin{definition}[Collapse index]
Let $t = (\alpha,\nu) \in \Lev_\alpha(T)$. The \emph{collapse-index} of $t$ is the quadruple $(\alpha_\nu,\beta_\nu,\gamma_\nu,\delta_\nu)$. 
\end{definition}

\begin{convention}\label{conv:collapseindex}
    Let $t\in T$. Whenever the collapse index of a node $t$ is decided by a condition $p$ and the idensityt of the condition $p$ is clear from context, then the collapse index of $t$ is denoted by $(\alpha_t,\beta_t,\gamma_t,\delta_t)$.
\end{convention}

Collapse index will indicate the part of the generic used to define the branch below $t$. We want to define a wide tree $T$ which will provide branches that are very generic over any small part of it. Collapse index will take care of this. 

\begin{remark}\label{Remark:CollapseIndexIsInjective}
    The assignment of $t \in T$ to its collapse index $(\alpha_t,\beta_t,\gamma_t,\delta_t) \in \kappa^4$ is injective.
\end{remark}

\noindent In $V[G\rest((\alpha+1)\times\kappa\times\kappa)]$, the function associated to each node $t=(\alpha,\nu)\in\Lev_\alpha(T)$ is the function $f_t = f^{G(\alpha_\nu,\beta_\nu,\gamma_\nu)}_{\delta_\nu}$. Set $b_t = b^{f_t}$. We have defined the map from $\Lev_\alpha(T)$ to cofinal branches through $\Lev_{<\alpha}(T)$. We let
\[
s<_Tt:\iff s\in b_t.
\]
This concludes the construction of the tree $T$. 
We state a number of basic properties of $T$  that follow from its construction. 
The first is an immediate consequence of our  level-by-level definition of $T$.

\begin{lemma}
Suppose that $M_\alpha \elem (H_{\kappa^{++}},\in)$ is an elementary substructure with ${}^\omega M_\alpha \subseteq M_\alpha$ and $M_\alpha \cap \kappa = \alpha$ is a regular cardinal,   then
$$\po \cap M_\alpha = \prod_{(\alpha',\beta',\gamma')\in \alpha^3} \po(\alpha',\beta',\gamma') \cap V_\alpha$$
is a regular subforcing of $\po$, and $T \cap M_\alpha$ is forced to be equivalent to a $(\po \cap M_\alpha)$-name of a subtree of $T$.
\end{lemma}

\noindent
Let $t \in \Lev_\alpha(T)$ and let $(\alpha,\beta,\gamma,\delta)$ be the collapse index of $t$, for some $\alpha<\kappa$ of uncountable cofinality.
Then $b_t$ is a cofinal branch in the subtree $T \uhr (\alpha \times \beta)$ whose order is defined in the intermediate generic extension $V' = V[G \uhr \alpha\times \kappa \times \kappa]$. The branch itself comes from the $V'$-generic filter $G(\alpha,\beta,\gamma,\delta)$ on $\col(\omega_1,\delta)$.
%\footnote{Using our assumption of $\alpha_\nu,\beta_\nu,\tau_\nu < \nu$ for every $\nu > 0$.} 
Working in $V'$,  definition \ref{Def:b_f} allows us to assign each $q \in \col(\omega_1,\delta)$ a node $\pi_q(t) \in T\uhr (\alpha \times \beta)$ which is the maximal node forced by $q$ to be in $b_t$. 
Back in $V$, $\name{\pi}_q(t)$ is a $\po\uhr (\alpha\times \kappa \times \kappa)$-name for a node in $\Lev_{<\alpha}(T)$ that is decided on a dense open set.

\begin{definition}\label{Def:D(t)}
For each node $t \in \Lev_{\alpha}(T)$, 
let $D(t) \subseteq \po$ be the set of conditions $p$ that decide the collapse index of $t$ and so that $p\uhr  \alpha \times \kappa \times \kappa$ decides the $\po\rest(\alpha\times\kappa\times\kappa)$-name of the node $\name{\pi}_q(t) \in \Lev_{<\alpha}(T)$ where $q = p\uhr \{ (\alpha_t,\beta_t,\gamma_t,\delta_t)\}\in\col(\omega_1,\delta_t)$ (See Convention \ref{conv:collapseindex}). Being decided by $p$, we denote it by $\name{\pi}_p(t)$. 
\end{definition}

\noindent
It follows from the construction of $T$ that $T$ is $\sigma$-closed, i.e. that every countable increasing sequence of nodes has a supremum. It follows from the construction of the tree that for each $t \in T$, the projection $\name{\pi}_p(t)$ is defined for densely many $p \in \po$.
\begin{lemma}
$D(t)$ is dense and $\sigma$-closed for each $t \in T$.
\end{lemma}
\begin{proof}
    For density, fix $p\in\po$ and $t\in T$. Up to extending $p$, we assume that it decides the collapse index of $t$. Suppose that the collapse index of $t$ is $(\alpha,\beta,\gamma,\delta)$. Then the branch below $t$ is determined by the countable partial function $f:=p(\alpha,\beta,\gamma,\delta\cdot):\omega_1\to\delta\in\col(\omega_1,\delta)$. We may extend $p\cap V_{\alpha\cdot\beta}$ to some $p'$ in $\po\cap V_{\alpha\cdot\beta}$, such that $p'$ decides the tree-order among nodes in the set
    \[
    \rng(f)\cap \{\alpha\cdot\beta^*+\alpha^*<\alpha\cdot\beta\mid(\alpha^*,\beta^*)\in\alpha\times\beta\}
    \]
    Then the pointwise union $p'\cup p$ belongs to $D(t)$. The fact that $D(t)$ is $\sigma$-closed follows similarly.
\end{proof}

\noindent
The following is a straightforward application of Lemma \ref{Lemma:TreeProjectionAtTopONly} to countably many mutually generic branches. 

\begin{lemma}\label{Lemma:FlexibleProjections2}
Let $\la t_n \mid n < \omega\ra$ be a sequence of distinct nodes in $T$. 
Suppose that $p \in \bigcap_n D(t_n)$ and  $\la t'_n \mid n < \omega\ra$ is a sequence of nodes such that for each $n < \omega$, $t'_n \in T \uhr (\alpha_{t_n} \times \beta_{t_n})$.
Then there is an extension $p' \leq p$ with the following properties:
\begin{enumerate}
    \item $p'$ and $p$ are equal, except maybe at the collapse indices $(\alpha_{t_n},\beta_{t_n},\gamma_{t_n},\delta_{t_n})$ of the points $t_n$.
    \item for every $q \leq p'$, if $q \Vdash \name{\pi}_{p}(t_n) \leq_{T} t'_n$ then
    $q \Vdash t'_n \leq_{T} \name{\pi}_q(t_n)$.
\end{enumerate}
\end{lemma}
\noindent Note that $q \Vdash t'_n \leq_{T} \name{\pi}_q(t_n)$ implies $q \Vdash t'_n \leq_{T} t_n$.

\begin{proof}[Proof of Lemma \ref{Lemma:FlexibleProjections2}]
    For each $n<\omega$, denote by $f_n$ the partial function \[p(\alpha_{t_n},\beta_{t_n},\gamma_{t_n},\delta_{t_n},\cdot):\omega_1\to\delta_{t_n}\] in $\col(\omega_1,\delta_{t_n})$. For each $n<\omega$, let $j_n=\min(\omega_1-\dom(f_n))$. Suppose that $t'_n=(\alpha'_n,\beta'_n)$. Define $p'$ from $p$ by adding the pair
    \[
    (\alpha_{t_n},\beta_{t_n},\gamma_{t_n},\delta_{t_n},j_n,\alpha\cdot\beta'_n+\alpha'_n),
    \]
    for every $n<\omega$. Then $p'$ is as wanted.
\end{proof}

We summarize the section by a lemma that shows that the tree $T$ is wide Aronszajn in $V[G]$ for any collapse generic filter $G\subseteq\col(\omega_1,<\kappa)$.

\begin{lemma}
    Let $G\subseteq\col(\omega_1,<\kappa)$ be a generic filter. The tree $T$ is a wide $\kappa$-Aronszajn tree in $V[G]$.
\end{lemma}
\begin{proof}
    It is clear that $T$ is a wide $\kappa$-tree. Suppose to the contrary that there is $p\in\col(\omega_1,<\kappa)$ and a name $\dot{b}$ such that
    \[
    p\Vdash\dot b\text{ is a cofinal branch in }\dot T.
    \]
    Let $M\prec H_{\kappa^+}$ be such that $V_\alpha=V_\kappa\cap M$ and $\alpha$ is an ordinal of uncountable cofinality, and such that $\dot b\in M$. Up to extending $p$, we may assume that it decides the $\alpha$-th node on the branch $\dot b$, say $\dot b(\alpha)=t$. Also, without loss of generality $p\in D(t)$. Up to extending $p\cap V_\alpha$ in $\po\cap V_\alpha$, we may assume that it decides the node $\dot\pi_p(t)$, and that there are two dictinct nodes $t^L$ and $t^R$ at some level $\bar{\alpha}<\alpha$ such that $p\cap V_\alpha$ forces that they are above $\dot\pi_p(t)$. Now by Lemma \ref{Lemma:FlexibleProjections2} there are two extensions $p^L,p^R\leq p$ such that 
    \[
    r:=p^L\cap V_\alpha=p^R\cap V_\alpha=p\cap V_\alpha,
    \]
    and such that $p^L\Vdash t^L<_{\dot T}t$ and $p^R\Vdash t^R<_{\dot T}t$. Now $r,\dot b \in M$. By elementarity, there is an extension $w\leq r$ in $\po\cap M$ that decides the node $\dot b(\bar{\alpha})$, say, to be $\bar{t}$. The pointwise unions $w\cup t^L$ and $w\cup t^R$ are both conditions in $\po$. But this is a contradiction: if $\bar{t}\neq t^L$, then $w\cup p^L$ cannot be a condition since it would force both $\bar{t}<_{\dot T}t$ and $t^L<_{\dot T}t$, and similarly, if $\bar{t}\neq t^R$, then $w\cup p^R$ cannot be a condition.
\end{proof}

Our forcing iteration for embedding trees on $\kappa$ into $T$ will have to deal with $\kappa^+$ many trees $\la S_\eta \mid \eta<\kappa^+\ra$. Ideally, each $S_\eta$ would have its own part of $T$ into which it would be embedded. This cannot be accommodated because $T$ is supposed to be of size $\kappa$ and there are $\kappa^+$ many trees to embed. To deal with this, we assign to each node $t$ an ordinal $\gamma_t<\kappa$. This assignment will take care that, in critical places, two nodes from two different trees cannot be assigned to a node with the same label.

\section{The embedding poset}\label{Section:TheEmbeddingPoset}
We fix in $V$ a sequence of functions $\vec{\psi} = \la \psi_\tau \mid \tau < \kappa^+\ra$, such that for each $\tau$, $\psi_\tau : \kappa \to \tau$ is a bijection whenever $\tau\geq\kappa$. %These bijections will be used to tacitly identify each poset $\po_\tau$ with an isomorphic poset that is a subset of $V_\kappa$.
%We also fix a sequence of functions $\la h_\tau \mid \tau < \kappa^+\ra$, $h_\tau : \kappa \to \kappa$, which are pairwise almost everywhere disagree.

\begin{definition}
We say that a set $a \subseteq \kappa^+$ is \emph{$\alpha$-closed with respect to $\vec{\psi}$} %$(\vec{\psi},\alpha)$-closed 
for some $\alpha < \kappa$ if $\psi_\delta``\alpha \subseteq a$ for every $\delta \in a$.
\end{definition}

\noindent For example, if $M\prec H_{\kappa^{++}}$ is such that $\alpha=\kappa\cap M\in\Ord$ and $\vec{\psi}\in M$, then $\kappa^+\cap M$ is $\alpha$-closed with respect to $\vec{\psi}$.

We are going to define by induction a sequence of forcing notions $\po_\tau$ for $\tau<\kappa^+$, 
together with $\po_\tau$-names of wide trees $S_\tau \subseteq \kappa \times \kappa$ and sequences of structures $\vec{M}^\tau = \la M^\tau_\alpha : \alpha \in \dom(\vec{M}^\tau)\ra$. This is our iteration for the proof of Theorem 1.

\begin{remark}
Before giving the exact definitions, 
we give a brief informal description of $\po_\tau$. 
The first poset $\po_0$ will be equivalent to the Levy collapse forcing $\col(\omega_1,<\kappa)$. For $\delta \geq 1$, $\po_\delta$ will consists of pairs $p = \la f^p,N^p\ra$ of countable sets, so that 
$f^p = \la f^p_\gamma : \gamma \in \supp(p)\ra$ is a sequence of functions, such that $\supp(p) \subseteq \delta$ is a countable set, and for each $\gamma \in \supp(p)$, $f^p_\gamma$ is a countable partial function, which will be forced to be order preserving from the tree $S_\gamma$ to $T$. The tree (names) $S_\gamma$, $\gamma < \kappa^+$ will be chosen by a book-keeping function, which is planned to exhaust all wide trees $S$ of a certain kind. The set 
$N^p$ corresponds to the side condition part, which specifies a collection of structures $M$ for which we would like to secure the existence of strong generic conditions. 
This is realized  by having $N^p$ consist of pairs $(\alpha,a_\alpha)$ so that $\alpha < \kappa$ and $a_\alpha \subseteq \delta$ is a nonempty $\alpha$-closed set with respect to $\vec{\psi}$ of size $|a_\alpha| \leq \alpha$. 
The role of $(\alpha,a_\alpha)$ is to specify a set of $\gamma < \delta$ for which we would like $M^\gamma_\alpha \in \vec{M}^\gamma$ to have a strong generic condition (a.k.a. a master condition). For this, we add several natural requirements to the working parts $f^p_\gamma$. For example, we require that nodes $s \in \dom(f^p_\gamma) \cap M^\gamma_\alpha$ are mapped to $f^p_\gamma(s) \in T \cap M^\gamma_\alpha$.
\end{remark}

We proceed with the complete recursive definition of $\po_\tau$, $\vec{M}^\tau$, $\tau < \kappa^+$.  The definition will be given in steps, and require introducing a number of auxiliary definitions and notations. 
The auxiliary definitions will be extensively used to prove that $\po_\tau$ has the desired properties. \\

\noindent
To start, we fix in advance a well-order $<_{H_{\kappa^{++}}}$ of $H_{\kappa^{++}}$, as well as a book-keeping function $\Psi$ whose domain is the set of all posets $\po \in H_{\kappa^{++}}$ which preserve $\kappa$ and satisfy $\kappa^+$.c.c, and $\Psi(\po)$ is a $\po$-name of (wide) tree $S$ on $\kappa$ whose domain is $\kappa \times \kappa$.

\begin{definition}[$\po_0$]\label{de:P0}
The first poset $\po_0$ is equivalent to the Levy collapse forcing $\col(\omega_1,<\kappa)$. 
Formally, it consists of pairs $p = \la f^p,N^p\ra$ where $f^p = \la f^p_0\ra$ is a sequence with an element $f^p_0 \in \col(\omega_1,<\kappa)$, and $N^p = \emptyset$.
\end{definition}

Let $\tau<\kappa^+$ and suppose that $\po_\delta$ has been defined for every $\delta < \tau$. 
Before defining $\po_\tau$ we 
list seven inductive assumptions for $\po_\delta$, $\delta < \tau$.

\vspace{8pt}

\noindent\textbf{Inductive Assumption I: } For every $\delta < \tau$, $\po_\delta$ is a $\sigma$-closed poset of size $|\po_\delta| = \kappa$, and $\po_\delta \in H_{\kappa^{++}}$.

\vspace{8pt}

For each $\gamma < \tau$ let $\name{S}_\gamma = \Psi(\po_\gamma)$ denote the $\po_\gamma$-name for a wide tree with domain $\kappa \times \kappa$ chosen by a fixed book-keeping function $\Psi$.
Let \[
\mathcal{A}^\gamma = \la H_{\kappa^{++}},\in, <_{H_{\kappa^{++}}},\gamma,\Psi,\po_{\gamma}, \vec{\psi} \ra.
\] Note that $S_\gamma = \Psi(\po_\gamma)$ is definable in this structure.
Let \[\vec{M}^\gamma = \la M^\gamma_\alpha : \alpha \in \dom(\vec{M}^\gamma)\ra\] be the associated sequence of $\Pi_1^1$-reflecting elementary substructures of $\mathcal{A}^\gamma$ from Definition \ref{Def:ReflectingSequenceofModels}.
As mentioned after the definition, $\dom(\vec{M}^\gamma)$ belongs to the weakly compact filter on $\kappa$, $\fwc$.

\begin{lemma}
    If $M^\delta_\alpha\in\vec{M}^\delta$ and $\gamma\in\delta\cap M^\delta_\alpha$, then $M^\gamma_\alpha\in\vec{M}^\gamma$ and $M^\gamma_\alpha\subseteq M^\delta_\alpha$.
\end{lemma}
\begin{proof}
    The poset $\po_\gamma$ is definable from $\po_\delta$ and $\gamma$ if $\gamma<\delta$, so if $\gamma\in\delta\in M^\delta_\alpha$, then $M^\gamma_\alpha\subseteq M^\delta_\alpha$. The fact that $M^\gamma_\alpha\cap V_\kappa=V_\alpha$ follows too, if $x\in V_\alpha-M^\gamma_\alpha$, then $x$ is definable from the $<_{\kappa^{++}}$-minimal bijection $\psi:\kappa\to V_\kappa$ that is an element in both $M^\delta_\alpha$ and $M^\gamma_\alpha$ and that must map $\alpha$ to $V_\alpha$ by reflectivity of $M^\delta_\alpha$. The reflectivity follows from the reflectivity of $M^\delta_\alpha$ using the facts that $M^\gamma_\alpha\subseteq M^\delta_\alpha$ and $M^\gamma_\alpha\cap V_\kappa=V_\alpha$.
\end{proof}

\vspace{8pt}

\noindent\textbf{Inductive Assumption II: } For every $\delta < \tau$ and $p \in \po_\delta$, $p$ is of the form $\la f^p,N^p\ra$ where
\begin{enumerate}
    \item $f^p = \la f^p_\gamma : \gamma \in \supp(p)\ra$, a sequence whose domain  $\supp(p) \subseteq \delta$  (called the \emph{support of $p$}) is a countable set so that
    \begin{itemize}
    \item $0 \in \supp(p)$ and $f^p_0 \in \col(\omega_1,<\kappa)$,
    \item for each positive ordinal $\gamma \in \supp(p)$, $f^p_\gamma : \kappa \times \kappa \to T$ is a partial countable function.
    \end{itemize}
    
    \item $N^p$ is a countable set of pairs of the form $(\alpha,a_\alpha)$, where $\alpha<\kappa$ and $a_\alpha \subseteq \delta$ is a nonempty $\alpha$-closed set (w.r.t $\vec{\psi}$) of size $|a_\alpha| \leq \alpha$. For each $\alpha$ there is at most one pair $(\alpha,a_\alpha)$ in $N^p$. 
    
    %and for each $\gamma \in a_\alpha$, 
    %\begin{itemize}
     %   \item $M^\gamma_\alpha$ exists (as an element of $\vec{M}^\gamma$), and
     %   \item $M^\gamma_\alpha \cap \gamma \subseteq a_\alpha$.
    %\end{itemize}
\end{enumerate}

\noindent There is a natural projection operator between the posets:

\begin{definition}
For every $\delta < \delta^*$ and $p \in \po_{\delta^*}$, define $p\uhr \delta = \la f^{p\uhr \delta},N^{p\uhr \delta}\ra$ by 
$$f^{p\uhr\delta} = \la f^p_\gamma : \gamma \in \supp(p) \cap \delta\ra, $$ 
and 
$$
N^{p\uhr\delta} = \{ (\alpha,a_\alpha \cap \delta) : (\alpha,a_\alpha) \in N^p \text{ and } a_\alpha \cap \delta \neq \emptyset\}.
$$
\end{definition}

\noindent It will follow from Inductive Assumption III below that if $\delta'<\delta$ and $p\in\po_\delta$, then $p\rest\delta'\in\po_{\delta'}$.

\begin{definition}${}$

\begin{enumerate}
    \item 
    Let $p \in \po_\delta$, $\delta' < \delta$ and $M^{\delta'}_\alpha \in \vec{M}^{\delta'}$. 
    We say that $M^{\delta'}_\alpha$ \emph{appears} in $p$ (at coordinate $\delta'$)
    if $\delta'  \in a_\alpha$ where $(\alpha,a_\alpha) \in N^p$.

\item We say that a condition $p \in \po_\delta$ is \emph{amenable} to a structure $M^\delta_\alpha \in \vec{M}^
\delta$ if for every $\gamma \in M^\delta_\alpha \cap \delta$, the model $M^{\gamma}_\alpha$ appears in $p$.

\item For $\delta<\kappa^+$, let $\vec{N}^\delta$ be the sequence that consists of $\Pi^1_1$-reflecting elementary substructures of the expanded structure $(\mathcal{A}^\delta,\vec{M}^\delta)$ from Definition \ref{Def:ReflectingSequenceofModels}. We say that $M^\delta_\alpha$ is \emph{$\vec{M}^\delta$-reflective} if $\alpha$ is in the domain of $\vec{N}^\delta$.

%We say that $M^\delta_\alpha$ is \emph{$\vec{M}^\delta$-reflective} if $M^{\delta}_\alpha = N^{\delta}_\alpha \cap \mathcal{A}^{\delta}$ for some $N^{\delta}_\alpha \elem (H_{\kappa^{+++}}, \mathcal{A}^{\delta}, \vec{M}^{\delta})$.

    \end{enumerate}
    \end{definition}
 
\begin{remark}
    A condition $p$ is amenable to $M^\delta_\alpha$ if and only if $\delta\cap M^\delta_\alpha\subseteq a_\alpha$ where $(\alpha,a_\alpha) \in N^p$. 
\end{remark}

\begin{remark}\label{Remark:ReflectiveStructuresAreLimits}${ }$

\begin{enumerate}
    \item The domain of the sequence $\vec{N}^\delta$ of $\vec{M}^\delta$-reflective structures is a subset of the domain of $\vec{M}^\delta$. Furthermore, $\delta\cap M^\delta_\alpha=\delta\cap N^\delta_\alpha$ for each $\alpha\in\dom(\vec{N}^\delta)$. This follows from the fact that both structures are Skolem hulls: if $\gamma\in\delta\cap N^\delta_\alpha$, then $\gamma$ is definable from the $<_{\kappa^++}$-least bijection $\psi:\delta\to\kappa$ and an ordinal below $\alpha$, which belong to both $N^\delta_\alpha$ and $M^\delta_\alpha$.
    \item Since the domain of the sequence $\vec{M}^\delta$ is unbounded in $\kappa$ we get that if $M^\delta_\alpha$ is $\vec{M}^\delta$-reflective then $|\vec{M}^\delta\uhr \alpha| = \alpha$. In particular, $M^\delta_\alpha$ is a limit point of the sequence $\vec{M}^\delta$.
\end{enumerate} 
\end{remark}

\noindent For a condition $p\in\po_\delta$ and $\delta'<\delta$, there are two associated sequences of side conditions:
\begin{definition}\label{de:EN} Let $p\in\po_\delta$ be a condition and ${\delta'}<\delta$.

    \begin{enumerate}
    \item $N^p_{\delta'}:=\{\alpha<\kappa:M^{\delta'}_\alpha\text{ appears in }p\}$,
    \item $E^p_{\delta'}:=\{\alpha<\kappa:p\rest{\delta'}\text{ is amenable to }M^{\delta'}_\alpha\}$.
\end{enumerate}
\end{definition}

\noindent If $M^{\delta'}_\alpha$ appears in $p$, then $p\rest{\delta'}$ is amenable to $M^{\delta'}_\alpha$, so $N^p_{\delta'}\subseteq E^p_{\delta'}$. Note that $E^p_{\delta}$ makes sense for $p\in\po_\delta$. It will hold that the generic sequence $N^G_{\delta'}=\bigcup_{p\in G}N^p_{\delta'}$ consists of limit points of the generic sequence $E^G_\delta=\bigcup_{p\in G}E^p_{\delta'}$, for every generic $G\subseteq\po_\delta$ and $\delta'<\delta$. See Inductive Assumption III below.

\begin{definition}[$\vec{M}^{\delta,\name{G}(\po_\delta)}$]
For each $\delta < \tau$ let $\vec{M}^{\delta,\name{G}(\po_\delta)}$ be the $\po_\delta$-name of the sub-sequence of $\vec{M}^\delta$ consisting of all $M^\delta_\alpha \in \vec{M}^\delta$ which are $\vec{M}^\delta$-reflective and for which there is an $M^\delta_\alpha$-amenable condition $p'$ in the (canonical $\po_{\delta}$-name) generic filter $\name{G}(\po_\delta)$ for $\po_\delta$.
\end{definition}

\noindent So the sequence $\vec{M}^{\delta,\name{G}(\po_\delta)}$ is a $\po_\delta$-name for a subsequence of $E^G_\delta$ from above.

%\begin{definition}${}$
%\begin{enumerate}    \item 
%For $M^\delta_\alpha \in \vec{M}^\delta$ and $q \in \po_{\delta}$, we say that $q$ is \emph{amenable to $M_\alpha^\delta$} (or $M_\alpha^\delta$-amenable) if there is $(\alpha,a_\alpha) \in N^{q}$ such that $M^\delta_\alpha \cap \delta \subseteq a_\alpha$.  

%\item Suppose that $\delta^* > \delta$. We say that $p \in \po_{\delta^*}$ is strongly amenable to $M^\delta_\alpha$ if there is $(\alpha,a_\alpha) \in N^p$ such that $\delta \in a_\alpha$. \end{enumerate} \end{definition}

\vspace{8pt}

\noindent\textbf{Inductive Assumption III: }
For all $\delta < \delta^* < \tau$ and $p \in \po_{\delta^*}$:
\begin{enumerate}
    \item $p\uhr \delta \in \po_{\delta}$.
    \item If $M^\delta_\alpha$ appears in $p$, then $p\rest\delta$ forces that $M^{\delta}_\alpha$ is a limit point of the sequence $\vec{M}^{\delta,\name{G}(\po_\delta)}$.
\end{enumerate}
In particular, if $M^\delta_\alpha$ appears in $p\in\po_{\delta^*}$, then $p \uhr \delta$ is amenable to $M^{\delta}_\alpha$.

\vspace{8pt}

\noindent\textbf{Inductive Assumption IV:}
For every $\delta < \tau$, $M^\delta_\alpha \in \vec{M}^\delta$ and $p \in \po_\delta \cap M^\delta_\alpha$, there is an extension $p' \leq p$ which is amenable to $M^\delta_\alpha$. 

\vspace{8pt}

\begin{remark}
    Suppose $\gamma<\delta$. The sequence $\vec{M}^{\gamma}$ is definable from $\po_\delta$ and $\gamma$, as the set of models that appear in conditions in $\po_{\delta}$ at coordinate $\gamma$. Thus if $M\prec\mathcal{A}^\delta$ and $\gamma\in\delta\cap M$, then the sequence $\vec{M}^{\gamma}$ is an element in $M$. In particular, if $M^\delta_\alpha$ belongs to the sequence $\vec{M}^\delta$ and $\gamma\in\delta\cap M^\delta_\alpha$, then by elementarity the domain of $\vec{M}^\gamma$ must be unbounded in $\alpha=\kappa\cap M^\delta_\alpha$. So $\alpha$ is a limit point of each sequence $\vec{M}^{\gamma}$, where $\gamma\in \delta\cap M^\delta_\alpha$. By the same reasoning, if $M^\delta_\alpha$ is $\vec{M}^\delta$-reflective, then in addition it must be a limit point of the sequence $\vec{M}^\delta$ as well.
\end{remark}
    
\begin{lemma}\label{Lemma:M-sequence-is-unbounded} Let $\delta<\tau$.
\begin{enumerate}
    \item $\vec{M}^{\delta,\name{G}(\po_\delta)}$ is (forced to be) of cardinality $\kappa$.
    \item If $p$ forces some $M^\delta_\alpha$ to be in $\vec{M}^{\delta,\name{G}(\po_\delta)}$ then it forces
$M^{\delta'}_\alpha$ to be in $\vec{M}^{\delta',\name{G}(\po_{\delta}')}$
for every $\delta' \in \delta \cap M^\delta_\alpha$.
\end{enumerate}
\end{lemma}
\begin{proof}
    The first part is an immediate consequence of Inductive assumption IV.
    For the second item, note that for every $M^\delta_\alpha \in \vec{M}^\delta$, $\psi_\delta \in M^\delta_\alpha$ and therefore $M^\delta_\alpha \cap \delta = \psi_\delta``\alpha$. It follows that for every 
 $a$ which is $\alpha$-closed (w.r.t $\vec{\psi}$) and $\delta \in a$, one must have $\delta \cap M^\delta_\alpha \subseteq a$.
 The statement of this part now follows from Inductive Assumption III and Remark \ref{Remark:ReflectiveStructuresAreLimits}.
\end{proof}

%Moreover, if $\delta' \in \supp(p)$ then $p'\uhr \delta'$ forces the following about the countable partial map $f^p_{\delta'} : \kappa \times \kappa \to T^{x^{\delta'}}$:
%\begin{enumerate}   \item $f^p_{\delta'}$ is order-preserving and level preserving  map from the wide tree $(S_{\delta'},<_{S_{\delta'}})$ to $(T^{h_\delta'},<_{T^{h_\delta'}})$.
    
    %\item For every $s_1\neq s_2 \in \dom(f^p_{\delta'})$ there is some $s \in \dom(f^p_{\delta'})$ which is the meet of $s_1,s_2$ in the tree order of $S_{\delta'}$, and $f^p_{\delta'}(s)$ is the meet in $T$ of   $f^p_{\delta'}(s_1)$ and $f^p_{\delta'}(s_2)$.
    
   % \item For every $(\alpha,a) \in N^p$ with $\delta' \in a$ and $s \in \dom(f^p_{\delta'}) \setminus M^{\delta'}_\alpha$ then the exit node $e =e_{S_{\delta'}}(s,M^{\delta'}_\alpha)$ belongs to   $\dom(f^p_{\delta'})$ and $f^p_{\delta'}(e) = e_{T}(f^p_{\delta'}(s),M^{\delta'}_\alpha)$.\end{enumerate}

The goal is to show that each poset $\po_\delta$, $\delta<\kappa^+$, is strongly proper with respect to the models in the sequence $\vec{M}^\delta$. This is done by showing that if $p\in\po_\delta\cap M^\delta_\alpha$, then there is an extension of $p$ which is amenable to $M^\delta_\alpha$, and whenever $p$ is amenable to $M^\delta_\alpha$, then it is a strong master condition for $M^\delta_\alpha$. Hence the following inductive assumption is natural.

%\begin{definition}\label{Def:ModelAppearsInSideCondition}    Let $p \in \po_\tau$, we say that a structure $M^\delta_\alpha \in \vec{M}^\delta$ \emph{appears} in $p$ if $(\alpha,a) \in N^p$ for some $a$ with $\delta \in a$. \end{definition}

\vspace{8pt}

\noindent\textbf{Inductive Assumption V:}
The following requirements hold for every $\delta' < \delta < \tau$ and $p \in \po_\delta$ 
with $\delta' \in \supp(p)$:
\begin{enumerate}
    \item For any $s_1, s_2 \in \dom(f^p_{\delta'})$ there is some $s \in \dom(f^p_{\delta'})$ which is forced by $p\uhr \delta'$ to be the meet of $s_1,s_2$ in the tree order of $S_{\delta'}$, and $f^p_0 \in \coll$ forces that $f^p_{\delta'}(s)$ is the meet in $T$ of 
    $f^p_{\delta'}(s_1)$ and $f^p_{\delta'}(s_2)$.
    \item $p\rest\delta'$ forces that $f^p_{\delta'}$ is level-preserving.
\end{enumerate}

\begin{remark}
    Inductive Assumption V implies $p\uhr \delta'$ forces that $f^p_{\delta'}$ is injective and order preserving. This follows from the simple observation that a partial function $f : S \to T$ between two trees $S,T$, whose domain is closed under meets in $S$, mapping those meets to the meets of the images, is order preserving.
\end{remark}

The next inductive assumptions describe the connection/restrictions between the ``working'' parts $f^p_\delta$ in conditions $p$ and the ``side condition'' parts $M^\delta_\alpha$ that appear in $p$. We start with a brief discussion before giving the precise details.

Our main goal is to secure a strong properness property for a structure $M^\delta_\alpha$ that appears in a condition $p$. As usual, strong properness will imply that the restriction of the generic embedding $f_\delta$ to $M^\delta_\alpha$ will be generic over $M^\delta_\alpha$.
It is therefore natural to include a restriction saying that for every $s \in \dom(f^p_\delta)$, $s \in M^\delta_\alpha$ if and only if  $f^p_\delta(s) \in M^\delta_\alpha$.
Now, to secure this property while allowing every node $s' \in S_\delta$ to be added to the domain of an extension, we impose a similar requirement for branches $b_s$ of nodes $s \in \dom(f^p_\delta)$. Namely, we require that for every $(s,t) \in  f^p_\delta$, $b_s \subseteq M^\delta_\alpha$ if and only if $b_t \subseteq M^\delta_\alpha$. We recall that by the construction of our trees $T$, for every $t \in T$, the identity of $\beta < \kappa$ for which $b_t \subseteq M^\delta_\beta$ is determined from the collapse index $(\alpha',\beta',\gamma',\delta')$ of $t$ (specifically, we need $\alpha',\beta' \leq \beta$). 

We point out that the last restriction introduces the following additional complication: Say,  $(s,t) \in f_\delta^p$ are outside of a structure $M^\delta_\alpha$ which appears in $p$. By our strong properness aspirations for $M^\delta_\alpha$ we would like our ability to extend the conditions $f^p_\delta \cap M^\delta_\alpha$ inside $M^\delta_\alpha$ to be ``independent'' from considerations outside of $M^\delta_\alpha$. Because of this, we are at a risk of adding a new structure $M^\delta_\beta$ to the side condition part, when working inside $M^\delta_\alpha$ that will violate the branch requirement coming from $(s,t)$, which are outside of the structure. 
One approach to avoid such a problem will be to add a notion of excluded intervals to the side condition part (such as when adding a club to $\omega_1$ with finite side conditions, as in Baumgartner, \cite{BaumgartnerClub}) of the poset to exclude ``problematic'' structures $M^\delta_\beta$ below $M^\delta_\alpha$.
We take here a similar and slightly more implicit approach to avoid this problem: We can use the function $f^p_\delta$ to produce excluded intervals, essentially by mapping a node $s' \in M^\delta_\beta$ to $f^p_\delta(s')$ outside of $M^\delta_\beta$, but inside $M^\delta_\alpha$ so that this restriction appears when we move to $f_\delta^p \cap M^\delta_\alpha$ and excludes adding $M^\delta_\beta$ to the side condition part. The last description can be seen as a motivation for the next definition, of a refined sub-sequence $\vec{M}^{\delta,G(\po_\delta),f_\delta} \subseteq \vec{M}^{\delta,G(\po_\delta)}$ given by a partial countable function $f_\delta$.

\begin{definition}\label{Def:M^f_delta}($\vec{M}^{\delta,\name{G}(\po_\delta),f_\delta}$)\\
Let $f_\delta : S_\delta \to T_\delta$ be a countable partial function. Define a $\po_\delta$-name for a sub-sequence $\vec{M}^{\delta,\name{G}(\po_\delta),f_\delta}$ of $\vec{M}^{\delta,\name{G}(\po_\delta)}$ by having $\vec{M}^{\delta,\name{G}(\po_\delta),f_\delta}$ consists of all $M^\delta_\beta \in \vec{M}^{\delta,\name{G}(\po_\delta)}$ such that for every $(s,t) \in f_\delta$
\begin{enumerate}  
        \item $s \in M^\delta_\beta$ if and only if $t \in M^\delta_\beta$, and
        \item $s$ is an exit node from $M^\delta_\beta$ if and only if $t$ is an exit node from $M^\delta_\beta$.
\end{enumerate}
\end{definition}

\begin{definition}[$\beta_{f_{\delta}}(s)$]\label{Def:beta(s)}
    For every condition $p \in \po_\delta$, node $s \in \Lev_{\alpha}(S_{\delta})$ where $\alpha < \kappa$, and  partial countable function $f_\delta : S_\delta \to T_\delta$ there is an extension $q \leq p$ which either decides the minimal $\beta < \kappa$ such that $\beta\geq\alpha$, $b_s \subseteq \alpha \times \beta$, $s \not \in M^\delta_\beta$, and $M^{\delta}_\beta \in \vec{M}^{\delta,\name{G}(\po_{\delta}),f_{\delta}}$ or else forces that such $\beta$ does not exist. If $\beta$ exists, we denote it by $\beta_{f_{\delta}}(s)$.
\end{definition}

\noindent Suppose that $p'\in\po_{\delta^*}$ is a condition for some $\delta^*>\delta$ and $p=p^*\rest\delta$ and $f_\delta=f^{p'}_\delta$. The requirement $M^{\delta}_\beta \in \vec{M}^{\delta,\name{G}(\po_{\delta}),f_{\delta}}$ implies that the model $M^\delta_\beta$ is in some sense ``addable" to the side conditions of $p'$, the fact that it belongs to $\vec{M}^{\delta,\name{G}(\po_{\delta})}$ implies that $p$ is amenable to it and the fact that it belongs to $\vec{M}^{\delta,\name{G}(\po_{\delta}),f_{\delta}}$ implies that it is closed under the function $f_\delta$.

\begin{remark}\label{Remark:HowToValidate}
    Given a condition $p \in \po_\delta$, $\delta' < \delta$, and a node $s \in S_{\delta'}$, we can extend $p\uhr \delta'$ to determine the minimal $\beta < \kappa$ such that $b_s \subseteq \beta \times \beta$ and $s \not\in \beta \times \beta$. This makes $\beta$ a natural candidate for $\beta_{f_{\delta'}}(s)$ but not necessarily the correct one. It might happen that either $M^{\delta'}_\beta$ does not exist, or it does, but is not forced to be closed under $f_\delta$, i.e. inside the sequence $\vec{M}^{\delta,G(\po_\delta),f_{\delta'}^p}$. We thus have to further extend $p\rest\delta'$ to decide $\beta_{f_{\delta'}}(s)$ to be some $\beta'\geq\beta$.
\end{remark}

\begin{definition}\label{Def:TheSequence-M^s_n}
For every $\delta < \tau$ and $s \in S_\delta$, we define a $\po_\delta$-name of a structure $M^s_1$ to be
$M^\delta_\beta$ for the minimal $\beta$ for which $M^\delta_\beta \in \vec{M}^{\delta,\name{G}(\po_\delta)}$  and $s \in M^\delta_\beta$.
\end{definition}

\noindent It follows that $p\in\po_\delta$ is amenable to $M^1_s$, whenever it decides it. 

\begin{remark}
    It follows from Inductive Assumption IV and the fact that the sequence $\vec{M}^{\delta,\dot G(\po_\delta)}$ is forced to be cofinal in $\kappa$, that the model $M^1_s$ always exists. The model $M^1_s$ is always (forced to be) a successor point of the sequence $\vec{M}^{\delta,\name{G}(\po_\delta)}$, being the least structure in $M^{\delta,\name{G}(\po_\delta)}$ that sees the node $s$. Since the sequence $\vec{M}^{\delta,\name{G}(\po_\delta)}$ is included in the limit points of the sequence $\vec{M}^\delta$, $M^1_s$ is a limit point in the sequence $\vec{M}^\delta$ and limit point in the sequences $\vec{M}^{\gamma,\dot G(\po_\gamma)}$, for $\gamma\in\delta\cap M^1_s$.
\end{remark}

\begin{definition}[$s$-knowledgeable conditions]\label{Def:s-know}
    We say that a condition $q\in \po_{\delta}$ is \emph{$s$-knowledgeable} (or knowledgeable about $s$) for some $s \in S_{\delta}$ with respect to a partial countable function $f_{\delta} : S_{\delta} \to T$,   if
    \begin{itemize}
        \item $q$ decides the identity of $M^s_1$, 
        \item $q$ decides if $\beta_{f_{\delta}}(s)$ exists, and if so, determines its value.
    \end{itemize}
\end{definition}

\begin{comment}  Let $p\in P_{\delta+1}$, $s\in \dom(f^p_\delta)$, and $p_\delta(s)=t$. Suppose $(\eta,\delta)\in N^p$. Let $M=M^{\delta+1}_\eta$. Suppose that $s$ is an exit node of $M$. Hence $t$ is an exit node from $M$.
   Let $\beta$ be minimal such that $b(s)\subseteq \alpha\times\beta$. Note that $\beta\leq\eta$. We require that the ``slice" of $t$ is $\beta$. 
    (We assume that $p\uhr \delta$ determined $\beta$.)
    
    Claim: If $M^\ast\in M$ such that $s$ is an exit node of $M^\ast$ then $t$ is an exit node from $M^\ast$.
    
 $\beta\subseteq M^\ast$, otherwise $s$ can not be an exit node from it because $b(s)$ is cofinal in $\alpha\times\beta$.  $t\notin M^\ast$ because $M^\ast\subseteq M$ and $t\not\in M$. $b(t)\subseteq \alpha\times\beta$, hence $b(t)\subseteq M^\ast $.
\end{comment}

%\vspace{8pt}

\noindent\textbf{Inductive Assumption VI:}
For every  $p \in \po_\delta$, $\delta'<\delta$, every $M^{\delta'}_\alpha$ that appears in $p$ and every $s \in \dom(f_{\delta'}^p)$:
    \begin{enumerate}
        \item $s \in M^{\delta'}_\alpha$ if and only if $f^p_{\delta'}(s)  \in M^{\delta'}_\alpha$,
        \item $s$ is an exit node from $M^{\delta'}_\alpha$ if and only if  $f^p_{\delta'}(s)$ is,
        \item if $s\in\dom(f^p_{\delta'})-M^{\delta'}_\alpha$, then there is $s'\in\dom(f^p_{\delta'})$ such that $p\rest\delta'$ forces that $s'\leq_{\dot S_{\delta'}}s$ and $s'$ is an exit node from $M^{\delta'}_\alpha$.
        %\item If there is $\beta < \kappa$ such that $(\beta,a_\beta) \in N^p$ and $\delta' \in a_\beta$, then we have $p\uhr\delta' \Vdash M^{\delta'}_\beta \in \vec{M}^{{\delta'},\name{G}(\po_{\delta'}),f^p_{\delta'}}$
    \item if $p\rest\delta'$ forces that $s$ is an exit node from $M^{\delta'}_\alpha$, then $p\uhr \delta'$ is $s$-knowledgeable with respect to $f_{\delta'}^p$, and if $p\uhr \delta'$ forces $\beta = \beta_{f^p_{\delta'}}(s)$ exists then $M^{\delta'}_\beta$ appears in $p$.
    \end{enumerate}

%\vspace{8pt}

\noindent\textbf{Inductive Assumption VII:}
For every $p \in \po_\delta$ and $s \in \dom(f^p_{\delta'})$ for some $\delta' < \delta$ in the support of $p$,  if $s$ is forced by $p\rest\delta'$ to be an exit node from a model $M^{\delta'}_\alpha$ that appears in $p$, then $f^p_{\delta'}(s)$ is also forced by $p\rest\delta'$ to be an exit node from $M^{\delta'}_\alpha$ and the collapse-index $(\alpha,\beta,\gamma,\epsilon)$ of $f^p_{\delta'}(s)$ is forced by $p\uhr \delta'$ to satisfy the following requirements:
     \begin{enumerate}
         \item $\alpha$ is the level of $s$,
         \item $\beta = \beta_{f^p_{\delta'}}(s)$ if  exists, and
         \item $\gamma = M^s_1 \cap \kappa$.
     \end{enumerate}

%\vspace{8pt}

%When $p$ validates a node $s \in S_{\delta'}$ it determines essential information around $s$, including (i) downwards information about a maximal exit structure $M^{\delta'}_{\beta_p(s)} \in \vec{M}^\delta$ for $s$ (i.e., maximal between those which contain $b_s$ but do not include $s$.) that will appear a generic filter $G(\po_{\delta})$, and (ii) upwards information about the first $\omega$-structures in $\vec{M}^\delta$ that see $s$.\\ The downwards information about $s$ will be required for showing that $\po_\delta \cap M^{\delta}_\alpha$ is a subforcing of $\po_\delta$ for suitable structures $M^\delta_\alpha$ that do not see $s$. Specifically, it plays an important role in the definition \ref{Def:Super-Nice} of super-nice conditions and in Lemma \ref{Lemma:D(M)isDense}.\\ The upwards information about $M^s_n$ will be key in determining how ``far away'' we need the image $t = f^p_{\delta'}(s) \in T$ of $s$ needs to be to secure the independence of $b_t$ from $b_s$. The next inductive assumption takes a step in unfolding this.

\begin{remark}\label{Rmk:SuccessorVSLimitsStructures}
    It follows form Inductive Assumption VII that whenever $s\in \dom(f^p_{\delta'})$ is an exit node from a model that appears in $p$ at coordinate $\delta'$, then $s\in M^1_s$ and $f^p_{\delta'}(s)\notin M^1_s$. This is because if it happened that $f^p_{\delta'}(s)\in M^1_s$, then by elementarity the collapse index of $t:=f^p_{\delta'}(s)$ would also be an element in $M^1_s$, which is not possible because $\gamma_t=M^1_s\cap\kappa$.
    At first glance this may seem to be at odds with the first clause of item 2 in Inductive Assumption VI, which requires that for every $M^{\delta'}_\alpha$ which appears in $p$, if $s$ belongs to $M^{\delta'}_\alpha$ then so does $f^p_{\delta'}(s)$. 
    The two requirements are compatible as $M^1_s$ does not appear in $p$ (in other words, if $M^1_s=M^{\delta'}_\alpha$, then $\alpha$ does not belong to $N^p_{\delta'}$, i.e.  $\delta'$ does not belong to $a_\alpha$), as $M^1_s$ is a successor structure in $\vec{M}^{\delta',G(\po_{\delta'})}$ (being the first to see $s$) while the structures $M^{\delta'}_\alpha$ that are allowed to appear in $p$ are required to be limit points of the sequence $\vec{M}^{\delta',\name{G}(\po_{\delta'})}$. Rather, again writing $M^1_s=M^{\delta'}_\alpha$, the structure $M^\gamma_\alpha$ appears in $p$ for every $\gamma\in\delta'\cap M^1_s$. So in the notation of Definition \ref{de:EN}, $\alpha\in E^p_{\delta'}-N^p_{\delta'}$.
\end{remark}

We are ready to define $\po_\tau$. We split the definition between the case $\tau$ is a limit ordinal, and $\tau = \delta+1$ is a successor ordinal. Recall that $\po_0$ was defined in Definition \ref{de:P0}.

\begin{definition}[$\po_\tau$ for a limit ordinal $\tau$]
Conditions $p \in \po_\tau$ are pairs  $\la f^p,N^p\ra$ that satisfy the following requirements:
\begin{itemize}
    \item $f^p = \la f^p_\delta : \delta \in \supp(p)\ra$ is a sequence with $\supp(p) \subseteq \tau$ countable and $0 \in \supp(p)$,
    
    \item $N^p$ is a countable set of pairs $(\alpha,a)$ satisfying $\alpha < \kappa$ and $a \in [\tau]^{\leq \alpha}$ nonempty and $\alpha$-closed with respect to $\vec{\psi}$,
    
    \item For every $\delta < \tau$, $p\uhr \delta$ belongs to $\po_\delta$.
\end{itemize}

\noindent A condition $p' \in \po_\tau$ extends $p$, denoted $p' \leq p$, iff $p'\uhr\delta \leq_{\po_\delta} p\uhr\delta$ for all $\delta< \tau$.
\end{definition}

\begin{definition}[$\po_\tau$ for a successor ordinal $\tau$]
Suppose that $\tau = \delta+1$. The poset $\po_{\tau} =\po_{\delta+1}$ consists of all pairs $p = \la f^p,N^p\ra$ which satisfy the following conditions:
\begin{itemize}
    \item  $f^p = \la f^p_\delta : \delta \in \supp(p)\ra$ is a sequence with $\supp(p) \subseteq \delta+1$ countable and $0 \in \supp(p)$,
    
    \item  $N^p$ is a countable set of pairs $(\alpha,a)$ satisfying $\alpha < \kappa$ and $a \in [\tau]^{\leq \alpha}$ nonempty and $\alpha$-closed with respect to $\vec{\psi}$, and whenever $\delta \in a$, $M^\delta_\alpha \in \vec{M}^\delta$ exists,
    
    \item $p\uhr \delta$ is a condition in $\po_\delta$,
     
    \item If $\delta \in \supp(p)$ then 
    $p\uhr\delta$ forces the statements from Inductive Assumptions V,VI, and VII where $\delta'$ and $\delta$ are replaced with $\delta$ and $\tau$.
    \item If $M^\delta_\alpha$ appears in $p$, then $M^{\delta}_\alpha$ is $\vec{M}^\delta$-reflective and $p \uhr \delta$ is amenable to $M^{\delta}_\alpha$.
\end{itemize}
A condition $p' \in \po_{\delta+1}$ extends $p$ if it satisfies the following requirements:
\begin{enumerate}
\item $p'\uhr\delta \leq_{\po_\delta} p\uhr\delta$,
\item $ f^p_\delta \subseteq f^{p'}_\delta$,
\item for every $(a,\alpha) \in N^p$ there is some $a' \supseteq a$ such that $(a',\alpha) \in N^{p'}$.\footnote{Equivalently, every $M^{\delta'}_\alpha$ that appears in $p$ appears in $p'$.}
\end{enumerate}

\end{definition}

\begin{lemma}
$\po_\tau$ satisfies inductive assumptions I-VII.
\end{lemma}
\begin{proof}${ }$\vspace{6pt}

\noindent\textbf{Assumption I}: It is clear that $|\po_\tau|$ has size $\kappa$ and is $\sigma$-closed. \\

\noindent
\textbf{Assumption II}: The structural assumptions of conditions $p \in \po_\tau$ are immediate consequences of the definition of $\po_\tau$.  \\

\noindent
\textbf{Assumption III}: Immediate from the definition of $\po_\tau$.\\

\noindent
\textbf{Assumption IV}: We need to show that if $p \in \po_\tau \cap M^\tau_\alpha$, there is an extension $p' \leq p$ which is amenable to $M^\tau_\alpha$.
Let $p' = \la f^{p'},N^{p'}\ra$ where $f^{p'} = f^p$ and 
$N^{p'} = N^p\cup\{(\alpha,a_\alpha)\}$, where $a_\alpha=\tau\cap M^{\tau}_\alpha$. Note that $p\in M^\tau_\alpha$ implies that for every $\delta<\tau$, $f^p_\delta\in M^\delta_\alpha$. This is because $f^p_\delta$ is a countable subset of $V_\alpha\subseteq M^\delta_\alpha$ and thus element of $M^\delta_\alpha$. Thus there are no exit nodes from $M^\delta_\alpha$ for any $\delta<\tau$ in the support of $p$ and Inductive Assumption VI about the side conditions is satisfied.
It follows that $p'$ is a condition.
It is clear that $p'$ is amenable to $M^\tau_\alpha$.\\

\noindent 
\textbf{Assumptions V,VI,VII:}
If $\tau$ is limit then IA V,VI,VII are immediate.
Suppose $\tau = \delta+1$ is a successor ordinal. The fact that V,VI, and VII hold at $\delta' < \delta$ follows from the fact $p\uhr \delta \in \po_{\delta}$. The fact that V,VI,VII hold at $\delta$ follows from the definition of $\po_{\delta+1}$.  %to check that for $p \in \po_\tau$ with $\delta \in \supp(p)$, items V.1-V.5 hold at the last coordinate $\delta$. 
\end{proof}

\begin{remark}\label{Remark:ImagesOfMeets}
Let $p \in \po_\tau$ be a condition. Suppose that $s' \in \dom(f^p_\delta)$ for some $\delta \in \supp(p)$, and $s \in S_\delta$ are such that 

\begin{itemize}
    \item $p\uhr \delta$ decides the meet of $s,s'$ in $S_\delta$, denoted by $m(s,s') \in S_\delta$,
    \item $m(s,s') \in  \Lev_{\bar{\alpha}}(S_\delta)$ for some $\bar{\alpha} < \kappa$,
    \item the collapse part $f^p_0$ of $p$  decides the identity of the unique node below $f^p_\delta(s')$ at level $\bar{\alpha}$ to be some $\bar{t} \in \Lev_{\bar{\alpha}}(T)$. 
\end{itemize}
The assumptions do not imply that
$m(s,s')$  belongs to $\dom(f^{p'}_\delta)$ as we do not assume $s$ does. However, the image $f^q_\delta(m(s,s'))$ of $m(s,s')$ in all possible extensions $q$ of $p$ with $m(s,s') \in \dom(f^q_\delta)$ is already decided by $p$ to be $\bar{t}$. We call $\bar{t}$ the \emph{implicit image} of the meet $m(s,s')$ of $s,s'$ as determined by $p$, and denote it by $t(s,s')$ or $t_p(s,s')$. Furthermore, if we add the pair $(m(s,s'),t(s,s'))$ to $f^p_\delta$, it is easy to see that the extension of $p$ obtained this way is a condition in $\po_\tau$.
\end{remark}

\subsection*{Traces and simple amalgamations}

We end this section with definitions of operations that attempt to capture small pieces of conditions $p \in \po_\tau$, and join those pieces together (possibly pieces from different conditions). 
These are given by \emph{traces} to structures, and \emph{simple amalgamations}, respectively.
The two operations need not produce conditions in general, but as we will show, they do  under quite natural additional assumptions. 

For a condition $p \in \po_\tau$ and a structure $M = M^\tau_\alpha \in \vec{M}^\tau$ the trace of $p$ to $M$, denoted $[p]_M$, is meant to capture all the information $p$ has which is relevant to $M$.
A priori, the definition makes sense for any set $M$.
\begin{definition}\label{Def:Trace}
    Let $p \in \po_\tau$ for some $\tau < \kappa^+$ and $M \elem H_{\kappa^{++}}$.
    The \emph{trace} of $p = \la f,N\ra$ to $M$, is the pair $[p]_M = \la \bar{f},\bar{N}\ra$ where
    \begin{itemize}
    \item $\dom(\bar{f}) = M \cap \dom(f)$, 
    \item for every $\gamma \in \dom(\bar{f})$, $\bar{f}_\gamma = f_\gamma \cap M$,
    \item $\bar{N}$ consists of pairs $(\alpha,\bar{a})$ for which there is some $a$ such that $(\alpha,a) \in N$, $\alpha \in M$, and $\bar{a} = a \cap M$ is nonempty.\footnote{Since $\vec{\psi} \in M$ and $a$ is $\alpha$-closed, then so is $(\alpha,\bar{a})$.}
    \end{itemize}
\end{definition}

\noindent It is not true in general that $[p]_M$ is a condition in $\po_\tau$. For example, it is possible that the trace part $[p]_M \uhr \delta$ does not decide the meet of nodes $s_1\neq s_2$ for $s_1,s_2 \in \dom(\bar{f}_\delta)$. Our main argument below shows that when $p$ is amenable to $M$, then it has an extension $p'\leq p$ such that $[p']_M$ is a condition in $\po_\tau \cap M$.
Indeed, in the next section we introduce a property of a condition $p$ being \textit{super-nice with respect to $M$} and prove that it implies that $[p]_M$ is a condition and a residue of $p$ into $M$ (see Definitions \ref{Def:Residue} and \ref{Def:Super-Nice}). We also show that whenever $p$ is amenable to $M$, then it has an extension that it super-nice with respect to $M$. 

Next, we define the simple amalgamation operation of two conditions $p,p'$. 
Given two conditions $p,p' \in \po_\tau$, the natural attempt to find a common extension $q$ to $p,p'$ involves taking coordinate-wise unions of the ``working parts''  and the ``side condition'' parts.
The result, which need not be a condition, 
is called the \emph{simple amalgamation} of $p,p'$. 
\begin{definition}[Simple Amalgamations]
The \emph{simple amalgamation} of conditions $p = \la f,N \ra$ and $p' = \la f',N'\ra$ is the pair $\la f^*,N^*\ra$ defined by 
\begin{itemize}
    \item $\dom(f^*) = \dom(f) \cup \dom(f')$
    \item for each $\delta \in \dom(f^*)$, $f^*_\delta = f_\delta \cup f'_\delta$ (with $f_\delta$ or $f'_\delta$ taken to be empty in the case $\delta\notin\dom(f)$ or $\delta\notin\dom(f')$, respectively)
    \item $(\alpha^*,a^*) \in N^*$  if and only if there is either $a$ so that $(\alpha^*,a) \in N$ or 
    $a'$ so that $(\alpha^*,a') \in N'$, and then $a^* = a \cup a'$ (with $a$ or $a'$ taken to be empty in the case $\alpha^*\notin\dom(N)$ or $\alpha^*\notin \dom(N')$, respectively)
\end{itemize}
\end{definition}

\noindent 
The simple amalgamation $(f^*,N^*)$ need not be a condition even when $p,p'$ are compatible. 
A reason why the simple amalgamation might fail to be a condition even if $p$ and $p'$ were compatible is that it might form a condition in $\po_\delta$ up to some coordinate $\delta < \tau$ 
(this is clearly the case for $\delta = 0$) but it need not decide the meets in $S_\delta$ of nodes $s \in \dom(f^p_\delta)$ with nodes  $s' \in \dom(f^{p'}_\delta)$, or the exit node of some $s' \in \dom(f^{p'}_\delta)$ with all side condition structures $M^\delta_\alpha$ that appear in $p$, and vice versa. 
However, it is clear that if a condition $q$ is a common extension of $p,p'$ then it extends their simple amalgamation, i.e. with the above notation,
$f^*_\delta \subseteq f^q_\delta$ for every $\delta \in \dom(f^*)$, and for every $(\alpha^*,a^*) \in N^*$ there is $b \supseteq a^*$ with $(\alpha^*,b) \in N^q$. 

%In fact, it is not difficult to see that these are the only obstructions for the simple amalgamation being a c ondition. 

In a very simple case where $p' \in \po_\delta$ extends an initial segment $p\uhr \delta$ of $p$, it is an immediate consequence of the definition that the simple amalgamation is a condition. 
\begin{lemma}\label{Lemma:SimpleAmalgamationEasy}
Suppose that $p \in \po_\tau$, and $p' \in \po_\delta$ extends $p\uhr \delta$ for some $\delta < \tau$. Then the simple amalgamation of $p,p'$ belongs to $\po_\tau$
\end{lemma}

\section{Strong Properness}\label{Section:StrongProperness}

In this section, we show that $\po_\tau$ is strongly proper with respect to the structures $\vec{M}^\tau$, for $\tau<\kappa^+$. We commence by recalling the usual definition of strong properness:

\begin{definition}\label{Def:StrongProperness}
    A condition $p \in \po$ is \emph{strongly proper} with respect to a model $M \elem (H_\theta, \in,\po)$ (for some sufficiently large $\theta$) if $p$ forces that $\name{G} \cap M$ is a $V$-generic filter for $\po \cap M$. 
    We say that $\po$ is \emph{strongly proper} with respect to a collection $\mathcal{T}$ of models $M$,  if for every $M \in \mathcal{T}$ and  $q \in M \cap \po$, $q$ has an extension $p \in \po$ which is strongly proper with respect to $M$. 
\end{definition}

This definition is equivalent to the definition of a \emph{strong master condition} in Neeman and Gilton \cite{NeemanGilton2016}.
We will use the following notion of residue function to prove strong properness results for our poset. We denote $\po/p=\{q\in\po\mid q\leq p\}$.

\begin{definition}\label{Def:Residue}
Let $p \in \po$ and $M \elem (H_\theta, \in,\po)$. A \emph{residue function} for $M$ over $p$ is a function $r : D \to M\cap \po$, where $D \subseteq \po/p$ is dense below $p$ and for every $q \in D$ and $w \in M \cap \po$ such that $w \leq r(q)$, $w$ and $q$ are compatible in $\po$. 
\end{definition}

\begin{lemma}\label{Lem:ResidueImpliesStrongProper}
If $p$ and $M$ are as above, and $r : D \to \po \cap M$ is a residue function for $M$ over $p$, then $p$ is strongly proper with respect to $M$.
\end{lemma}
\noindent See Proposition 1.7 in \cite{NeemanGilton2016} for details.

The goal is to show that if $p\in\po_\tau$ is amenable to $M^\tau_\alpha$, then $p$ is a strong master condition for $M^\tau_\alpha$, for models $M^\delta_\alpha$ that are $\vec{M}^\delta$-reflective. Our residue function to $M$ will be given by the trace $[p]_M$ operator (see Definition \ref{Def:Trace}). 
We introduce the notion of a super-nice conditions with respect to a structure $M$, and prove that super-nice conditions are dense below conditions that are amenable to $M$, and when $p$ is super-nice with respect to $M$ then $[p]_M$ is a residue of $p$ in $M$.  

\begin{definition}\label{Def:Super-Nice}${}$
\begin{enumerate}
\item We say that $p \in \po_\tau$  \emph{nicely projects} to a structure $M^\tau_\beta$ if it satisfies the following requirements:
\begin{itemize}
    %\item For every $\delta \in M^\tau_\beta \cap \supp(p)$ and $s' \in \dom(f^p_\delta)$ which is outside of $M^\delta_\beta$ there exists some $s \in \dom(f^p_\delta)$ such that $p\uhr \delta$ forces $s$ to be an exit node from $M^\delta_\beta$ and $s \leq_{S_\delta} s'$. 

    \item $p$ is amenable to $M^\tau_\beta$
    
    \item  $f^p_0 \in D(t)$ for every $t \in \bigcup_{\delta \in \supp(p)}\rng(f^p_\delta)-M^\tau_\beta$ (see Definition 
    \ref{Def:D(t)} for the $\sigma$-closed dense set $D(t)$)
    
    \item For every $\delta \in M^\tau_\beta \cap \supp(p)$ and $s \in \dom(f^p_\delta)$ which is forced by $p\uhr \delta$ to be an exit node from $M^\delta_\beta$,  then the node projection $\bar{t} = \pi_{f^{p}_0}(f^p_\delta(s))$ %is of sufficiently high level so that it does not belong to $M^\delta_{\alpha}$ for any $\alpha < \beta_p(s)$ (see Definition \ref{Def:beta(s)} for $\beta_p(s)$), and 
    satisfies that either 
\begin{enumerate}
\item ($\bar{t}$ has a preimage) \\
there is some $\bar{s} \in \dom(f^p_\delta) \cap M^\tau_\beta$ such that $p\uhr \delta \Vdash \bar{s} <_{S_\delta} s$ and
$f^p_\delta(\bar{s}) = \bar{t}$, or

\item ($b_{\bar{t}}$ has cofinal preimages)\\
there is a sequence of nodes $\la \bar{s}_n\ra_n \subseteq \dom(f^p_\delta) \cap M^\tau_\beta$ such that $p\uhr \delta$ forces it is increasing in $S_\delta$, bounded by $s$, and the images by $f^p_\delta$ are cofinal in $\bar{t}$.
\end{enumerate}
And moreover, if the height of $s$ is a successor-ordinal, then the immediate predecessor of $s$ belongs to $\dom(f^p_\delta)$.
\end{itemize}

\item 
Let $p \in \po_\tau$ be a condition which satisfies that for every $\delta \in \supp(p)\setminus 1$ and $s \in \dom(f^p_\delta)$, the condition $p\uhr \delta$ decides the identity of $M^1_s$. In particular, $p\uhr \delta$ is $M^1_s$-amenable.
We define $E(p,M^\tau_\alpha)$ to be the countable set of pairs $(M^\delta_\beta,\delta)$ obtained by starting with the singleton $\{(M^\tau_\alpha,\tau)\}$ and closing under the following operation:
given $(M^\delta_\beta,\delta)$ in our set and $\gamma\in\supp(p)\cap\delta\cap M^\delta_\beta$, if $s\in\dom(f^p_\gamma)$ is forced by $p\uhr\delta$ to be an exit node from $M^\gamma_\beta$, then we add $(M^1_s,\gamma)$.

\item  $p$ is \emph{super-nice with respect to $M^\tau_\alpha$}  
if it decides $M^1_s$ for every $s\in\dom(f^p_\delta)$, $\delta\in\supp(p)$, and $p\uhr \delta$ nicely projects to $M^\delta_\beta$ 
for every $(M^\delta_\beta,\delta) \in E(p,M^\tau_\alpha)$.
The set of super-nice conditions $p \in \po_\tau$ with respect to $M^\tau_\alpha$ is denoted by $D_\tau(M^\tau_\alpha)$.
\end{enumerate}
\end{definition}

\begin{remark}
It follows from the definition of $E(p,M^\tau_\alpha)$ that for every pair $(M^\delta_\beta,\delta) \in E(p,M^\tau_\alpha)$ that is not the initial one $(\tau,M^\tau_\alpha)$, the structure $M^\delta_\beta$ is forced by $p\uhr \delta$ to have form $M^1_s$ for some $s\in\dom(f^p_\delta)$ and thus to be a successor structure in the sequence $\vec{M}^{\delta,\name{G}(\po_\delta)}$. 
This observation together with the fact that if $(M^\delta_\beta,\delta)\in E(p,M^\tau_\alpha)$ and  $\delta' \in \delta\cap M^\delta_\beta$, then $M^{\delta'}_\beta$ is a limit structure in the sequence 
$\vec{M}^{\delta',\name{G}(\po_{\delta'})}$, imply that for all $(M^\delta_\beta,\delta),(M^{\delta'}_{\beta'}, \delta')\in E(p,M^\tau_\alpha)$, if $\delta'\in\delta\cap M^\delta_\beta$, then $\beta \neq \beta'$. 
\end{remark}

\begin{remark}
    It also follows that if $(\delta,M^\delta_\alpha)\in E(p,M^\tau_\alpha)$, then $E(p\rest\delta,M^\delta_\beta)\subseteq E(p,M^\tau_\alpha)$.
\end{remark}

\noindent The set $E(p,M^\tau_\alpha)$ can be understood through the concept of simple ``paths":

\begin{definition}\label{de:path}
    Let $\gamma<\tau<\kappa^+$ and $\alpha<\beta<\kappa$. A \emph{path from $M^\tau_\alpha$ to $M^\gamma_\beta$} is a finite sequence $\la (M^{\gamma_i}_{\beta_i},\gamma_i)\mid {i\leq n}\ra$ such that $(M^{\gamma_0}_{\beta_0},\gamma_0)=(M^\tau_\alpha,\tau)$ and $(M^{\gamma_n}_{\beta_n},\gamma_n)=(M^\gamma_\beta,\gamma)$ and for every $k<n$:
    \begin{enumerate}
        \item $\gamma_{k+1}\in\gamma_k\cap M^{\gamma_k}_{\beta_k}$,
        \item $\beta_{k+1}>\beta_k$.
    \end{enumerate}
    A path is said to be \textit{in $p$}, for $p\in \po_\tau$, if $\gamma_k\in\supp(p)\cup\{\tau\}$ and $\beta_k\in E^p_{\gamma_k}$, using notation from \ref{de:EN}.
\end{definition}
\noindent The set $E(p,M^\tau_\alpha)$ is included in the closure of the singleton $\{(M^\tau_\alpha,\tau)\}$ under paths from $M^\tau_\alpha$.

Building on Lemma \ref{Lemma:FlexibleProjections2}, we obtain a lemma according to which, given a condition $p$ that is super-nice with respect to $M^\tau_\alpha$, it is possible to extend $p$ to modify the information about the branches below certain nodes in $T$ while preserving super-niceness of $p$.

\begin{lemma}\label{lem:freem}
    Let $\tau<\kappa^+$. Suppose that $p\in\po_\tau$ is super-nice with respect to $M^\tau_\alpha$ and $t_n$, $n<\omega$, are countably nodes in $T$ with $\gamma_{t_n}\leq\alpha$ for each $n<\omega$, and $\bar{t}_n$, $n<\omega$, are countably many nodes in $T\cap V_\alpha$ such that $\bar{t}_n\in\alpha_n\times\beta_n$ (See Convention \ref{conv:collapseindex} for $\alpha_{t_n}$, $\beta_{t_n}$).
    Then for any $\sigma:\omega\to 2$ there is a condition $p'\leq p$ such that $p'$ is super-nice with respect to $M^\tau_\alpha$, $[p']_{M^\tau_\alpha}=[p]_{M^\tau_\alpha}$ and any extension of $p'$ that forces $\dot\pi_p(t_n)\leq\bar{t}_n$ must also force
    \[
    \bar{t}_n<^{\sigma(i)}_{\dot T}t_n,
    \]
    where $<^0_{\dot T}$ is $\not<_{\dot T}$ and $<^1_{\dot T}$ is $<_{\dot T}$, for every $n<\omega$. Furthermore, the extension $p'\leq p$ is minimal in the sense that only the coordinate $f^p_0$ was extended, and for any $t\in T-(V_\alpha\cup\{t_n:n<\omega\})$,
    \[
    \dot\pi^{p'}(t)=\dot\pi^p(t).
    \]
\end{lemma}
\begin{proof}
    By Lemma \ref{Lemma:FlexibleProjections2} there is a collapse condition $g\in \col(\omega_1,<\kappa)$ that extends $f^p_0$, forces the wanted conclusion, and satisfies the minimality condition that $g\cap V_\alpha=f^p_0\cap V_\alpha$ and for any $t\in T-(V_\alpha\cup\{t_n:n<\omega\})$, 
    \[
    \dot\pi^g(t)=\dot\pi^{f^p_0}(t).
    \]
    It exists by the definition of the tree $\dot T$. Let $p'$ be obtained from $p$ by replacing $f^p_0$ by $g$. 
    
    Now clearly we still have $[p]_{M^\tau_\alpha}=[p']_{M^\tau_\alpha}$. We claim that $p'$ is super-nice with respect to $M^\tau_\alpha$.

    To show super-niceness, suppose towards contradiction that $p'$ is not super-nice with respect to $M^\tau_\alpha$. Since $p$ is super-nice with respect to $M^\tau_\alpha$, this implies that we modified the projection below a node that is relevant for super-niceness. Specifically, there is a path
    \[
    (M^{\delta_0}_{\beta_0},\delta_0),\dots,(M^{\delta_n}_{\beta_n},\delta_n)
    \]
    in $p'$ such that $\delta_0=\tau$ and $\beta_0=\alpha$, and such that for some $\xi\in\delta_n\cap M^{\delta_n}_{\beta_n}$ and pair of exit nodes $(s,t)\in f^p_\xi$ from $M^{\xi}_{\beta_n}$, we have
    \[
    \dot\pi^{p'}(t)\neq\dot\pi^p(t).
    \]
    Note that the path is also in $p$ since only the collapse condition was extended. By the minimality consition, we have $t=t_n$ for some $n$. In particular, $\gamma_{t}=\gamma_{t_n}\leq\alpha$. By definition of the poset, we have $s\in V_{\gamma_t}$. Then: $s\in M^1_s\cap V_\kappa= V_{\gamma_t}\subseteq V_\alpha= V_{\beta_0}\subseteq V_{\beta_n}$. This contradicts the assumption that $s$ was exit node from $M^\xi_{\beta_n}$. Hence $p'$ must be super-nice with respect to $M^\tau_\alpha$.

\end{proof}

\begin{lemma}\label{Lemma:DensityMainLemma}
    Let $\tau<\kappa^+$. Suppose that for every $\delta < \tau$ and for every $M^\delta_\alpha$ that is $\vec{M}^\delta$-reflective, the following holds:
    \begin{itemize}
        \item The set $D_\delta(M^\delta_\alpha)$ is $\sigma$-closed and dense below every condition amenable to $M^\delta_\alpha$.
        \item For every $p \in D_\delta(M^\delta_\alpha)$, the trace $[p]_{M^\delta_\alpha}$ is a condition in $\po_\delta \cap M^\delta_\alpha$ and a residue of $p$ into $M^\delta_\alpha$.
    \end{itemize}
    Then:
    \begin{enumerate}
        \item \emph{(Node density)} For every $p \in \po_\tau$, $\delta < \tau$,  and $s \in S_\delta$, there is an extension $p' \leq p$ with $s \in \dom(f^p_\delta)$.
        \item \emph{(Super-nice density)} For every $M^\tau_\alpha \in \vec{M}^\tau$, the set $D_\tau(M^\tau_\alpha)$ is $\sigma$-closed and dense below every condition that is amenable to $M^\tau_\alpha$.
        \item For every condition $p \in D_\tau(M^\tau_\alpha)$, the trace $[p]_{M^\tau_\alpha}$ is a condition in $\po_\tau \cap M^\tau_\alpha$. 
    \end{enumerate}
\end{lemma}
\begin{proof}
Starting with the node density assertion,  let $p \in \po_\tau$ and $s \in S_\delta$, $\delta < \tau$. We would like to extend $p$ and add $s$ to $\dom(f^p_\delta)$. 
Up to extending $p\rest\delta$, we may assume that it decides the following information:
\begin{itemize}
    \item Meets in the set $\dom(f^p_\delta)\cup\{s\}$ and their implicit images (see Remark \ref{Remark:ImagesOfMeets}).
    \item For every $\alpha\in N^p_\delta$ with $s\notin M^\delta_\alpha$, the node $\bar{s}_\alpha$ that is below $s$ and an exit node from $M^\delta_\alpha$ as well as the identity of $M^1_{\bar{s}_\alpha}\in\vec{M}^{\delta,\dot G(\po_\delta)}$.
\end{itemize}
By Remark \ref{Remark:ImagesOfMeets}, we may assume that the meet $s'\wedge s$ is already in the domain of $f^p_\delta$ for every $s'\in\dom(f^p_\delta)$. We need to find images $t$ and $\bar{t}_\alpha$ for the nodes $s$ and $\bar{s}_\alpha$, respectively, where $\alpha\in N^p_\delta$. We need to make sure that for every $\alpha\in N^p_\delta$ and $\beta\in N^p_\delta$ with $\beta>\alpha$, 
\[
\bar{s}_\alpha\in M^\delta_\beta\iff\bar{t}_\alpha\in M^\delta_\beta,
\]
and $s\in M^\delta_\beta$ iff $t\in M^\delta_\beta$. 

\begin{claim}\label{claim:M1s}
    For all $\beta\in N^p_\delta$ and $\alpha\in N^p_\delta$ with $\alpha<\beta$, if $\bar{s}_\alpha\in M^\delta_\beta$, then $M^1_{\bar{s}_\alpha}\in M^\delta_\beta$.
\end{claim}
\begin{proof}[Proof of Claim \ref{claim:M1s}]
    Let $\bar{s}_\alpha\in X$ and let $\beta\in N^p_\delta$ be such that $\bar{s}_\alpha\in M^\delta_\beta$. Say $p\rest\delta$ decides $M^1_{\bar{s}_\alpha}$ to be the model $M\in \vec{M}^{\delta,\dot G(\po_\delta)}$. Up to extending $p\rest\delta$ further, we assume by induction that it is super-nice with respect to $M^\delta_\beta$. Therefore, by assumption its trace $[p\rest\delta]_{M^\delta_\beta}$ is a residue of $p\rest\delta$ into $M^\delta_\beta$. Since $\bar{s}_\alpha\in M^\delta_\beta$, by elementarity and by the fact that $M^\delta_\beta$ is reflective, we find $w\in\po_\delta\cap M^\delta_\beta$ extending the trace $[p\rest\delta]_{M^\delta_\beta}$ that decides $M^1_{\bar{s}_\alpha}$, say to be the model $M'\in\vec{M}^{\delta,\dot G(\po_\delta)}$. But then $w$ and $p\rest\delta$ are compatible, so we must have $M=M'\in M^\delta_\beta$.
\end{proof}

\noindent Next, let $\bar{t}_\alpha$, $\alpha\in N^p_\delta$, be nodes of collapse index $(\alpha_{\bar{t}_\alpha},\beta_{\bar{t}_\alpha},\gamma_{\bar{t}_\alpha},\delta_{\bar{t}_\alpha})$ that satisfy:
\begin{itemize}
    \item the collapse part $f^p_0$ does not decide anything about them, i.e. \[(\alpha_{\bar{t}_\alpha},\beta_{\bar{t}_\alpha},\gamma_{\bar{t}_\alpha},\delta_{\bar{t}_\alpha})\notin\dom(f^p_0),\]
    \item $\alpha_t=\height(\bar{s}_\alpha)$,
    \item and $\beta_t=\beta_p(\bar{s}_\alpha)$ if $\beta_p(\bar{s}_\alpha)$ exists,
    \item $\gamma_{\bar{t}_\alpha}=\kappa\cap M^1_{\bar{s}_\alpha}$,
    \item $\delta_{\bar{t}_\alpha}<\beta$ for every $\beta$ such that $\bar{s}_{\alpha}\in M^\delta_\beta$,
    \item if $\bar{t}_\alpha\neq\bar{t}_{\alpha'}$ and $\alpha<\alpha'$, then $\bar{t}_\alpha\in\alpha_{\bar{t}_{\alpha'}}\times\beta_{\bar{t}_{\alpha'}}$.
\end{itemize}
Furthermore, let $t$ be a node of collapse index $(\alpha_t,\beta_t,\gamma_t,\delta_t)$ which the collapse part $f^p_0$ does not decide anything and such that $\bar{t}_\alpha\in \alpha_t\times\beta_t$ for every $\alpha\in N^p_\delta$, and if $s\in M^\delta_\alpha$, then $t\in M^\delta_\alpha$.
There is an extension of the collapse part $f\leq f^p_0$ that decides the following:
\begin{itemize}
    \item $\bar{t}_\alpha<_{\dot T}t$,
    \item for every $s'\in\dom(f^p_\delta)$, $f^p_\delta(s')\wedge t=f^p_{\delta}(s\wedge s')$.
\end{itemize}
Let $p'$ be the extension of $p$ obtained by replacing the collapse part $f^p_0$ by $f$ and by letting
\[
f^{p'}_\delta:=f^p_\delta\cup\{(\bar{s}_\alpha,\bar{t}_\alpha):\alpha\in N^p_\delta,s\notin M^\delta_\alpha\}\cup\{(s,t)\}.
\]
Then, by construction, $p'$ is a condition in $\po_\tau$, it extends $p$, and satisfies $s\in\dom(f^{p'}_\delta)$. This ends the Node density assertion.

\vspace{8pt}

Next, we prove the second assertion of $D_\tau(M^\tau_\alpha)$ being dense below every condition $p$ which is amenable to $M^\tau_\alpha$.
Since $\po_\tau$ is $\sigma$-closed one can iterate $\omega$ many times the Node density assertion, by using a standard bookkeeping argument, to construct an extension $p'\leq p$ such that for every $\delta<\tau$:
\begin{itemize}
\item $f^p_0 \in D(t)$ for all $t\in \rng(f^{p'}_\delta)$,
\item for every structure $M^\delta_\beta$ that appears in $p'$, and pair $(s,t) \in f^{p'}_\delta$ that are forced by $p'\rest\delta$ to be exit nodes from $M^\delta_\beta$, $p'\uhr \delta$ decides a countable sequence of predecessors of $s$, $\la\bar{s}_n\mid n<\omega\ra$, whose height are cofinal in the height of $\dot\pi_{p'}(t)$, and that are in the domain of $f^p_\delta$. They then satisfy that the sequence $\la f^p_\delta(s_n)\mid n<\omega\ra$ is cofinal in $\dot\pi_{p'}(t)$.
\end{itemize}
These two items take care of the super-niceness of $p'$ with respect to $M^\tau_\alpha$, \ref{Def:Super-Nice}. (In fact, they take care of $p'$ being super-nice with respect to all structures with respect to which $p'$ is amenable.)

To achieve this, consider a single step of the construction. Let $(s,t) \in f^p_\delta$ be exit nodes from some $M^\delta_\beta$ that appears in $p$. 
We define a decreasing sequence of conditions $\la q^n \mid n < \omega\ra$ below $q^0 = p\uhr \delta$, whose collapse parts are all in $D(t)$. 
Extend first $q^0$ to $q'$ to decide the node $\bar{t}_0 = \pi_{q^0}(t)$ and denote its level by $\bar{\alpha}_0$. Extend then $q^1\leq q'$ to decide the unique predecessor of $s$ at the height of $\bar{t}_0$, call it $\bar{s}_0$. This $\bar{s}_0$ is the only possible candidate for the pre-image of $\bar{t}_0$. Indeed, any extension of $q^1$ and $p$ in the poset $\po_\tau$ that has $\bar{s}_0$ in the domain of the $\delta$-th embedding must map $\bar{s}_0$ to $\bar{t}_0$. Continue like this: Extend $q^1$ to some $q''$ that decides the node $\bar{t}_1:=\pi_{q_1}(t)$ and find $q^2\leq q''$ that decides the predecessor of $s$ at the height of $\bar{t}_1$, call it $\bar{s}_1$. Continue $\omega$ many times. Finally, let $q^\omega$ extend each $n$. Then let $p_1$ be the simple amalgamation of $q^\omega$ and $p$. It is a condition in $\po_\tau$ by Lemma \ref{Lemma:SimpleAmalgamationEasy}. By enumerating all relevant pairs of nodes $(s,t)$ in the supports of the $p_n$, using a bookkeeping function, we finally obtain a condition in $\po_\tau$ that satisfies the above two bulletpoints and is thus super-nice with respect to $M^\tau_\alpha$. 

\vspace{8pt}

We move to the third assertion of the lemma. Given $p \in D_\tau(M^\tau_\alpha)$ we would like to show that $[p]_{M^\tau_\alpha}$ is a condition  in $\po_\tau$.
Having that $[p\uhr \delta]_{M^\delta_\alpha} \in \po_\delta$ for every $\delta \in M^\tau_\alpha \cap \tau$, it suffices to verify the assertion for the case $\tau = \delta+1$ is a successor ordinal.
In this case, one has to verify  that $[p]_{M^{\delta+1}_\alpha}$ satisfies the requirements in assumptions V,VI, and VII regarding decisions for pairs $(s,t)\in f^p_\delta \cap M^{\delta+1}_\alpha$ and structures $M^{\delta}_\beta$, $\beta < \alpha$ that appear in $p$. 
Having $p \in D_{\delta}(M^{\delta}_\alpha)$ implies $p\uhr \delta \in D_\delta(M^\delta_\alpha)$, which by our assumption, implies that $\po_\delta \cap M^{\delta+1}_\alpha$ is a regular sub-forcing of $\po_\delta/(p\uhr\delta)$. It follows that all decisions about nodes $s \in \dom(f_p^\delta) \cap M^{\delta}_\alpha$, including (i) the meets, (ii) their exit nodes from every $M^\delta_\beta$ that appears in $p$, (iii) the identity of $M^s_1$, and (iv) $\beta_p(s)$, which are determined by $p\uhr \delta$, must already be determined by the trace $[p\uhr\delta]_{M^\delta_\alpha}$. Otherwise, there will be an extension $w \in \po_\delta \cap M^{\delta+1}_\alpha$ which gives incompatible information. But such a condition $w$ could not be compatible with $p\uhr \delta$, contradicting the inductive assumption of the lemma for $\delta$.
\end{proof}

\noindent Our goal is to show the following:

\begin{proposition}\label{Proposition:MainStrongProper}
Suppose $M^\tau_\alpha \in \vec{M}^\tau$ and $q \in \po_\tau$ is amenable to $M^\tau_\alpha$.
The function which takes $p \in D_\tau(M^\tau_\alpha)$ to $[p]_{M^\tau_\alpha}$ is a residue function to $M^\tau_\alpha$ over $q$. 
\end{proposition}

\noindent Before proving Proposition~\ref{Proposition:MainStrongProper} we give its immediate corollary and make some preparations.

\begin{corollary} Let $\tau<\kappa^+$ and $\alpha<\kappa$.
\begin{enumerate}
    \item Every condition $q \in \po_\tau$ which is amenable to $M^\tau_\alpha$ is strongly proper with respect to $M^\tau_\alpha$. In particular, $\po_\tau \cap M^\tau_\alpha$ is a regular subforcing of $\po_\tau/q$.
    
    \item For every condition $q \in D_\tau(M^\tau_\alpha)$, $[q]_{M^\tau_\alpha}$ forces that $q$ belongs to the quotient forcing $\po_\tau/(\po_\tau \cap M^\tau_\alpha)$. 
\end{enumerate}
 
\end{corollary}

Given a condition $p \in D_\tau(M^\tau_\alpha)$ as in the statement of Proposition \ref{Proposition:MainStrongProper}, and an extension $w \leq [p]_{M^\tau_\alpha}$ in $\po_\tau\cap M^\tau_\alpha$, the goal of the proof is to find a common extension of $w$ and $p$. Taking the simple amalgamation of $w$ and $p$ will not work in general. The main part of the argument is based on an inductive construction that results in a common extension $p'$ of $p$ and $w$.

%To prove Proposition \ref{Proposition:MainStrongProper} is to start frim $based on we use a somewhat technical generalization of the definition of super-nice conditions and strong properness, defined relative to a countable collection of structures, rather than a single one. 

\begin{definition}${}$
\begin{enumerate}
\item Let $\tau < \kappa^+$. A \textit{$\tau$-sequence} is a sequence of the form $\la (M_{\alpha_i}^{\tau_i},\tau_i)\mid i<\nu\ra$ where $\nu$ is a countable ordinal and:
\begin{enumerate}
    \item if $i<j$, then $\alpha_i\leq\alpha_j$, %\todo{$\alpha_i\leq\alpha_j$}
    \item if $i<\nu$, then $\tau_i\leq \tau$,
    \item for every $i>0$ there is $j<i$ such that $\tau_i\leq \tau_j$ 
      and  $\tau_i\in M^{\tau_j}_{\alpha_j}$.
      %\item for every $0 \leq j<i$ such that $\tau_i<\tau_j$ then  $\tau_i\in M^{\tau_j}_{\alpha_j}$
\end{enumerate}
\item The \emph{root} of a $\tau$-sequence $\la(M^{\tau_i}_{\alpha_i},\tau_i)\mid i<\nu\ra$ is the model $M^{\tau_0}_{\alpha_0}$.

\item Suppose that $\la (M^{\tau_i}_{\alpha_i},\tau_i) \mid i < \nu\ra$ is a $\tau$-sequence and let $p \in \po_\tau$. We say that the sequence is \emph{closed to $p$} if $M^{\tau_i}_{\alpha_i}$
appears in $p$,
$p\uhr \tau_i \in D_{\tau_i}(M^{\tau_i}_{\alpha_i})$, and $E(p\uhr\tau_i,M^{\tau_i}_{\alpha_i})$ is contained in $\la (M^{\tau_i}_{\alpha_i},\tau_i) \mid i < \nu\ra$, for all $i < \nu$.

\item  Suppose that $\la (M^{\tau_i}_{\alpha_i},\tau_i) \mid i < \nu\ra$ is a $\tau$-sequence and let $p \in \po_\tau$. The \emph{$p$-closure} of the sequence is the natural extension obtained by adding the pairs from $E(p\uhr\tau_i,M^{\tau_i}_{\alpha_i})$ for each $i < \nu$. %It is clear from the definition that the extended sequence remains a $\tau$-sequence.
\end{enumerate}
\end{definition}

\begin{remark}
It is clear from the definitions that the $p$-closure of a $\tau$-sequence remains a $\tau$-sequence (with the same root).    
\end{remark}

The main combinatorial feature we will need from $\tau$-sequence is given in the following Lemma. 

\begin{lemma}\label{LEM:tau-sequenceProperty}
    Let 
    $\la (M_{\alpha_i}^{\tau_i} ,\tau_i):i<\nu\ra$ be a $\tau$-sequence. For every $i < j$, if $\tau_i 
    \leq \tau_j$ then $\tau_i \in M^{\tau_j}_{\alpha_j}$.
\end{lemma}

\noindent To prove the Lemma we make use of the idea that a $\tau$-sequence can be described in terms of the simple paths from Definition \ref{de:path}. 

\noindent We shall prove a claim :

\begin{claim}\label{claim:tausq} Let  $\la (M_{\alpha_i}^{\tau_i} ,\tau_i):i<\nu\ra$ be a $\tau$-sequence of models. For all $i<j<\nu$, 
$$ M^{\tau_i}_{\alpha_i}\cap(\min(\tau_i,\tau_j)+1)\subseteq M^{\tau_j}_{\alpha_j}.$$

\end{claim}

\begin{proof}[Proof of Claim \ref{claim:tausq}]
For $\delta<\kappa^+$ let $f_\delta:\delta\rightarrow\kappa$ be an injection of $\delta$ into $\kappa $ which is minimal in the well ordering of $H_{\kappa^+}$. Obviously $f_\delta\in M^\rho_\mu$ if $\delta\in M^\rho_\mu$.

We use the fact that the rooted sequence of models is a union of paths that start from $(M^{\tau}_\alpha,\tau)$.

We show first that the claim holds for any two nodes of our sequence that lie on the same path.
Let $$(M^{\gamma_i}_{\beta_i},\gamma_i)_{i\leq n}$$ be a path. 
We show by induction on $n$ that for all $i < j$ along the path, 
$M^{\gamma_i}_{\beta_i} \cap \gamma_j \subseteq M^{\gamma_j}_{\beta_j}$.
Suppose that claim holds for $n$ and prove for $n +1$. It suffices to check that 
for every $i \leq n$, 
$M^{\gamma_i}_{\beta_i} \cap \gamma_{n+1} \subseteq  M^{\gamma_{n+1}}_{\beta_{n+1}}$.
Since $\gamma_{n+1} < \gamma_n$ we can use the inductive assumption from $i < n$ to get
$M^{\gamma_i}_{\beta_i} \cap \gamma_{n+1} \subseteq M^{\gamma_n}_{\beta_n} \cap \gamma_{n+1}$. 
As $\gamma_{n+1} \in M^{\gamma_n}_{\beta_n}$ then for every 
$\rho \in  M^{\gamma_n}_{\beta_n} \cap \gamma_{n+1}$ the index $\nu < \kappa$ of $\rho$ in a fixed well-order of the $\kappa$-sized structure $M^{\gamma_{n+1}}$ belongs to 
$M^{\gamma_n}_{\beta_n}$. Hence $\nu < \beta_n < \beta_{n+1}$ and therefore
$\rho \in M^{\gamma_{n+1}}_{\beta_{n+1}}$.\\

\noindent
So we now have to verify the claim for members of two different paths.
Let $$(M^{\gamma_i}_{\beta_i},\gamma_i)_{i\leq m}$$
and $$(M^{\delta_j}_{\alpha_j},\delta_j)_{j\leq n}$$ be paths. 
%Since we need to consider only such pairs of paths that are included in a common rooted sequence of models , we can assume that for $0<i\leq m,0<j\leq n$ $\alpha_j\neq\beta_i$.
 The claim now is that if $i\leq m,j\leq n$ then if 
$\beta_i<\alpha_j $, then 
$$ M^{\gamma_j}_{\beta_j}\cap(\min (\gamma_j, \delta_j)+1)\subseteq M^{\delta_i}_{\alpha_i}.$$
We prove this by induction on $n+m$. 
Note that if $n = 0$ or $m = 0$ then the models lie on the same path and therefore the claim follows from the previous argument. 
We therefore assume that $m,n>0$ and without loss of generality that  $\alpha_n \leq \beta_m$.
Using the induction assumption, we have the claim for the pair of paths $$(M^{\gamma_j}_{\beta_j},\gamma_j)_{j\leq m-1}$$
and $$(M^{\delta_i}_{\alpha_i},\delta_i)_{i\leq n}$$
So the only  additional cases of the claim we have to consider are between $M^{\delta_i}_{\alpha_i}$ and $M^{\gamma_m}_{\beta_m}$ for all $i\leq n$. By our assumption, $\alpha_i \leq \beta_m$ for all $i \leq n$. \\

\noindent 
We distinguish several cases :
\begin{description}
    \item[\textbf{ Case 1:  $\alpha_i \leq \beta_{m-1}$.}] 
    By the induction assumption for the two paths $$(M^{\gamma_j}_{\beta_j},\gamma_j)_{j\leq m-1}$$ and $$(M^{\delta_i}_{\alpha_i},\delta_i)_{i\leq n}$$ and the two members of these pairs
$M^{\delta_i}_{\alpha_i}$ and $M^{\gamma_{m-1}}_{\beta_{m-1}}$ we have
$$M^{\delta_i}_{\alpha_i}\cap (\min(\delta_i,\gamma_{m-1})+1)\subseteq M^{\gamma_{m-1}}_{\beta_{m-1}}.$$

Since we proved the claim for the path $$(M^{\delta_i}_{\alpha_i},\delta_i)_{i\leq n}$$
$$M^{\gamma_{m-1}}_{\beta_{m-1}}\cap (\min(\gamma_{m-1},\gamma_m)+1)\subseteq M^{\gamma_m}_{\beta_m}.$$
 Since $\gamma_m<\gamma_{m-1}$, $\min(\gamma_m,\gamma_{m-1})=\gamma_{m-1}$ and $\min(\delta_i,\gamma_m)\leq \min (\delta_i,\gamma_{m-1})$ we get
 $$M^{\delta_i}_{\alpha_i}\cap (\min(\delta_i,\gamma_m)+1)\subseteq M^{\gamma_{m-1}}_{\beta_{m-1}}\cap (\min(\delta_i,\gamma_m)+1)\subseteq M^{\gamma_m}_{\beta_m}.$$

\item[\textbf{ Case 2:  $\beta_{m-1}<\alpha_i \leq \beta_m$  and $\delta_i\geq \gamma_m$.}]${}$\\
 Recall that $\gamma_m\in M^{\gamma_{m-1}}_{\beta_{m-1}}$. By the induction assumption on the pair of paths:
 $$(M^{\gamma_j}_{\beta_j},\gamma_j)_{j\leq m-1}$$
and $$(M^{\delta_i}_{\alpha_i},\delta_i)_{i\leq n}$$

$M^{\gamma_{m-1}}_{\beta_{m-1}}\cap (\min(\gamma_{m-1}, \delta_i)+1)\subseteq M^{\delta_i}_{\alpha_i}$. Since $\gamma_m\leq \delta_i$ and $\gamma_m<\gamma_{m-1}$ , we get that $\gamma_m\in M^{\delta_i}_{\alpha_i}$.  Hence $f_{\gamma_m}\in M^{\delta_i}_{\alpha_i}$. 

Let $\rho \in M^{\delta_i}_{\alpha_i}\cap (\min(\delta_i,\gamma_m)+1)$. To see that $\rho \in M^{\gamma_m}_{\beta_m}$, consider the ordinal $\eta=f_{\gamma_m}(\rho)$.  We have that $\eta\in\kappa\cap M^{\delta_i}_{\alpha_i}=\alpha_i$ and as $\alpha_i \leq \beta_m=\kappa\cap M^{\gamma_m}_{\beta_m}$, $\eta\in M^{\gamma_m}_{\beta_m} $. It follows that $\rho=f_{\gamma_m}(\eta)\in M^{\gamma_m}_{\beta_m}$. 
%We proved $M^{\delta_i}_{\alpha_i}\cap (\min(\gamma_{m}, \delta_i)+1)\subseteq M^{\gamma_m}_{\beta_m} $.

\item[\textbf{ Case 3:  $\beta_{m-1}<\alpha_i \leq \beta_m$  and $\delta_i< \gamma_m$.}]${}$\\
Let $k\leq n$ be maximal such that $\alpha_k\leq\beta_{m-1}$ or $\delta_k\geq\gamma_m$. Note that
$k<i$. We prove by induction on $l$, where $k<l\leq i$, that $\delta_l\in M^{\gamma_m}_{\beta_m}$.
For each such $l$ we have that
$\alpha_{l} > \beta_{m-1}$ and $\delta_{l} < \gamma_m$. \\

The base case of $l = k+1$ requires a separate argument. We distinguish two cases.\\

\noindent
\textbf{($l = k+1$, first subcase)} Assume first that $\alpha_k \leq \beta_{m-1}$. \\
Then $\delta_{k+1}<\gamma_m<\gamma_{m-1}$ and $\delta_{k+1}\in M^{\delta_k}_{\alpha_k}$. By the induction assumption $M^{\delta_k}_{\alpha_k}\cap (\min(\delta_k,\gamma_{m})+1) \subseteq M^{\gamma_{m-1}}_{\beta_{m-1}}$. Hence $\delta_{k+1}\in M^{\gamma_{m-1}}_{\beta_{m-1}}$. It follows from our proof for the case of a single path that $\delta_{k+1}\in M^{\gamma_m}{\beta_m}$.\\

\noindent 
\textbf{($l = k+1$, second subcase)} Assume now $\beta_{m-1}<\alpha_k$.\\
Hence $\beta_{m-1}<\alpha_k \leq \beta_m$ and $\delta_k\geq\gamma_m$. By case 2 above, $M^{\delta_k}_{\alpha_k}\cap\gamma_m\subseteq M^{\gamma_m}_{\beta_m}$. $\delta_{k+1}<\delta_m$ and $\delta_{k+1}\in M^{\delta_k}_{\alpha_k}$ we get $\delta_{k+1}\in M^{\gamma_m}_{\beta_m}$.\\

\noindent
Establishing the assertion for $l = k+1$ we proceed by induction to show the same holds for other $l \leq i$.\\

\textbf{(Inductive step $l \to l+1$)}
Assume the claim  for $l$ and prove it for $l+1$. By assumption $\delta_l\in M^{\gamma_m}_{\beta_m}$. We have $\delta_{l+1}<\delta_l$ and $\delta_{l+1}\in M^{\delta_l}_{\alpha_l}$. By assumption $f_{\delta_l}\in M^{\gamma_m}_{\beta_m}\cap M^{\delta_l}_{\alpha_l}$. So $\rho=f^{-1}_{\delta_l}(\delta_{l+1})\in \alpha_i$. Hence $\rho\in \beta_m$. So $f_{\delta_l}^{-1}(\rho)\in M^{\gamma_m}_{\beta_m}$.

Using this for $l=i$ we get that $\delta_i\in M^{\gamma_m}_{\beta_m}$. As in the previous cases we get that $M^{\delta_i}_{\alpha_i}\cap (\min(\delta_i,\gamma_m)+1)\subseteq M^{\gamma_m}_{\beta_m} $.
\end{description}
This concludes the proof of the Claim. 
\end{proof}

\begin{proof}[Proof of Lemma \ref{LEM:tau-sequenceProperty}]
The lemma follows easily from the claim: if $\tau_i<\tau_j$, then by the claim we have $\tau_i\in M^{\tau_i}_{\alpha_i}\cap(\tau_i+1)\subseteq M^{\tau_j}_{\alpha_j}$.

\end{proof}

\noindent Claim \ref{claim:tausq} gives us another lemma that will be useful:

\begin{lemma}\label{lem:restau}
    Suppose that $\la (M^{\tau_i}_{\alpha_i},\tau_i)\mid i<\nu\ra$ is a $\tau$-sequence. If $k<j<\nu$ are such that $\tau_j\leq\tau_k$ and $\tau_j\in M^{\tau_k}_{\alpha_k}$, then the sequence
    \[
    \la (M^{\min\{\tau_i,\tau_j\}}_{\alpha_i},\min\{\tau_i,\tau_j\})\mid k\leq i<j\ra
    \]
    is a $\tau_j$-sequence and an element of $M^{\tau_j}_{\alpha_j}$.
\end{lemma}
\begin{proof}
    The fact that the sequence is an element of $M^{\tau_j}_{\alpha_j}$ follows from the fact that $M^{\tau_j}_{\alpha_j}$ is closed for countable sequences and using Claim \ref{claim:tausq}.
    
    We first show that if $i$ satisfies $k\leq i<j$, then $M^{\min\{\tau_i,\tau_j\}}_{\alpha_i}$ is well-defined, i.e. $\vec{M}^{\min\{\tau_i,\tau_j\}}$-reflective. For this, it suffices to show that if $\tau_j<\tau_i$, then
    \[
    \tau_j\in M^{\tau_i}_{\alpha_i},
    \]
    since this implies that $\alpha_i$ belongs to the sequence relevant to the coordinate $\tau_j$ and thus $M^{\tau_j}_{\alpha_i}$ is defined. To this end, suppose that $i$ satisfies $k\leq  i<j$ and $\tau_j<\tau_i$. If $i=k$, we are done. Otherwise, applying Claim \ref{claim:tausq} to $k<i$, we obtain the conclusion:
    \[
    \tau_j\in M^{\tau_k}_{\alpha_k}\cap(\min\{\tau_k,\tau_i\}+1)\subseteq M^{\tau_i}_{\alpha_i}.
    \]
    This is enough to conclude that each model $M^{\min\{\tau_i,\tau_j\}}_{\alpha_i}$ for $k\leq i<j$ is well-defined.

    We then show that the sequence is a $\tau_j$-sequence. For $i<j$, we denote
    \[
    \tau^j_i:=\min\{\tau_i,\tau_j\}.
    \]
    It suffices to verify that for any $i\in (k,j)$ there is $l\in [k,i)$ such that $\tau^j_i\leq\tau^j_{l}$ and $\tau^j_i\in M^{\tau^j_{l}}_{\alpha_{l}}$. Fix some such $i$. If $\tau_j\leq \tau_i$, then $l:=k$ is as wanted, since then $\tau^j_i=\tau_j$, so in particular $\tau^j_i\leq\tau^j_k$ and $\tau^j_i\in M^{\tau^j_k}_{\alpha_k}$. Suppose then $\tau_i<\tau_j$. Note that $\tau_j\leq\tau_k$ by assumption. Since we started with a $\tau$-sequence, there is $i'<i$ such that $\tau_i\leq \tau_{i'}$ and $\tau_i\in M^{\tau_{i'}}_{\alpha_{i'}}$.
    \begin{itemize}
        \item If $i'\leq k$, then Claim \ref{claim:tausq} gives
        \[
        \tau_i\in M^{\tau_{i'}}_{\alpha_{i'}}\cap(\min\{\tau_{i'},\tau_k\}+1)\subseteq M^{\tau_k}_{\alpha_k},
        \]
        and since $\tau_i,\tau_j\in M^{\tau_k}_{\alpha_k}$ and $\tau_i\leq\tau_j$, we have $\tau_i\in M^{\tau_j}_{\alpha_k}$, so $l:=k$ is as wanted.
        \item Suppose that $k< i'$.
        \begin{itemize}
            \item If $\tau_{i'}\leq\tau_j$, then $l:=i'$ is as wanted, since then $M^{\tau^j_{i'}}_{\alpha_{i'}}=M^{\tau_{i'}}_{\alpha_{i'}}$.
            \item If $\tau_j<\tau_{i'}$, then by Claim \ref{claim:tausq} applied to $k<i'$, we have
            \[
            \tau_j\in M^{\tau_k}_{\alpha_k}\cap(\min\{\tau_k,\tau_{i'}\}+1)\subseteq M^{\tau_{i'}}_{\alpha_{i'}}.
            \]
            So $\tau_i,\tau_j\in M^{\tau_{i'}}_{\alpha_{i'}}$, and since $\tau_i<\tau_j$, we have $\tau_i\in M^{\tau_j}_{\alpha_{i'}}=M^{\tau^j_{i'}}_{\alpha_{i'}}$. Thus $l:=i'$ is as wanted.
        \end{itemize}
    \end{itemize}
    We have now seen that $\la (M^{\tau^j_i}_{\alpha_i},\tau^j_i)\mid k\leq i<j\ra$ is a $\tau_j$-sequence in $M^{\tau_j}_{\alpha_j}$.
\end{proof}

\begin{definition}${}$

\begin{enumerate}
\item 
Let $p \in \po_\tau$. A pair of countable sequences $\la(M^{\tau_i}_{\alpha_i},\tau_i) \mid i < \nu\ra$ and $\la q^i \mid i < \nu\ra$ is a \emph{$p$-Multi-Extension}, $p$-M.E. in short, if it satisfies 

\begin{enumerate}

\item $p\uhr\tau_i \in D_{\tau_i}(M^{\tau_i}_{\alpha_i})$ for every $i < \nu$. 

 \item $q^i \in \po_{\tau_i} \cap M^{\tau_i}_{\alpha_i}$ and $q^i \leq [p\uhr\tau_i]_{M^{\tau_i}_{\alpha_i}}$ for every $i<\nu$.
 
\item $\la(M^{\tau_i}_{\alpha_i},\tau_i) \mid i < \nu\ra$ is a $\tau$-sequence and is an initial closure of its $p$-closure.

\item For each non-zero $j < \nu$, there is $k<j$ such that denoting
$\tau_i^j  =\min\{\tau_i,\tau_j\}$ then the sequence
$\la(M^{\tau^j_i}_{\alpha_i},\tau^j_i) \mid k\leq i < j\ra$ is a $\tau_j$-sequence.

    \item\label{ME:iteme} For every $i < j< \nu$, $q^j \uhr \tau^j_i \leq q^i \uhr \tau^j_i$.

\item For every $i < \nu$ and $s \in (\dom(f^p_{\tau_i}) \cap M^{\tau_i}_{\alpha_i})-\bigcup_{j<i}V_{\alpha_j}$, $q^i$ decides the following information about $b_s$:
\begin{itemize}
    \item the meet $m(s,s')$ in $S_{\tau_i}$ for every $s' \in  \dom(f^{q^j}_{\tau_i})$ for any $j < i$ with $\tau_i\in\supp(q^i)$. 
    \item the implicit image $t(s,s')$ of $m(s,s')$ (as described in Remark  \ref{Remark:ImagesOfMeets}).
    
    \item the exit node $\bar{s}_\beta:=e(s,M^{\tau_i}_\beta)$ for every structure $M^{\tau_i}_\beta$ which appears in $q^j$ at coordinate $\tau_i$ and its $M^1_{\bar{s}_\beta}$, for every $j < i$.
\end{itemize}
\item\label{ME:itemg} For all $i<i'<\nu$ and $s\in\dom(f^p_{\tau_i})$ such that $t:=f^p_{\tau_i}(s)\in M^{\tau_{i'}}_{\alpha_{i'}}$, the collapse coordinate of $q^{i'}$ forces the following information about $b_t$:
\begin{itemize}
    \item the implicit image $t(s,s')$ is the meet of $f^p_{\tau_i}(s)$ and $f^{q^j}_{\tau_i}(s')$ in $\dot T$, for any $s'\in\dom(f^{q^j}_{\tau_i})$ and $j<i$ with $\tau_i\in\supp(q^j)$,
    \item for every structure $M^{\tau_i}_\beta$ which appears in $q^j$ at coordinate $\tau_i$ for any $j<i$, there is an exit node $\bar{t}_\beta$ from $M^{\tau_i}_\beta$ at the height of $\bar{s}_\beta$ that is below $t$ and satisfies $\gamma_{\bar{t}_\beta})=\kappa\cap M^1_{\bar{s}_\beta}$, and furthermore for every $s'\in\dom(f^{q^j}_{\tau_i})$,
    \[
    t\wedge f^{q^j}_{\tau_i}(s')=\bar{t}_\beta\wedge f^{q^j}_{\tau_i}(s').
    \]
\end{itemize}
\end{enumerate}

\item The \emph{simple amalgamation} of a sequence of conditions $\la q^i \mid i < \nu\ra$ is the pair $(f,N)$ defined by  
\begin{itemize} 
\item $\dom(f) = \bigcup_{i < \nu} \supp(q^i)$.
\item for each $\delta \in \dom(f)$, $f_\delta = \bigcup\{f^{q^i}_\delta \mid i < \nu \text{ and } \delta \in \supp(q^i)  \}$.  
\item for each $\alpha < \kappa$, the set $N$ includes a pair $(\alpha,a_\alpha)$ if and only if there is some pair $(\alpha,a')$  in $\bigcup_i N^{q^i}$, in which case $a_\alpha$ is taken to be the union of all such $a'$. \end{itemize}
\end{enumerate}
\end{definition}

\begin{lemma}\label{Lemma:PropernessForSequences}${}$\\
For every $p \in \po_\tau$ and $p$-M.E. pair of sequences $\la(M^{\tau_i}_{\alpha_i},\tau_i) \mid i < \nu\ra$ and $\la q^i \mid i < \nu\ra$ there is a condition $q \in \po_\tau$ such that $q \leq p$ and for each $i < \nu$, 
$q\uhr\tau_i \leq q^i$.\\
%Moreover, if $\tau = \delta+1$ is a successor ordinal, $\alpha < \alpha_0$ and $w \leq [p]_{M^\tau_{\alpha}}$ satisfies $w\uhr \tau_i \geq q^i$ for all $i < \nu$ then one can find $q$ such that $q \leq w$.
\end{lemma}                                                                                                                    
\begin{proof}
The proof is by induction on pairs $(\tau, \alpha)$ (with the usual lexicographic ordering) where $\alpha < \kappa$ is minimal so that $\alpha \geq \cup_{i < \nu} \alpha_i$, and $p \in M^\tau_\alpha$. 

%\textbf{(Limit Case)} We start by proving the inductive step for the case $\tau$ is a limit ordinal, and its main subcase in which $\tau = \liminf_{i < \nu}\tau_i$. 

Let $\la (\tau_j,M^{\tau_j}_{\alpha_j}) \mid j < \nu^*\ra$ be the $p$-closure of $\la (M^{\tau_i}_{\alpha_i},\tau_i) \mid i < \nu\ra$. Denote
\[
\tau_{\nu^*}:=\tau\text{ and }\alpha_{\nu^*}:=\kappa,
\]
with the convention that $M^\tau_\kappa=Hull^{\mathcal{A}^\tau}(\kappa)$.
By the assumption, the $p$-closure is an end extension of $\la (M^{\tau_i}_{\alpha_i},\tau_i) \mid i < \nu\ra$.
Moreover, it follows from the definition of super-nice conditions and the fact $p\uhr \tau_i \in D_{\tau_i}(M^{\tau_i}_{\alpha_i})$ for all $i < \nu$ that $p\uhr \tau_j \in  D_{\tau_j}(M^{\tau_j}_{\alpha_j})$ for all $j < \nu^*$.

\noindent
Our next step is to extend the corresponding sequence of conditions $\la q^i \mid i < \nu\ra$ to a sequence $\la q^{j} \mid j \leq \nu^*\ra$, so that 
$\la (\tau_j,M^{\tau_j}_{\alpha_j}) \mid j \leq \nu^*\ra$  and $\la q^{j} \mid j \leq  \nu^*\ra$ form a $p$-Multi-Extension.
Suppose that  $\la q^j \mid j < \eta\ra$ has been defined for some $\nu \leq \eta \leq \nu^*$, 
such that $\la (\tau_j,M^{\tau_j}_{\alpha_j}) \mid j < \eta\ra$ and $\la q^j \mid j < \eta\ra$  are $p$-M.E.
We define $q^\eta$ in four steps that go through auxiliary conditions ${q}^{\eta,0}$ and ${q}^{\eta,1}$. The entire construction happens inside $M^{\tau_\eta}_{\alpha_\eta}$.

\begin{itemize}
    \item[(${q}^{\eta,0}$)]
    Let ${q}^{\eta,0} = [p \uhr \tau_\eta]_{M^{\tau_\eta}_{\alpha_\eta}}$. Also, denote for each $j < \eta$, $\tau^\eta_j = \min(\tau_j,\tau_\eta)$. 
    It follows from the definition of a $p$-M.E. pair and the fact 
    $p\uhr \tau_j \in D_{\tau_j}(M^{\tau_j}_{\alpha_j})$ for all $j < \eta$, that there is $k<\eta$ such that the pair of sequences $\la (M^{\tau^\eta_j}_{\alpha_j}, \tau^\eta_j) \mid k\leq  j < \eta\ra$ and $\la q^j\uhr \tau^\eta_j \mid j < \eta\ra$ forms a 
    $q^{\eta,0}$-M.E.

    \item[(${q}^{\eta,1}$)]
      To obtain ${q}^{\eta,1}$, we extend the collapse part  of $q^{\eta,0}$ to make decisions about nodes in $T$ that fit relevant decisions about meets and exit nodes in various trees $S_\delta$, which were made by previous conditions $q^j$, $j < \eta$, according to property (\ref{ME:itemg}) in the definition of $M.E.$ sequences. 
      More precisely, 
      let $\la (s_n,t_n) \mid n < M\ra_{n < M}$, $M \leq \omega$, enumerate all the pairs $(s,t)$ for which there is $j < \eta$ such that
      \begin{itemize}
          \item $(s,t)\in f^p_{\tau_j}$,
          \item $t\in M^{\tau_\eta}_{\alpha_\eta}-\bigcup_{j<\eta}M^{\tau_j}_{\alpha_j}$,
      \end{itemize}
      Note that it follows that $\gamma_t\leq\sup_{j<\eta}\alpha_j$, since the model $M^1_s$ occurs in the $\tau$-sequence before $M^{\tau_j}_{\alpha_j}$ and $\gamma_t=\kappa\cap M^1_s$.
      
      By Lemma \ref{lem:freem}, there is an extension $q^{\eta,1}\leq q^{\eta,0}$ such that for every $j<\eta$, it is super-nice with respect to every $M^{\tau^\eta_j}_{\alpha_j}$, satisfies $[q^{\eta,1}]_{M^{\tau^\eta_j}_{\alpha_j}}=[q^{\eta,0}]_{M^{\tau^\eta_j}_{\alpha_j}}$ and forces the information satisfying (\ref{ME:itemg}) for $t_n$, $n<M$.

    \item[(${q}^{\eta}$)] This step breaks into two cases: either $\tau_\eta=\tau$, or $\tau_\eta<\tau$.
    
    \item[${}$]\textbf{Case $\boldsymbol{\tau_\eta<\tau}$:}
    We want to apply the inductive assumption of the lemma inside the structure $M^{\tau_\eta}_{\alpha_\eta}$. To this end, we claim that the sequences $\la q^{\eta}\rest\tau^\eta_j:j<\eta\ra$ and $\la(M^{\tau^\eta_j}_{\alpha_j},\tau^\eta_j):k\leq j<\eta\ra$ form a $q^{\eta,1}$-Multi-Extension. But this follows from Lemma \ref{lem:restau} and the fact that, by assumption, they form a $q^{\eta,0}$-Multi-Extension, which implies together with the fact that the condition $q^{\eta,1}$ is super-nice with respect to each $M^{\tau^\eta_j}_{\alpha_j}$, and was obtained from $q^{\eta,0}$ by extending only the collapse coordinate, that they also form a $q^{\eta,1}$-Multi-Extension. 
    
    Thus, by induction hypothesis applied inside the model $M^{\tau_\eta}_{\alpha_\eta}$, there is a condition $q^{\eta,2}\in\po_{\tau_\eta}\cap M^{\tau_\eta}_{\alpha_\eta}$ that extends each $q^j\rest\tau^\eta_j$ and $q^{\eta,1}$. 
    
    Finally, working inside $M^{\tau_\eta}_{\alpha_\eta}$, we find the final extension $q^\eta\leq q^{\eta,2}$ which determines the following information for every $s\in(\dom(f^p_{\tau_\eta})\cap M^{\tau_\eta}_{\alpha_\eta})-\bigcup_{j<\eta}V_{\alpha_j}$:
        \begin{itemize}
            \item The meets of all pairs of nodes $m(s,s')$ where $s' \in \bigcup_{i < j}\dom(f^{q^{i}}
            _{\tau_\eta})$
            \item The implicit images $t(s,s')$ of $m(s,s')$ (see Remark \ref{Remark:ImagesOfMeets}).
            \item For every $j<\eta$ with $\tau_\eta\in\supp(q^j)$ and every structure $M^{\tau_\eta}_{\beta}$ that appears in $q^j$, the exit node below $s$ from $M^{\tau_\eta}_\beta$.
        \end{itemize}

    \item[${}$]\textbf{Case $\boldsymbol{\tau_\eta=\tau}$:} 
    First, let $q'$ be the simple amalgamation of the conditions $\la q^j:j<\eta\ra$. It follows from item (\ref{ME:iteme}) the definition of Multi-Extension that $q'$ is a condition in $\po_{\sup_{j<\eta}\tau_j}$. Then, for every $\gamma\in\tau\cap M^\tau_{\alpha_\eta}$, let $f^1_\gamma$ be the extension of the function $f^{q'}_\gamma$ by the pairs
    \[
    (s,t),(m(s,s'),t(s,s'))
    \]
    where $(s,t)\in f^{q^{\eta,0}}_\gamma$ and $s'\in\dom(f^{q^j}_{\gamma})$, $j<\eta$. By construction, $q'\rest\gamma$ decides meets in the domain of $f^1_\gamma$ and forces that $f^1_\gamma$ is a level- and meet-preserving injective tree-embedding from $\dot S_\gamma\cap M^\tau_{\alpha_\eta}$ to $\dot T$.
    Furthermore, still by construction, using item \ref{ME:itemg}, for every $s\in\dom(f^{q^{\eta,1}}_\gamma)$ and model $M^\gamma_\beta$ that appears in $q^j$ at coordinate $\gamma$ for any $j<\eta$, the condition $q'\rest\gamma$ decides the exit node $\bar{s}_\beta$ from $M^\gamma_\beta$ below $s$, and there is a ``image" $\bar{t}_\beta$ such that if we let $f^2_\gamma$ to be the extension of $f^1_\gamma$ by the pairs
    \[
    (\bar{s}_\beta,\bar{t}_\beta),
    \]
    where $s\in\dom(f^{q^{\eta,1}}_\gamma)$ and $M^\gamma_\beta$ appears in $q^j$ at coordinate $\gamma$ for some $j<\eta$.

    Finally, let $q^{\eta}$ be the condition obtained by extending $q^{\eta,1}$ at each coordinate $\gamma\in\tau\cap M^\tau_{\alpha_\eta}$ by the function $f^2_\gamma$. By construction, $q^\eta$ is a condition in $\po_\tau\cap M^\tau_{\alpha_\eta}$.
    
\end{itemize}

\vspace{7pt}

\noindent It is clear from the construction that the sequences
$\la (M^{\tau_j}_{\alpha_j},\tau_j) \mid j \leq  \eta\ra$ and
$\la q^j \mid j \leq  \eta\ra$ constitute a $[p\rest\tau_\eta]_{M^{\tau_\eta}_{\alpha_\eta}}$-Multi-Extension.

This concludes the recursion. Now, by construction, the condition $q^{\nu^*}$ is a condition in $\po_\tau$ that extends $p$ and each $q^j$.

\end{proof}

We can now prove Proposition \ref{Proposition:MainStrongProper}.
\begin{proof}[Proof of Proposition \ref{Proposition:MainStrongProper}]
Let $p \in D_{\tau}(M^\tau_\alpha)$. To show that $[p]_{M^\tau_\alpha}$ is a residue for $p$, we need to verify that it is compatible with every condition $w \in \po_\tau \cap M^\tau_\alpha$ which extends $[p]_{M^\tau_\alpha}$. This is an immediate consequence of the previous Lemma with the $p$-M.E. pair of sequences of length $\nu = 1$ with $(M^{\tau_0}_{\alpha_0},\tau_0) = (M^\tau_\alpha,\tau)$, and $q^0 = w$. 
\end{proof}

\noindent Phrased using notation from \ref{de:EN}, Proposition \ref{Proposition:MainStrongProper} says that if $p\in\po_\tau$ and $\alpha\in E^p_\tau$, then $p$ has a residue into $M^\tau_\alpha$.

\begin{theorem}\label{Theorem-Main-Sec1}
Let $\po_{\kappa^+} = \bigcup_{\tau < \kappa^+} \po_\tau$. 
\begin{enumerate}
\item $\po_{\kappa^+}$ has $\kappa^+$-c.c. 
\item $\po_{\kappa^+}$ is $\sigma$-closed and thus, does not collapse $\omega_1$.
\item $\po_{\kappa^+}$ collapses all cardinals between $\omega_1$ and $\kappa$.
\item  $\po_{\kappa^+}$ does not collapse $\kappa$ 
\item $2^{\aleph_1} = \kappa^+$ in $V^{\po_{\kappa^+}}$.
\item For each $\tau < \kappa^+$, the $\tau$-th wide-tree $S^\tau$ chosen by the book-keeping function $\Psi$ embeds into $T$. In particular, if $\Psi$ covers all $\po_\tau$ names of wide trees on $\kappa$ for all $\tau < \kappa^+$, then $T$ is a maximal wide tree in the generic extension by $\po_{\kappa^+}$.
\end{enumerate}
\end{theorem}

\begin{proof}${}$
\begin{enumerate}
\item This is a standard consequence of the fact that each $\po_\tau$, $\tau < \kappa^+$ has size $\kappa$, and that there is a stationary set of $\tau < \kappa^+$ for which $\po_\tau$ is a direct limit of $\po_\delta$, $\delta < \tau$  (the set of all limit $\tau < \kappa^+$ of uncountable cofinality).

\item Immediate from the fact each $\po_\tau$, $\tau < \kappa^+$ is $\sigma$-closed (inductive assumption I).

\item Immediate. As $\po_{\kappa^+}$ embeds $\po_0 = \col(\omega_1,<\kappa)$. 

\item This follows from the fact that $\po_{\kappa^+}$ is $\kappa^+$.c.c, and from the strong properness of all $\po_\tau$, $\tau < \kappa^+$, 
for structures of size $\aleph_1$. 

\item For every $\tau < \kappa^+$, the poset $\po_{\tau+1}$ introduces a tree embedding $f_\tau : S_\tau \to T$, so that for every structure $M^{\tau+1}_\alpha \in \vec{M}^{\tau+1,G(\po_{\tau+1})}$, 
$f_\tau \uhr M^{\tau+1}_\alpha$ is generic over $V[G(\po_\tau)]$ and introduces an embedding of $S_\tau \cap M^{\tau+1}_\alpha$ to $T \cap M^{\tau+1}_\alpha$, which is a new subset of $M^{\tau+1}_\alpha$ of size $\aleph_1$. 

\item Immediate by the construction of the posets $\po_\tau$, $\tau < \kappa^+$.
\end{enumerate}
\end{proof}

\subsection{Strong properness of  quotients}

To prove that the forcing does not add new branches to $T$ given the chosen trees $S_\tau$, $\tau < \kappa^+$, are all wide Aronszajn, we need a slightly stronger version of strong properness that can be applied in quotients. The setup to have in mind is having some $\tau < \kappa^+$, $\alpha < \kappa$ and a condition $p \in D_\tau(M^\tau_\alpha)$.

\begin{lemma}\label{lemma:PropernessOfQuotients}
Let $M^\tau_\alpha,M^\tau_\beta \in \vec{M}^\tau$, $\alpha <\beta$.
Suppose that $p \in D_\tau(M^\tau_\beta) \cap D_\tau(M^\tau_\alpha)$. 
Let $G_\alpha \subseteq \po_\tau \cap M^\tau_\alpha$ be generic over $V$ with 
$[p]_{M^\tau_\alpha} \in G_\alpha$.
\begin{enumerate}
    \item 
The quotient $\po_\tau/G_\alpha$ has a $\sigma$-closed dense subset.

\item 
For every $w \in (\po_\tau/ G_\alpha) \cap M_\beta^\tau[G_\alpha]$ with $w \leq [p]_{M^\tau_\beta}$, 
there is a common extension $q \leq_{\po_\tau/G_\alpha} w,p$, such that $q \in D_\tau(M^\tau_\beta)$.
\end{enumerate}
\end{lemma}

\begin{proof}${}$
\begin{enumerate}
    \item 
The quotient is $\sigma$-closed with respect to conditions in $D_\tau(M^\tau_\alpha)$. This is an immediate consequence of the fact that 
$D_\tau(M^\tau_\alpha)$ is $\sigma$-closed dense below $p$, of Proposition \ref{Proposition:MainStrongProper}, 
and the fact that the trace map $p \mapsto [p]_{M^\tau_\alpha}$ respects countable joins.

\item\label{quo:item2} 
Let $p \in D_\tau(M^\tau_\alpha) \cap D_\tau(M^\tau_\beta)$, $p \in \po_\tau/G_\alpha$ and $w \leq [p]_{M^\tau_\beta}$, $w \in \po_\tau/G_\alpha$, $w \in M^\tau_\beta[G_\alpha]$. By extending $w$ we may assume that it belongs to $D_\tau(M^\tau_\alpha)$, and therefore that $[w]_{M^\tau_\alpha}$ forces that $w$ belongs to the quotient $\po_\tau/G_\alpha$. 

Suppose towards contradiction that there exists an extension $r \in M^\tau_\beta$ of $w$ which is incompatible with $p$ in the quotient forcing $\po_\tau/G_\alpha$. 

This means that in $V$, there is $w' \leq [w]_{M^\tau_\alpha}$, $w' \in \po_\tau \cap M^\tau_\alpha$, which forces that $r$ and $p$ are incompatible as conditions of the quotient forcing, over an extension by $\po_\tau \cap M^\tau_\alpha$.
By strong properness (Proposition \ref{Proposition:MainStrongProper}) with respect to $M^\tau_\alpha$, $r$ and $w'$ are compatible by some $r' \in M^\tau_\beta$. Since $r' \leq r \leq w \leq [p]_{M^\tau_\beta}$ we can apply strong properness again with respect to $M^\tau_\beta$  and conclude that $r'$ and $p$ are compatible by some $q \in D^\tau(M^\tau_\beta) \cap D^\tau(M^\tau_\alpha)$. But now $[q]_{M^\tau_\alpha}$ extends $w'$ and forces contradictory information. 
\end{enumerate}

\end{proof}

We also need an extrapolation of item (\ref{quo:item2}) of the previous lemma. In what follows, we will have a generic filter $G\subseteq\po_\tau\cap M^\tau_\alpha$ and a path from $M^\tau_\alpha$ to some $M^{\tau_i}_{\alpha_i}$, where $\tau_i\leq\tau$. By Claim \ref{claim:tausq} it follows that
\[
\tau_i\cap M^\tau_\alpha\subseteq M^{\tau_i}_{\alpha_i}.
\]
Thus, if $\alpha_i>\alpha$, then the poset $\po_{\tau_i}\cap M^\tau_\alpha$ belongs to $M^{\tau_i}_{\alpha_i}$ and thus filter $G\rest\tau_i=\{p\rest\tau_i:p\in G\}$ is also an element in $M^{\tau_i}_{\alpha_i}$ and $M^{\tau_i}_{\alpha_i}$-generic on $\po_{\tau_i}\cap M^\tau_\alpha$.

\begin{lemma}\label{lem:tau-sq:quo}
    Let $\tau<\kappa^+$. If $p\in\po_\tau$ is super-nice with respect to $M^\tau_\alpha$, $G\subseteq\po_\tau\cap M^\tau_\alpha$ is a generic that contains the trace of $p$ and $\la((M^{\tau_i}_{\alpha_i},\tau_i):i<\nu\ra$ and $\la q^i:i<\nu\ra$ is a $p$-M.E. pair of sequences such that $\alpha_i>\alpha$ and $q^i\in\po_{\tau_i}/G\rest\tau_i$ for every $i<\nu$, then there is $q\in\po_\tau/G$ such that $q\leq p$ and $q\rest\tau_i\leq q^i$ for every $i<\nu$.
\end{lemma}
\begin{proof}
    As the proof of Lemma \ref{Lemma:PropernessForSequences}, but using the above observation along with the stronger inductive assumption that a common extension can be found the in the quotient.
\end{proof}

\section{No new branches}\label{Section:NoNewBranches}

In this section, we prove that if all chosen names $\name{S}^\tau$ are names for wide Aronszajn trees, then $T$ does not get a cofinal branch, and so can become a universal wide Aronszajn tree on $\kappa = \aleph_2$.

If $p$ is a condition in $ \po_\delta$ and $\alpha\in N^p_\gamma$ for every $\gamma\in\delta\cap M^\tau_\alpha$, then automatically $p'$ obtained from $p$ by letting $N^{p'}:=N^p\cup\{(\gamma,\alpha):\gamma\in\tau\cap M^\tau_\alpha\}$ is a condition in $ \po_\tau$ that satisfies $p'\rest\delta=p$. 

\begin{definition}
    Let $\delta\leq\tau$ and let $\alpha$ be such that $M^\tau_\alpha$ belongs to the $\tau$-th sequence of side conditions. A condition $q\in \po_\delta$ is \emph{super-nice with respect to $M^\tau_\alpha$} if $\alpha\in N^q_\gamma$ for every $\gamma\in\delta\cap M^\tau_\alpha$ and the condition $q'$ obtained from $q$ by extending the side conditions by replacing every $(\alpha,a_\alpha)\in N^q$ by
    \[
    \begin{cases}
    (\alpha,a_\alpha\cup\{\gamma\})\quad& \text{if }\gamma\in\tau\cap M^\tau_\alpha,\\
    (\alpha,a_\alpha)&\text{otherwise,}
    \end{cases}
    \]
    is a condition in $ \po_\tau$ that is super-nice with respect to $M^\tau_\alpha$.
\end{definition}

\noindent Also the trace operator from Definition \ref{Def:Trace} is defined for conditions in $\po_\delta$ and models $M^\tau_\alpha$, where possibly $\delta<\tau$. It works as expected in the case of this generalised version of super-niceness:

\begin{lemma}
    Let $\delta\leq \tau$ and let $\alpha<\kappa$. If $p\in \po_\delta$ is super-nice with respect to $M^\tau_\alpha$, then:
    \begin{enumerate}
        \item the conditions that are super-nice with respect to $M^\tau_\alpha$ are dense below $p$ in $ \po_\delta$,
        \item the trace $[p]_{M^\tau_\alpha}$  is a condition in $\po_\delta\cap M^\tau_\alpha$ and
        a residue of $p$ into $M^\tau_\alpha$.
    \end{enumerate}
\end{lemma}

\noindent It follows that if $p\in \po_\delta$ is super-nice with respect to $M^\tau_\alpha$, then every generic $G\subseteq \po_\delta\cap M^\tau_\alpha$ that contains $[p]_{M^\tau_\alpha}$ extends to a generic on $ \po_\delta$ that contains $p$.

\begin{lemma}
    If there is a path from $M^\tau_\alpha$ to $M^\delta_\beta$ (see Definition \ref{de:path}), then the trace map from $\po_\delta\cap M^\delta_\beta$ to $\po_\delta\cap M^\tau_\alpha$ is a residue map and element in $M^\delta_\beta$.
\end{lemma}
\begin{proof}
    Note that if there is a path from $M^\tau_\alpha$ to $M^\delta_\beta$, then $\delta\cap M^\tau_\alpha\subseteq M^\delta_\beta$. Thus, if $p\in\po_\delta\cap M^\delta_\beta$ is super-nice with respect to $M^\tau_\alpha$ and $w\in\po_\delta\cap M^\tau_\alpha$ extends $[p]_{M^\tau_\alpha}$, then $w\in M^\delta_\beta$, and thus by elementarity $w$ and $p$ have a common extension in $\po_\delta\cap M^\delta_\beta$.
\end{proof}

\begin{definition}
    Two conditions $p,q\in\po_\tau$ are said to \emph{split a pair of nodes $(s,s')$} in a tree $\dot S$ if there are distinct nodes $\bar{s}\neq\bar{s}'$ of the same height $\bar{\alpha}$ such that
    \begin{enumerate}
        \item $p\Vdash \bar{s}<_{\dot S}s$ %and $\height(\bar{s})=\alpha$,
        \item $q\Vdash\bar{s}'<_{\dot S}s'$ %and $\height(\bar{s}')=\alpha$.
    \end{enumerate}
    We say that $p$ and $q$ split the node $s$ if they split the pair $(s,s)$.
\end{definition}

\begin{lemma}\label{lem:split1step}
    Let $\tau<\kappa^+$ and suppose that $\alpha$ belongs to the $\tau$-th sequence of side conditions. If $G\subseteq \po_\tau\cap M^\tau_\alpha$ is a generic filter  and if $s\in \dot S_{\tau}$ is an exit node from $M^\tau_\alpha$ at a limit level, then the branch below $s$ is not introduced by $G$.
\end{lemma}
\begin{proof}
    Suppose towards contradiction that $b\subseteq (\dot S_\tau\cap V_\alpha)^G$ is forced to be the branch below $s$. By the $\Pi^1_1$-reflection, the tree $(\dot S_\tau\cap V_\alpha)^G$ is a wide $\alpha$-Aronszajn tree. If $s$ has height $\alpha$, then the branch below it cannot be introduced by $G$, for it would be a cofinal branch in $(\dot S_\tau\cap V_\alpha)^G$.
    Thus the height $\bar{\alpha}$ of $s$ must satisfy $\bar{\alpha}<\alpha$. Look at $\bar{\beta}:=$ the minimal ordinal such that $b\subseteq\bar{\beta}\times\bar{\alpha}$. We have $\bar{\beta}\leq\alpha$ because $s$ is an exit node from $M^\tau_\alpha$.
    
    If $\bar{\beta}=\alpha$, then $b$ induces a cofinal function from $\bar{\alpha}$ to $\alpha$, which is absurd since $\alpha=\aleph_2^{V[G]}$. 
    
    Thus $\bar{\beta}<\alpha$. Let $G'$ be a generic on $ \po_\tau$ that extends $G$. By a standard argument the set $M^\tau_\alpha[G']=\{\dot a^{G'}:\dot a\in V^{ \po_\tau}\cap M^\tau_\alpha\}$ is an elementary submodel of $H_{\kappa^{++}}[G']$ and closed under $<\alpha$-sequences. Thus $b\in M^\tau_\alpha[G']$. And by elementarity, $b$ must have a supremum in $M^\tau_\alpha[G']$. Since $s$ is at a limit level, it is the unique supremum of $b$. This is absurd since $s$ is an exit node from $M^\tau_\alpha[G']$.
\end{proof}

\begin{lemma}
    Suppose that there is a path from $M^\tau_\alpha$ to $M^\delta_\beta$. Suppose that $p\in \po_\delta$ and $q\in \po_\delta$ are conditions that are super-nice with respect to $M^\tau_\alpha$ and $[p]_{M^\tau_\alpha}=[q]_{M^\tau_\alpha}$. Suppose also that $p$ is super-nice with respect to $M^\delta_\beta$. For any countable sets $A,B\subseteq\dot S_\delta$ of exit nodes from $M^\delta_\beta$, there are two conditions $\hat{p}\leq p$ and $\hat{q}\leq q$ that split every pair in $A\times B$, are super-nice with respect to $M^\tau_\alpha$ and $[p]_{M^\tau_\alpha}=[q]_{M^\tau_\alpha}$. Furthermore, for any $\beta'>\beta$ such that $A\subseteq M^\delta_{\beta'}$, if $p\in M^\delta_{\beta'}$, then we may assume that $\hat{p}\in M^\delta_{\beta'}$.
\end{lemma}
\begin{proof}
    Fix a generic $G\subseteq\po_\delta\cap M^\tau_\alpha$ that contains the common trace of $p$ and $q$. Enumerate $A\times B$ as $(a_0,b_0),(a_1,b_1),\dots$. By recursion on $n$, define conditions $p_n$ and $q_n$ in $ \po_\delta$. Let $p_0:=p$ and $q_0:=q$. At step $n+1$, assume that we have defined $p_n$ and $q_n$ and they satify:
    \begin{itemize}
        \item $p_n$ and $q_n$ are super-nice with respect to $M^\tau_\alpha$ and $[p_n]_{M^\tau_\alpha}=[q_n]_{M^\tau_\alpha}$,
        \item $p_n$ and $q_n$ split every pair $(a_i,b_i)$, $i<n$,
        \item $p_n$ is super-nice with respect to $M^\delta_\beta$ and $p_n\in M^\delta_{\beta'}$.
    \end{itemize}
    Look at the node $a_n$. The branch below $a_n$ is not introduced by $ \po_\delta\cap M^\delta_\beta$ by Lemma \ref{lem:split1step}, so it is not introduced by $ \po_\delta\cap M^\tau_\alpha$ either. Thus there are $p^L,p^R\leq p_n$ in $ \po_\delta/G$ that split $a_n$, at some level $\bar{\beta}<\beta$ with some distinct nodes $a^L$ and $a^R$. Since $A\subseteq M^\delta_{\beta'}$, by elementarity we may choose $p^L,p^R\in M^\delta_{\beta'}$. Find $q_{n+1}\leq q_n$ in $ \po_\delta/G$ that decides the predecessor $\bar{b}$ of $b_n$ at height $\bar{\beta}$. If $\bar{b}\neq a^L$, let $p_{n+1}:=p^L$ and otherwise let $p_{n+1}:=p^R$. Finally, let $\hat{p}$ be the pointwise union of the $p_n$ and let $\hat{q}$ be the pointwise union of the $q_n$. They are both in $ \po_\delta/G$ so they can be assumed to be super-nice with respect to $M^\tau_\alpha$ and $[\hat{p}]_{M^\tau_\alpha}=[\hat{q}]_{M^\tau_\alpha}$. They split every pair in $A\times B$, as desired. Furthermore, $\hat{p}\in M^\delta_{\beta'}$, since $M^\delta_{\beta'}$ is closed for countable sequences.
\end{proof}

\begin{lemma}\label{lem:sup}
    Let $\tau<\kappa^+$. For every $p\in \po_\tau$ that is super-nice with respect to $M^\tau_\alpha$ there is $q\leq p$ that is super-nice with respect to $M^\tau_\alpha$ and satisfies that if $\la (M^{\tau_i}_{\alpha_i},\tau_i):i<\nu\ra$ is the $p'$-closure of $\la(M^\tau_\alpha,\tau)\ra$, then 
    \[
    \sup\{\alpha_i:i<\nu\}=\sup\{\alpha_i:\tau_i\in\tau\cap M^\tau_\alpha, i<\nu\}.
    \]
    Moreover, this supremum can be chosen as high in $\kappa$ as wanted.
\end{lemma}
\begin{proof}
    By adding nodes to coordinates in $M^\tau_\alpha$ using Node density Lemma.
\end{proof}

\begin{definition}
    Let $\delta\leq\tau$. A pair of conditions $(p,q)$ from $\po_\delta$ \emph{nicely splits with respect to $M^\tau_\alpha$} if
    \begin{enumerate}
        \item $\alpha\in N^p_\gamma\cap N^q_\gamma$ for every $\gamma\in\delta\cap M^\tau_\alpha$,
        \item for every $\gamma\in\tau\cap M^\tau_\alpha$ and every pair $(s,s')$ of exit nodes from $M^\gamma_\alpha$ at limit levels where $s\in \dom(f^p_\gamma)$ and $s'\in \dom(f^q_\gamma)$, the conditions $p\rest\gamma$ and $q\rest\gamma$ split the pair $(s,s')$ with a pair $(\bar{s},\bar{s}')$ where $\bar{s}\in\dom(f^p_\gamma)$ and $\bar{s}'\in\dom(f^q_\gamma)$.
        \item for every $\gamma\in\tau\cap M^\tau_\alpha$ and every $s\in\dom(f^p_{\gamma})$ (resp. $s\in\dom(f^q_\gamma)$) exit node from $M^\gamma_\alpha$ at a successor level, the condition $p\rest\gamma$ (resp. $q\rest\gamma$) decides the immediate predecessor $\bar{s}$ of $s$ and $\bar{s}\in\dom(f^p_\gamma)$ (resp. $\bar{s}\in\dom(f^q_\gamma)$).
    \end{enumerate}
\end{definition}

\begin{lemma}\label{lem:split}
    Let $\tau<\kappa^+$, $\alpha<\kappa$ and $p,q \in \po_\tau$.
    \begin{enumerate}
        \item\label{item1:lem:split} If $p$ and $q$ nicely split with respect to $M^\tau_\alpha$ and $[p]_{M^\tau_\alpha} = [q]_{M^\tau_\alpha} $ then there are $\hat{p}\leq p$ and $\hat{q}\leq q$ that are super-nice with respect to $M^\tau_\alpha$ and $[\hat{p}]_{M^\tau_\alpha} = [\hat{q}]_{M^\tau_\alpha} $.
        \item\label{item2:lem:split} If $p$ and $q$ are super-nice with respect to $M^\tau_\alpha$ and $[p]_{M^\tau_\alpha} = [q]_{M^\tau_\alpha} $, then there are $\hat{p}\leq p$ and $\hat{q}\leq q$ that nicely split with respect to $M^\tau_\alpha$ and $[\hat{p}]_{M^\tau_\alpha} = [\hat{q}]_{M^\tau_\alpha} $. Moreover, for any $t\in T$ that is an exit node from $M^\tau_\alpha$ at a limit level, we may choose $\hat{p}$ and $\hat{q}$ such that they also split $t$.
        \item\label{item3:lem:split} $\dot T$ is a wide $\kappa$-Aronszajn tree in $V^{ \po_\tau}$.
    \end{enumerate}
\end{lemma}
\begin{proof}
    The proof is by induction on $\tau$. 

    \begin{remark}
        If Lemma \ref{lem:split} holds for $\tau$, then any pair of conditions $p,q\in\po_\tau$ having the same trace to $M^\tau_\alpha$ satisfies that if either they nicely split or are super-nice with respect to $M^\tau_\alpha$, then there are extensions that still have the same trace to $M^\tau_\alpha$ and both nicely split and are super-nice with respect to $M^\tau_\alpha$. This can be proved by iterating $\omega$ many times items (\ref{item1:lem:split}) and (\ref{item2:lem:split}).
    \end{remark}

    \noindent\textbf{Proof of item $\boldsymbol{\ref{item1:lem:split}}$.}

    \vspace{4pt}
    
    \noindent We look at the successor case $\tau+1$ first. Let $p,q\in\po_{\tau+1}$ be such that they have a common trace to $M^{\tau+1}_\alpha$ and nicely split with respect to it. By induction hypothesis and the above remark we may assume that $p\rest\tau$ and $q\rest\tau$ nicely split and are super-nice with respect to $M^\tau_\alpha$. 

    Let $G\subseteq\po_\tau\cap M^\tau_\alpha$ be a generic that contains the common trace of $p\rest\tau$ and $q\rest\tau$. Then $p\rest\tau,q\rest\tau\in\po_\tau/G$ and the common trace of $f^p_\tau$ and $f^q_\tau$ to $M^\tau_\alpha$ is a meet- and level-preserving tree-embedding from $(\dot S_\tau\cap V_\alpha)^G$ to $(\dot T\cap V_\alpha)^G$. In the quotient, find $p_0,q_0\in\po_\tau/G$ such that
    \begin{itemize}
        \item $p_0\leq p\rest\tau$ and $q_0\leq q\rest\tau$,
        \item $p_0$ is super-nice with respect to every $M^\tau_\beta$ where $\beta\in E^p_\tau$ and $q_0$ is super-nice with respect to every $M^\tau_\beta$ where $\beta\in E^q_\tau$,
        \item $p_0$ decides, for every pair $(s,t)\in f^p_\tau$ of exit nodes from $V_\alpha$, cofinal implicit preimages for the node projection $\dot\pi^{p_0}(t)$, i.e. there are nodes $(\bar{t}^p_n:n<\omega)$ cofinal in $\dot \pi^{p_0}(t)$ and $p_0$ decides, for each $n$, the node $\bar{s}^p_n$ below $s$ at the height of $\bar{t}^p_n$.
        \item same for $q_0$: $q_0$ decides, for every pair $(s,t)\in f^p_\tau$ of exit nodes from $V_\alpha$, cofinal implicit preimages for the node projection $\dot\pi^{q_0}(t)$, i.e. there are nodes $(\bar{t}_n^q:n<\omega)$ cofinal in $\dot \pi^{q_0}(t)$ and $q_0$ decides, for each $n$, the node $\bar{s}^q_n$ below $s$ at the height of $\bar{t}^q_n$.
        \item the conditions $p_0$ and $q_0$ nicely split and are super-nice with respect to $M^\tau_\alpha$.
    \end{itemize}
For every $s\in\dom(f^p_\tau)$ and $s'\in\dom(f^q_\tau)$, look at $\bar{s}$ and $\bar{s}'$ in $V_\alpha$ that witness the splitting, i.e. $p_0\Vdash\bar{s}<s$ and $q_0\Vdash\bar{s}'<s'$. Then the meet of the nodes $\bar{s}$ and $\bar{s}'$ is decided by $G$, so in particular $p_0$ and $q_0$ decide it the same way. Furthermore, for any $n$ and $m$ we have $\bar{s}^p_n\wedge \bar{s}'^q_m=\bar{s}\wedge \bar{s}'$ and moreover, 
\begin{align*}
    &p_0\Vdash s\wedge \bar{s}'^q_m=\bar{s}\wedge \bar{s}',\\
    &q_0\Vdash \bar{s}^p_n\wedge s'=\bar{s}\wedge \bar{s}'.
\end{align*}
We also have that the height of the node 
\[
f^p_\tau(\bar{s})\wedge f^q_\tau(\bar{s}')
\]
in the tree $(\dot T\cap V_\alpha)^G$ is exactly the height of the node $\bar{s}\wedge \bar{s}'$ in the tree $(\dot S_\tau\cap V_\alpha)^G$. This follows from the fact that the trace are the same,
\[
[{p_0}^\smallfrown(f^p_\tau,N^p_\tau)]_{M^{\tau+1}_{\alpha}}=[{p_0}^\smallfrown(f^p_\tau,N^p_\tau)]_{M^{\tau+1}_{\alpha}}
\]
so in particular the nodes $\bar{s}$ and $\bar{s}'$ belong to the $\tau$th embedding of the common trace, call it $f:=f^p_\tau\cap V_\alpha=f^q_\tau\cap V_\alpha$. Thus, since  $f$ is meet-preserving, we have for any $(s,t)\in f^p_\tau$ and $(s',t')\in f^q_\tau$, that
\[
f(\bar{s}\wedge\bar{s}')=\bar{t}^p_m\wedge \bar{t'}^q_m.
\]
Hence, the condition $p_0$ decides meets in the set 
\begin{align*}
    \dom(f^p_\tau)&\cup\{\bar{s}^p_n:n<\omega,s\in\dom(f^p_\tau)\text{ exit node from }V_\alpha\}\\
    &\cup \{\bar{s}'^q_n:n<\omega,s'\in\dom(f^q_\tau)\text{ exit node from }V_\alpha\}
\end{align*} 
and the condition $q_0$ decides meets in the set
\begin{align*}
    \dom(f^q_\tau)&\cup\{\bar{s}^p_n:n<\omega,s\in\dom(f^p_\tau)\text{ exit node from }V_\alpha\}\\
    &\cup \{\bar{s}'^q_n:n<\omega,s'\in\dom(f^q_\tau)\text{ exit node from }V_\alpha\}.
\end{align*}
and the meets are already in the domain of the function $f$.
Hence we may let
\begin{align*}
    \hat{f}^p:=f^p_\tau&\cup\{(\bar{s}^p_n,\bar{t}^p_n):(s,t)\in f^p_\tau\}\\
        &\cup\{(\bar{s}'^q_m,\bar{t}'^q_m):(s',t')\in f^q_\tau\}\\
    \hat{f}^q:=f^q_\tau&\cup\{(\bar{s}^p_n,\bar{t}^p_n):(s,t)\in f^p_\tau\}\\
        &\cup\{(\bar{s}'^q_m,\bar{t}'^q_m):(s',t')\in f^q_\tau\}.
\end{align*}
It follows from above that both $\hat{f}^p$ and $\hat{f}^q$ are injective functions, $p_0$ forces that the domain of the function $\hat{f}^p$ is closed under meets and is level- and meet-preserving, and $q_0$ forces that the function $\hat{f}^q$ is closed under meets and is level- and meet-preserving. It also follows that
\begin{align*}
    &\hat{p}:={p_0}^\smallfrown(\hat{f}^p,N^p_\tau),\\
    &\hat{q}:={q_0}^\smallfrown(\hat{f}^q,N^q_\tau)
\end{align*}
are conditions in $\po_{\tau+1}$ that are super-nice with respect to $M^{\tau+1}_\alpha$ and have the same trace to it.

\vspace{5pt}

We consider the case when $\tau$ is a limit ordinal of countable cofinality. This is straightforward. Let $p,q\in\po_\tau$ nicely split with respect to $M^\tau_\alpha$ satisfying $[p]_{M^\tau_\alpha}=[q]_{M^\tau_\alpha}$. Fix a cofinal sequence $(\tau_j)_{j<\omega}$ converging to $\tau$. Proceed by recursion on $j<\omega$. Let $p_0:=p\rest\tau_0$ and $q_0:=q\rest\tau_0$. At step $j+1$, look at $p_j$ and $q_j$. Suppose that they nicely split and are super-nice with respect to $M^\tau_\alpha$. Then $p':={p_j}^\smallfrown p\rest[\tau_j,\tau_{j+1})$ and $q':={q_j}^\smallfrown q\rest[\tau_j,\tau_{j+1})$ nicely split and satisfy $[p']_{M^\tau_\alpha}=[q']_{M^\tau_\alpha}$. By induction hypothesis there are $p_{j+1}\leq p'$ and $q_{j+1}\leq q'$ that nicely split and are super-nice with respect to $M^\tau_\alpha$ and satisfy $[p_{j+1}]_{M^\tau_\alpha}=[q_{j+1}]_{M^\tau_\alpha}$. Finally the pointwise unions of the $p_j$ and $q_j$, respectively, are as wanted.

The limit case of uncountable cofinality is a straightforward implication of induction hypothesis.

\vspace{5pt}

\noindent\textbf{Proof of item $\boldsymbol{\ref{item2:lem:split}}$.}

\vspace{4pt}
    
\noindent We prove by induction on $\tau$: 
\textit{If $p$ and $q$ are super-nice with respect to $M^\tau_\alpha$ and have the same trace to it, and $t\in T$ is an exit node from $V_\alpha$ at a limit level, then there are $\hat{p}\leq p$ and $\hat{q}\leq q$ that nicely split with respect to $M^\tau_\alpha$, have the same trace to it, and split $t$, and furthermore, for any $\beta$ it holds that if $p\in M^\tau_\beta$, then $\hat{p}\in M^\tau_\beta$.
}

\vspace{5pt}

Suppose that $p$ and $q$ are super-nice with respect to $M^\tau_\alpha$ and have the same trace to it. Let $t\in T$ be an exit node from $M^\tau_\alpha$ at a limit level.
We find $\hat{p}\leq p$ and $\hat{q}\leq q$ that split $t$ and nicely split with respect to $M^\tau_\alpha$, and satisfy $[\hat{p}]_{M^\tau_\alpha}=[\hat{q}]_{M^\tau_\alpha}$, and moreover such that $\hat{p}\in M^\tau_{\beta}$ for any large enough $\beta$ such that $p\in M^\tau_\beta$.

We may assume without loss of generality that for every $\gamma\in\tau\cap M^\tau_\alpha$ and $s\in\dom(f^p_\gamma)$ (resp. $s\in\dom(f^q_\gamma)$) that is an exit node from $M^\tau_\alpha$ at a successor level, the condition $p\rest\gamma$ (resp. $q\rest\gamma$) decides the immediate predecessor $\bar{s}$ of $s$ and $\bar{s}\in\dom(f^p_\gamma)$ (resp. $\bar{s}\in\dom(f^q_\gamma)$). This assumption can be made by fixing a generic $G\subseteq\po_\tau\cap M^\tau_\alpha$ containing the common trace of $p$ and $q$ and extending $p$ and $q$ in the quotient $\po_\tau/G$ using Node density. Similarly, if $\la(M^{\tau_i}_{\alpha_i},\tau_i):i<\nu\ra$ is the $p$-closure  of $M^\tau_\alpha$, up to extending $p$ in the quotient we may assume that 
\[
\sup\{\alpha_i:i<\nu\}=\sup\{\alpha_i:i<\nu\text{ and }\tau_i\in M^\tau_\alpha\},
\] and moreover, that $t\in V_{\alpha_i}$ for some $i<\nu$. See Lemma \ref{lem:sup}.

Fix thus such $p$ and $q$ and let $\la(M^{\tau_i}_{\alpha_i},\tau_i):i<\nu\ra$ be the $p$-closure of $M^\tau_\alpha$.
Denote $\tau_\nu:=\tau$ and $\alpha_\nu:=\kappa$.
By recursion on $j\leq\nu$, we define conditions $p_j\in \po_{\tau_j}\cap M^{\tau_j}_{\alpha_j}$ and conditions $q_j\in\po_\tau$. The final conditions $p_\nu$ and $q_\nu$ will nicely split with repect to $M^\tau_\alpha$, split $t$, and satisfy $[p_\nu]_{M^\tau_\alpha}=[q_\nu]_{M^\tau_\alpha}$.
Let $p_0:=[p]_{M^\tau_\alpha}$ and $q_0:=q$. 

\vspace{5pt}

\noindent\textbf{At step $\boldsymbol{j}$:} Assume that we have defined $p_i$ and $q_i$ for $i<j$ and they satisfy:
\begin{enumerate}
    \item $(p_i\rest\min\{\tau_i,\tau_j\}:i<j)$ forms a $[p\rest\tau_j]_{M^{\tau_j}_{\alpha_j}}$-M.E.,
    \item $(q_i:i<j)$ is a decreasing sequence of conditions in $\po_\tau$,
    \item $p_i$ and $q_i\rest\tau_i$ nicely split and are super nice with respect to $M^{\tau}_\alpha$ and satisfy $[p_i]_{M^\tau_\alpha}=[q_i\rest\tau_i]_{M^\tau_\alpha}$,
    \item If $\tau_i\in\tau\cap M^\tau_\alpha$ and $t\in V_{\alpha_i}$, then $p_i$ and $q_i\rest\tau_i$ split $t$,
    \item If $\tau_i\in\tau\cap M^\tau_\alpha$, then for every pair $(s,s')\in A_i\times B_i$, where 
    \begin{align*}
        &A_i=\{s\in\dom(f^p_{\tau_i}):s\text{ is an exit node from }V_\alpha\}-(V_{\alpha_i}-\bigcup_{i'<i}V_{\alpha_{i'}}),\\
        &B_i=\{s\in\dom(f^{q_i}_{\tau_i}):s\text{ is an exit node from }V_\alpha\},
    \end{align*}
    $p_i$ and $q_i\rest\tau_i$ split the pair $(s,s')$ with a pair $(\bar{s},\bar{s}')$. Also, the common trace $[p_i]_{M^\tau_\alpha}=[q_i\rest\tau_i]_{M^\tau_\alpha}$ decides the meet $\bar{s}\wedge\bar{s}'$ and $p_i$ forces  
    \[
    s\wedge \bar{s}'=\bar{s}\wedge\bar{s}'.
    \]
    Furthermore, $\bar{s},\bar{s}'\in\dom(f^{q_i}_{\tau_i})$, and if $i'\in[i+1,j)$ is such that $f^p_{\tau_i}(s)\in V_{\alpha_{i'}}$, then $p_{i'}$ forces that
    \[
    f^{q_i}_{\tau_i}(\bar{s})<_{\dot T}f^p_{\tau_i}(s).
    \]
    
    \item The extension $q_i\leq\bigcup_{i'<i}q_{i'}$ is minimal in the sense that we only extended the part of $q_i$ up to $\tau_i$ and the function $f^{q_i}_{\tau_i}$ below $\alpha$ by adding witnesses for splitting. Specifically: 
    \begin{itemize}
        \item ${q_i}\rest [\tau_i+1,\tau)=\bigcup_{i'<i}q_{i'}\rest[\tau_{i}+1,\tau)$,
        \item $N^{q_i}_{\tau_i}=\bigcup_{i'<i}N^{q_{i'}}_{\tau_i}$,
        \item $f^{q_i}_{\tau_i}-V_\alpha=\bigcup_{i'<i}f^{q_{i'}}_{\tau_i}-V_\alpha$,
        \item $q_i\rest\tau_i\Vdash \dom(f^{q_i}_{\tau_i})$ is the least set closed under meets and exit nodes from models in $\bigcup_{i'<i}N^{q_{i'}}_{\tau_{i'}}$ that contains $\quad \bigcup_{i'<i}f^{q_{i'}}_{\tau_{i'}}$ and $\{\bar{s},\bar{s}':(s,s')\in A_i\times B_i\}$.
    \end{itemize}
\end{enumerate}
We define $p_j$ and $q_j$.

\begin{claim}\label{claim:tilde:split}
    There is $\tilde p\leq [p\rest\tau_j]_{M^{\tau_j}_{\alpha_j}}$ in $\po_{\tau_j}\cap M^{\tau_j}_{\alpha_j}$ such that $\tilde p$ is super-nice with respect to $M_{\alpha_i}^{\tau_i}$ and $[\tilde p\rest\min\{\tau_j,\tau_{i}\}]_{M^{\tau_{i}}_{\alpha_{i}}}=[p\rest\min\{\tau_j,\tau_{i}\}]_{M^{\tau_{i}}_{\alpha_{i}}}$ for cofinally many $i<j$, and such that for any $i<j$ and any $(s,s')\in A_i\times B_i$, if $f^p_{\tau_i}(s)\in V_{\alpha_j}$, then
    \[
    \tilde p\Vdash f^{q_i}_{\tau_i}(\bar{s})<_{\dot T}f^p_{\tau_i}(s).
    \]
\end{claim}
\begin{proof}
    By Lemma \ref{lem:freem}.
\end{proof}

Let $\tilde p$ be as in Claim \ref{claim:tilde:split}. 

\begin{claim}\label{claim:70}
    There are conditions $p_j\in\po_{\tau_j}\cap M^{\tau_j}_{\alpha_j}$ and $q_j\in\po_\tau$ such that:
    \begin{enumerate}
        \item $p_j\leq \tilde p$ and $p_j\leq p_i\rest\tau_j$ for every $i<j$,
        \item $q_j$ extends every $q_i$, $i<j$, and agrees with their pointwise union after $\tau_j$, i.e. $q_j\rest[\tau_j,\tau)=\bigcup_{i<j}q_i\rest[\tau_j,\tau)$,
        \item $p_j$ and $q_j\rest\tau_j$ nicely split with respect to $M^\tau_\alpha$ and $[p_j]_{M^\tau_\alpha}=[q_j\rest\tau_j]_{M^\tau_\alpha}$.
    \end{enumerate}
\end{claim}
\begin{proof}[Proof of Claim \ref{claim:70}]
    
    Suppose first that $\tau_j<\lim\sup_{i<j}\tau_i$. Let $q_j$ be the pointwise union of the $q_i$, $i<j$. Let $G\subseteq\po_{\tau_j}\cap M^\tau_\alpha$ be a generic filter that contains the common trace $[p_i\rest\tau_j]_{M^\tau_\alpha}=[q_i\rest\tau_j]_{M^\tau_\alpha}$ for every $i<j$. Then $q_j$ is in the quotient $\po_{\tau_j}/G$, and so is every $p_i\rest\tau_j$. By strong properness for $\tau_j$-sequences in the quotient, Lemma \ref{lem:tau-sq:quo}, there is a condition $p_j\in(\po_{\tau_j}\cap M^{\tau_j}_{\alpha_j})/G$ that extends $\tilde p$ and each $p_i\rest\tau_j$, $i<j$. Up to extending $q_j\rest\tau_j$ and $p_j$ in the quotient $\po_{\tau_j}/G$, we may assume that they are super-nice with respect to $M^\tau_\alpha$ and have the same trace to it, $[p_j]_{M^\tau_\alpha}=[q_j\rest\tau_j]_{M^\tau_\alpha}$. By induction hypothesis for $\tau_j$ we may assume further that they nicely split with respect to $M^\tau_\alpha$. Then $p_j$ and $q_j$ are as wanted.

    Suppose then that $\tau_j\geq \lim\sup_{i<j}\tau_i$. Again, let $q_j$ be the pointwise union of the $q_i$, $i<j$. Let $p_j$ be the condition obtained by taking the pointwise union of the conditions $p_i\rest\tau_j$, $i<j$ and $[p\rest\tau_j]_{M^{\tau_j}_{\alpha_j}}$, and furthermore, whenever $\gamma\in\tau_j\cap M^\tau_\alpha\cap M^{\tau_j}_{\alpha_j}$, then the $\gamma$-th coordinate is
    \[
    f_\gamma:=(f^{q_i}_\gamma\cap V_\alpha)\cup \bigcup_{i<j}\bigcup_{\gamma\in\supp(p_i)}f^{p_i}_\gamma\cup (f^p_\gamma\cap V_{\alpha_j}).
    \]
    It follows by construction, as in the proof of strong properness (at step $q^{\eta}$ of proof of Lemma \ref{Lemma:PropernessForSequences}), that $p_j$ is a condition in $\po_{\tau_j}\cap M^{\tau_j}_{\alpha_j}$, and it follows by construction that $p_j$ and $q_j\rest\tau_j$ nicely split and $[p_j]_{M^\tau_\alpha}=[q_j]_{M^\tau_\alpha}$. 
    
\end{proof}

Let $p_j$ and $q_j$ be as in Claim \ref{claim:70}. Up to extending $p_j$ inside $\po_{\tau_j}\cap M^{\tau_j}_{\alpha_j}$, we may assume that $(p_i\rest\min\{\tau_i,\tau_{j+1}\}:i< j+1)$ forms a $[p\rest\tau_{j+1}]_{M^{\tau_{j+1}}_{\alpha_{j+1}}}$-Multi-Extension. 

\vspace{5pt}

If $j=\nu$, we are done. 

If $\tau_j\notin M^\tau_\alpha$, we keep $p_j$ and $q_j$ as they are. 

If $j<\nu$ and $\tau_j\in M^\tau_\alpha$, we make two more extensions to $p_j$ and $q_j$.

\vspace{5pt}

Since $j<\nu$, we have $\tau_j<\tau$. Thus by induction hypothesis, we may assume that $p_j$ and $q_j$ are super-nice with respect to $M^\tau_\alpha$, in addition to nicely splitting and having the same trace to $M^\tau_\alpha$.

Let $G\subseteq\po_{\tau_j}\cap M^\tau_\alpha$ be a generic filter that contains the common trace of $p_j$ and $[q_j\rest\tau_j]_{M^\tau_\alpha}$. Then $p_j,q_j\rest\tau_j\in\po_{\tau_j}/G$. 

\begin{claim}\label{claim:splitt}
    If $t\in V_{\alpha_j}$, then there are extensions of $p_j$ and $q_j\rest\tau_j$ that split $t$ in $\po_{\tau_j}/G$ such that $p_j\in M^{\tau_j}_{\alpha_j}$.
\end{claim}
\begin{proof}[Proof of Claim \ref{claim:splitt}]
    By assumption $\tau_j<\tau$, so by induction hypothesis the tree $\dot T$ is a wide $\kappa$-Aronszajn tree in $V^{\po_{\tau_j}}$. As in Lemma \ref{lem:split1step}, we argue that the branch below $t$ cannot be introduced by $\po_{\tau_j}\cap M^{\tau}_\alpha$: it follows from the case assumption that $\po_{\tau_j}\in M^\tau_\alpha$, so in particular $M^\tau_\alpha$ reflects the fact that $\Vdash_{\po_{\tau_j}}$ $\dot T$ is wide $\kappa$-Aronszajn tree.
    Hence the branch below $t$ cannot be introduced by $\po_{\tau_j}\cap M^\tau_\alpha$. There are two extensions $p^L,p^R\leq p_j$ in $\po_{\tau_j}/G$ that split $t$ at some level $\bar{\alpha}<\alpha$ with some nodes $t^L\neq t^R$. Extend $q'\leq q_j\rest\tau_j$ in $\po_{\tau_j}/G$ to decide the predecessor of $t$ at level $\bar{\alpha}$, call it $\bar{t}$. If $\bar{t}\neq t^L$, then $p^L$ and ${q'}^\smallfrown q_j\rest[\tau_j,\tau)$ are as wanted and otherwise $\bar{t}\neq t^R$ and $p^R$ and ${q'}^\smallfrown q_j\rest[\tau_j,\tau)$ are as wanted.
\end{proof}

By Claim \ref{claim:splitt}, we assume that $p_j$ and $q_j\rest\tau_j$ split $t$ in case $t\in V_{\alpha_j}$. We make a last extension inside the quotient $\po_{\tau_j}/G$. Let
\begin{align*}
    &A_j=\{s\in\dom(f^p_{\tau_j}):s\text{ is an exit node from }V_\alpha\}-(V_{\alpha_j}-\bigcup_{i<j}V_{\alpha_{i}}),\\
    &B_j=\{s\in\dom(f^{q_j}_{\tau_j}):s\text{ is an exit node from }V_\alpha\}.
\end{align*}
Up to extending $p_j$ and $q_j\rest\tau_j$ in $(\po_{\tau_j}\cap M^{\tau_j}_{\alpha_j})/G$ and $\po_{\tau_j}/G$, respectively, we may assume that they split every pair $(s,s')\in A_j\times B_j$ with some pair $(\bar{s},\bar{s}')$ of distinct nodes of same height. For every $i<j$ such that $\tau_j\in \supp(p_i)$, consider the function $f^{p_i}_{\tau_j}$. We have
\[
f^{p_i}_{\tau_j}\cap V_\alpha=f^{q_i}_{\tau_j}\cap V_\alpha.
\]
Let
\[
f:=\bigcup_{i<j,\tau_j\in\supp(p_i)}f^{p_i}_{\tau_j}\cap V_\alpha.
\]
It follows that $f=f^{q_j}_{\tau_j}\cap V_\alpha$, and further, that $f$ is a level- and meet-preserving tree-embedding
\[
f:(\dot S_{\tau_j}\cap V_\alpha)^G\to(\dot T\cap V_\alpha)^G.
\]
Note that $\tau_j+1\leq\tau$ and $\tau_{j+1}\in M^\tau_\alpha$. By Node density Lemma applied in the poset $(\po_{\tau_j+1}\cap M^\tau_\alpha)/G$, there is a function $f'\supseteq f$ whose domain is the closure under meets of the set
\[
\dom(f)\cup\{\bar{s},\bar{s}':(s,s')\in A_j\times B_j\},
\]
and which is level- and meet-preserving tree-embedding from $(\dot S_{\tau_j}\cap V_\alpha)^G$ to $(\dot T\cap V_\alpha)^G$.

Note also that if $(s,s')\in A_j\times B_j$, then , the meet $\bar{s}\wedge \bar{s}'$ is decided by the generic $G$, since $\bar{s},\bar{s}'\in(\dot S_{\tau_j}\cap V_\alpha)^G$, and furthermore 
\begin{align*}
    p_j&\Vdash s\wedge \bar{s}'=\bar{s}\wedge \bar{s}',\\
    q_j\rest\tau_j&\Vdash \bar{s}\wedge s'=\bar{s}\wedge \bar{s}'.
\end{align*}
This implies in particular, that if $s\neq s'$, then $p_j$ forces that $\bar{s}'$ is not below $s$. And further, that $p_j$ decides the meets in the set
\[
\dom(f')\cup(\dom(f^p_{\tau_j})\cap V_{\alpha_j})\cup\bigcup_{i<j,\tau_j\in\supp(p_i)}f^{p_i}_{\tau_j}.
\]
Recall that if $s\in A_j$, then $f^p_{\tau_j}(s)\notin M^{\tau_j}_{\alpha_j}$, so in particular we are free to extend the collapse coordinate of $p$ to force $f'(\bar{s}\wedge \bar{s}')<f^p_{\tau_j}(s)$ and $f'(\bar{s}')\not<f^p_{\tau_j}(s)$ at the least step $j'>j$ when it happens that $f^p_{\tau_j}(s)\in V_{\alpha_{j'}}$.

Up to extending $p_j$ and $q_j\rest\tau_j$ in $(\po_{\tau_j}/G)\cap M^{\tau_j}_{\alpha_j}$ and $\po_{\tau_j}/G$, respectively, we may assume that the common trace $[p_j]_{M^\tau_\alpha}=[q_j\rest\tau_j]_{M^\tau_\alpha}$ decides meets in $\dom(f')$ and forces that $f'$ is level- and meet-preserving tree-embedding from $\dot S_{\tau_j}$ to $\dot T$. Then, up to extending $q_j$ at the coordinate $\tau_j$, we may assume that
\[
f^{q_{j}}_{\tau_j}=f'\cup\bigcup_{i<j}f^{q_i}_{\tau_j}.
\]
Then $p_j$ and $q_j$ are as wanted. This ends item (\ref{item2:lem:split}).

\vspace{5pt}

\noindent\textbf{Proof of item $\boldsymbol{\ref{item3:lem:split}}$.}

\vspace{4pt}
    
\noindent Assume that $\po_\tau$ introduces a cofinal branch to $\dot T$. Fix a $\po_\tau$-name $\name{b}$ and a condition $p$ such that $p\Vdash\name{b}$ is a cofinal branch in $\dot T$. Find $\alpha$ such that $p,\name{b}\in M^\tau_\alpha$ and $q\leq p$ which is super-nice with respect to $M^\tau_\alpha$. Up to extending $q$, assume that it forces $\name{b}(\alpha)=t$ for some node $t\in T$. By item \ref{item2:lem:split} applied to the pair $q_1:=q_2:=q$, there are $\hat{q}_1\leq q_1$ and $\hat{q}_2\leq q_2$ that are super-nice with respect to $M^\tau_\alpha$, have the same trace $[\hat{q}_1]_{M^\tau_\alpha}=[\hat{q}_2]_{M^\tau_\alpha}$, and split $t$ with some distinct nodes $t^L$ and $t^R$ at some level $\bar{\alpha}<\alpha$. The the common trace $[\hat{q}_1]_{M^\tau_\alpha}=[\hat{q}_2]_{M^\tau_\alpha}$ is a common residue for $\hat{q}_1$ and $\hat{q}_2$. Let $w\in\po_\tau\cap M^\tau_\alpha$ extend the common trace to decide the $\bar{\alpha}$-th node on the branch, say it forces $\name{b}(\bar{\alpha})=\bar{t}$. If $\bar{t}\neq t^L$, then $w$ cannot be compatible with $\hat{q}_1$ and if $\bar{t}\neq t^R$, then $w$ cannot be compatible with $\hat{q}_2$. This is in contradiction with the fact that the common trace was a common residue for $\hat{q}_1$ and $\hat{q}_2$.

\end{proof}

\begin{theorem}\label{nocofinal}
Suppose that the book-keeping function $\Psi$ is such that for each $\tau<\kappa^+$, the name $S_\tau$ is a $\po_\tau$-name for q (wide) $\kappa$-Aronszajn tree. Then
no cofinal branches are added to the wide tree $T$ by $\po_{\kappa^+}$.
\end{theorem}

\begin{proof}
By Lemma \ref{lem:split} item (\ref{item3:lem:split}) and $\kappa^+$-cc of $\po_{\kappa^+}$.
\end{proof}

\begin{theorem}\label{Thm:MainAleph_2}
Suppose that the book-keeping function $\Psi$ which picks the trees $S_\tau$, picks only names for wide $\kappa$-Aronszajn trees and covers all $\po_\tau$-names for wide $\kappa$-Aronszajn trees, for all $\tau < \kappa^+$.
Let $G \subseteq \po_{\kappa^+}$ be a generic filter. In $V[G]$, $\kappa = \aleph_2$, $T$ is a wide $\aleph_2$-Aronszajn tree on $\kappa$ which embeds all wide $\aleph_2$-Aronszajn trees.
\end{theorem}
\begin{proof}
It follows from the proof \ref{Theorem-Main-Sec1} that $T$ embeds all wide $\aleph_2$-Aronszajn trees on $\kappa$. By the last theorem $T$ does not get a cofinal branch, and hence remains $\aleph_2$-Aronszajn. 
\end{proof}

\section{Open Problems}\label{Section:OpenProblems}

The following problems  are left open by this work: 
\begin{enumerate}
\item Is the weakly compact cardinal needed for Theorem \ref{Theorem:main}? \\
We conjecture that the answer is yes.

\item Is it consistent to have a universal (narrow) Aronszajn tree?\\ 
For $\aleph_2$-Aronszajn trees, we expect this to be possible by incorporating ideas from Mitchell's construction of a model without Aronszajn trees on $\omega_2$ and work towards verifying the details. The case of  $\aleph_1$-Aronszajn trees remains unclear. 

%\item Can we replace $\aleph_2$ by $\aleph_1$ in our result? \\
%It seems plausible that an adaptation of the construction to $\omega_1$, in which finite supports and approximations are used can lead to a desired result. However, several new arguments are needed to avoid the prolific use of the $\sigma$-closure property of our poset.

\item Can the universality result for wide Aronszajn trees hold at successors of singular cardinals? \\
A positive answer would likely require  developing new methods for iteration at successors of singular cardinals. 

\item Can one consistently have  a maximal wide $\aleph_2$-Aronszajn and, at the same time,  a maximal wide $\aleph_3$-Aronszajn tree?
\end{enumerate}
%\bibliographystyle{plain}
%\bibliography{maximal}

\end{document}